\documentclass[10pt,a4paper]{article}

\usepackage[T1]{fontenc}
\usepackage{lmodern}
\usepackage{microtype}
\usepackage[a4paper,margin=25mm]{geometry}
\usepackage{amsmath,amssymb,amsthm,mathtools,mathrsfs}
\usepackage{enumitem}
\usepackage{tikz}

\usepackage[dvipsnames]{xcolor}

\usepackage{hyperref}

\hypersetup{
  hidelinks,
  pdftitle={Future Completeness, C0-Inextendibility and Cauchy Horizons in Homogeneous Einstein Spacetimes},
  pdfauthor={Bobby Eka Gunara},
  pdfsubject={Future causal completeness, continuous inextendibility, and Cauchy horizons in homogeneous Einstein spacetimes},
  pdfkeywords={future geodesic completeness, C0-inextendibility, homogeneous Einstein spacetime, nonlinear sigma model, two-step nilpotent group, compact nilmanifold, Bianchi II, Cauchy horizon}
}

\setlist{topsep=4pt,itemsep=2pt,parsep=0pt,leftmargin=*}
\allowdisplaybreaks
\numberwithin{equation}{section}

\newtheorem{theorem}{Theorem}[section]
\newtheorem{proposition}[theorem]{Proposition}
\newtheorem{lemma}[theorem]{Lemma}
\newtheorem{corollary}[theorem]{Corollary}
\theoremstyle{definition}
\newtheorem{definition}[theorem]{Definition}
\theoremstyle{remark}
\newtheorem{remark}[theorem]{Remark}
\newtheorem{example}[theorem]{Example}

\newcommand{\R}{\mathbb R}
\newcommand{\dd}{\mathrm d}
\newcommand{\diam}{\operatorname{diam}}
\newcommand{\Vol}{\operatorname{Vol}}
\newcommand{\Lip}{\operatorname{Lip}}
\newcommand{\cM}{\mathcal M}
\newcommand{\eps}{\varepsilon}
\newcommand{\past}{\partial^-}
\newcommand{\Ric}{\operatorname{Ric}}
\newcommand{\Ad}{\operatorname{Ad}}

\title{Future Completeness, $C^0$-Inextendibility and Cauchy Horizons\\
in Homogeneous Einstein Spacetimes}
\author{Bobby Eka Gunara\\[2mm]
\small Theoretical Physics Laboratory, Faculty of Mathematics and Natural Sciences,\\
\small Institut Teknologi Bandung, Jl. Ganesa 10, Bandung 40132, Indonesia\\
\small \texttt{bobby@itb.ac.id}}
\date{}

\begin{document}
\maketitle

\begin{abstract}
We prove the future-completeness conjecture of G\"odeke and Rendall for
expanding spatially homogeneous vacuum spacetimes in four spatial
dimensions.  More generally, an expansion-energy estimate gives future
timelike and null geodesic completeness for homogeneous Einstein equations
coupled to maps into complete Riemannian targets, through nine spatial
dimensions for vanishing potential.  In arbitrary dimension we obtain an
affine-length criterion which includes positive potential floors and
potentials that decay sufficiently slowly along the scalar trajectory.
For future-global homogeneous solutions, the critical nine-dimensional
bound also holds for a general stress tensor in an explicit normal
energy-condition window, including a broad range of perfect fluids.
At the opposite time end, we establish a quotient-stable
$C^0$-inextendibility criterion for two-step nilpotent cosmologies.  Its
proof combines boundary localization with an optimal, intrinsically
defined allocation of the central correction in timelike homotopies, and
continues to apply when the diameter of compact spatial slices collapses.
Exact Heisenberg solutions realize both conclusions.  On every fixed
compact Bianchi~II quotient, the expanding invariant vacuum data split
into an open dense set with globally $C^0$-inextendible developments and
a codimension-two locally rotationally symmetric locus with analytic
compact Cauchy horizons.  We determine the chronology-violating region of
the exceptional extensions and quantify the exponentially small
proper-time scale on which transverse anisotropy replaces the horizon by
a singular end.
\end{abstract}



\tableofcontents

\section{Introduction}

The two ends of a maximal globally hyperbolic spacetime pose different
regularity-sensitive questions.  At an expanding end one asks for
infinite affine length of every causal geodesic; at a finite-time end one
asks whether the metric itself extends continuously and nondegenerately.
The latter question cannot be decided by curvature alone, because
curvature is not defined for a merely continuous extension
\cite{GLS2018,Sbierski2018,SbierskiProof2018}.
G\"odeke and Rendall asked whether every expanding spatially homogeneous
vacuum spacetime with a connected, simply connected four-dimensional
symmetry group other than $\mathrm{SU}(2)\times\mathbb R$ is future
geodesically complete.  We prove their conjecture.  The proof is a
dimension-sensitive expansion estimate which requires neither
diagonalizability nor convergence or bounded oscillation of the
generalized Kasner exponents.  It also treats homogeneous spaces with
isotropy and nonlinear sigma models with complete Riemannian targets, as
arise in multifield cosmology \cite{Lazaroiu2022}.  Besides the massless
result through nine spatial dimensions, it gives an all-dimensional
affine-length criterion for nonnegative potentials, including positive
floors and a class of potentials decaying along the field trajectory.
For future-global homogeneous matter models, the same ninth-dimensional
bound follows from a normal energy-condition window which, for a perfect
fluid, ranges from $p=-(n-2)\rho/n$ to the stiff-fluid bound $p=\rho$.

At the opposite, finite-proper-time end, a continuous Lorentzian
extension preserves causal cones and Lorentzian length but need not
preserve the connection or curvature.  Timelike completeness and
spacelike diameter provide geometric obstructions, but spatial quotients
create a further difficulty: an entire compact slice may collapse in
diameter while every fixed open set in its universal cover expands.  We
show that this local expansion still excludes every continuous extension.

The resulting picture has three parts: complete future evolution,
$C^0$-singular past ends, and exceptional regular Cauchy horizons.  The
past obstruction applies to two-step nilpotent groups and every discrete
left quotient.  Exact Einstein solutions on every odd-dimensional
Heisenberg group realize it.  On a fixed compact Bianchi~II quotient, all
expanding invariant vacuum data are classified: an open dense set gives
globally $C^0$-inextendible developments, whereas a codimension-two
submanifold gives analytic compact Cauchy horizons.  This is a complete
dichotomy in the finite-dimensional space of invariant tensor data on the
fixed quotient; no equation, symmetry, or differentiability beyond
continuity is imposed on a candidate extension.

Our geometric setting is
\begin{equation}
 \cM=(0,t_+)\times X,\qquad
 g=-N(t)^2\dd t^2+h_t,\qquad X=\Gamma\backslash Y,
 \label{eq:intro-spacetime}
\end{equation}
where $Y=L/H$ has a complete invariant Riemannian metric $k$,
$\Gamma<L$ acts freely and properly discontinuously, and $h_t$ is
$L$-invariant.  Small quotient distances do not control winding between
covering sheets.  Instead we bound the full $k$-length of causal spatial
projections.  Vanishing length near the past end, local homogeneous
isometries and timelike filling confine a chronological past to one
regular boundary chart.  Its spacelike sections have uniformly Lipschitz
parametrizations; lifting these sections gives diameter and volume
obstructions on $Y$.  No lift of the extension itself is assumed.  An
irrational compact Kasner example has diameter tending to zero while
these lifted sections force continuous inextendibility.

For a simply connected two-step nilpotent group, put
$Z=Z(\mathfrak n)$ and $V=\mathfrak n/Z$.  The invariant metric has
the Schur form
\begin{equation}
 h_t=H_t(\dd v,\dd v)+
 A_t(\Theta+L_t\dd v,\Theta+L_t\dd v),\qquad
 \Theta=\dd z-\tfrac12\mathsf b(v,\dd v).
 \label{eq:intro-schur}
\end{equation}
Horizontal interpolation creates a central endpoint error equal to a
bracket area.  A distributed central correction cancels it while
preserving strict timelikeness.  Duality of integrated central ellipsoids
determines its least uniform bound.  This construction permits arbitrary
anisotropic $A_t$ and Schur coupling $L_t$.
Proposition~\ref{prop:einstein-optimal-separation} realizes the same
strict separation in vacuum and canonical Einstein-massless-scalar
developments on $H_3\times H_3$; flat products give examples in every
spatial dimension at least six.  Separate central corrections already
suffice for this family, so its significance is the failure of
area-direction allocation within a solution of the field equations.
The final obstruction also requires causal-length control and growth
in the full abelianization or in volume.  Bracket area and nilpotent
filling are classical \cite{Milnor1976,Eberlein1994,Young2013}; related
control questions occur in sub-Lorentzian Heisenberg geometry
\cite{Grochowski2006,SachkovSachkova2023}.
The causal geometry of fixed left-invariant Lorentz metrics on two-step
groups, including closed timelike geodesics on compact nilmanifolds and
global-hyperbolicity criteria, was studied by Guediri
\cite{GuediriLorentz2003,GuediriGlobal2003,GuediriCriterion2008}.  Our
chronology result instead concerns a time-dependent analytic Einstein
extension: it locates the entire chronology-violating region and gives
the sharp horizontal-length threshold in each constant-time slice.  The
quotient inextendibility theorem is Riemannian on the slices and applies
to arbitrary discrete left quotients, independently of this Lorentzian
classification.

The closest causal model is Miethke's Kasner argument
\cite{Miethke2024}.  The FLRW results of Sbierski and Ling
\cite{SbierskiFLRW,LingFLRW} treat a complementary singularity regime;
Gunara \cite{Gunara2026} and Mosani \cite{Mosani2026} use related localization for warped-product
black holes.  Chru\'sciel and Klinger \cite{ChruscielKlinger2018} obtain
restrictions for expanding homogeneous ends, while Lott
\cite[Proposition~13 and Example~6]{Lott2020} relates integrable inverse
scales to shrinking causal pasts.  Compact homogeneous models also carry
lattice and moduli data
\cite{KoikeTanimotoHosoya1994,TanimotoKoikeHosoya1997}.
The local translation argument is proved in
Lemma~\ref{lem:transporters} and
Proposition~\ref{thm:quotient-localization}.  The Heisenberg central-area
correction is derived in Proposition~\ref{thm:schur-filling}; its diagonal
form and a computable integral criterion are given in
Proposition~\ref{prop:diagonal-heisenberg}.  The anisotropic Taub vacuum
application is stated in Corollary~\ref{cor:taub-vacuum-specialization}.
These arguments lead to localization by local homogeneous isometries and
lifted slice embeddings, optimal anisotropic central allocation for
general two-step groups with Schur coupling, and an obstruction on
arbitrary discrete quotients.  The later sections establish the
higher-dimensional Einstein realizations and the vacuum-data
classification on each fixed compact quotient.

The Einstein constructions build on quiescent asymptotics
\cite{AnderssonRendall2001}, potential-dependent Kasner analysis
\cite{Ritchie2022}, and homogeneous scalar classification
\cite{Ringstrom2025}.  The future estimate refines expansion-eigenvalue
criteria of Rendall and G\"odeke-Rendall
\cite{Rendall1995Matter,GoedekeRendall2010}: in the vacuum case it removes
the imposed oscillation hypothesis of their Theorem~3 through eight
spatial dimensions and includes the ninth-dimensional endpoint.  It also
gives an affirmative answer to their four-spatial-dimensional conjecture;
see Corollary~\ref{cor:four-spatial-vacuum}.  This comparison does not
assert absence of dynamical oscillation.  Their pointwise
generalized-Kasner-exponent criterion and their reflection-symmetric
product theorem remain complementary.  Their concluding discussion
identifies homogeneous spaces with isotropy as a further direction.
The bound of nine is sufficient, and asserts no incompleteness above
that dimension.  For a positive potential floor, our completeness
conclusion requires only smoothness and the lower bound, whereas the
stronger asymptotic conclusions in three dimensions established by
Rendall \cite[Theorems~1-2]{Rendall2004Scalar} use further potential
hypotheses.

Exact Heisenberg metrics have substantial precedents.  Gibbons-L\"u-Pope-Stelle
\cite[Section~4.1.1]{GibbonsLuPopeStelle2002} construct a Riemannian
Ricci-flat metric with five-dimensional Heisenberg orbits;
Fr\'e-Rulik-Trigiante \cite[Section~3.1.1]{FreRulikTrigiante2004} give
a symmetric Bianchi II massless-scalar cosmology.  Curvature-wall
reflections belong to the cosmological billiard description
\cite{DamourHenneauxNicolai2003}.  Cort\'es and Saha
\cite[Remark~4.6]{CortesSaha2022} describe the symmetric Ricci-flat
Lorentzian family by analytic continuation.  Our conclusions concern
the exact Lorentzian transition law, affine lifetime, and continuous
extendibility on discrete quotients.  Compact homogeneous vacuum horizons
and their additional symmetry are also classical
\cite{ChruscielRendall1995}; that work already permits extensions without
field equations.  Here the extension regularity is lowered to an
arbitrary continuous nondegenerate Lorentzian metric, and the exceptional
set is identified intrinsically in regular invariant data on each fixed
compact quotient.  Siklos had already shown, in a broader classical
homogeneous analysis, that horizon solutions depend on two fewer
parameters than generic solutions \cite{Siklos1978}.  Accordingly, the
codimension-two count by itself is not claimed as new.  The contribution
here is its realization as an if-and-only-if $C^0$ extension criterion in
the nine-dimensional Bianchi~II constraint-data manifold, together with
the global future conclusion and the quantitative
anisotropy-to-curvature crossover.

The three main results are Theorems~\ref{thm:homogeneous-future},
\ref{thm:two-step-obstruction} and~\ref{thm:vacuum-bianchi-alternative}.
They give, respectively, the complete expanding end, the quotient-stable
past obstruction, and the sharp vacuum Bianchi~II alternative obtained by
combining the two.  The first is proved immediately below the introduction.
Boundary localization and timelike filling then lead to the past
obstruction, followed by the Einstein realizations and the invariant-data
classification.
Detailed filling constants, arithmetic refinements, metric comparisons
and additional horizon coordinates appear in the appendices.
The results do not include inhomogeneous genericity or the
all-positive-exponent regime of stable quiescent big-bang formation
\cite{FournodavlosRodnianskiSpeck2023,OudeGroenigerPetersenRingstrom}.
The critical-bracket example imposes no matter model.  The product
Heisenberg construction supplies Einstein realizations with distinct
central scales, while Corollary~\ref{cor:optimal-product} gives an exact
tensorization law and permits mixed, non-product discrete subgroups.

\subsection{Principal results and Einstein realizations}

For a spatially homogeneous map $\Phi(t)$ into a Riemannian target
$(\mathcal T,\mathsf G)$, the Einstein-scalar equations take the form
\begin{equation}
 \Ric(g)=\Phi^*\mathsf G+\frac{2\mathcal V(\Phi)}{n-1}g,
 \qquad D_t\dot\Phi+\theta\dot\Phi+
             \operatorname{grad}_{\mathsf G}\mathcal V=0,
 \qquad \theta=\tfrac12\operatorname{tr}_{h_t}\dot h_t,
 \label{eq:main-sigma-normalization}
\end{equation}
Here $D_t$ denotes covariant differentiation in the target.

\begin{theorem}[Future completeness and the G\"odeke-Rendall conjecture]
\label{thm:homogeneous-future}
Let $X=L/H$ be a connected homogeneous space of dimension $n\ge2$,
where $L$ is a connected Lie group and $H$ is closed.  Let
$(\mathcal T,\mathsf G)$ be a finite-dimensional connected complete
Riemannian manifold and let $\mathcal V\in C^\infty(\mathcal T)$.
Suppose that $g=-\dd t^2+h_t$, with $h_t$
$L$-invariant, and $\Phi=\Phi(t)$ form a maximal homogeneous
solution of~\eqref{eq:main-sigma-normalization}.
Assume that $\theta(t_0)=\theta_0>0$ and that
$R(h_t)\le0$ and $\mathcal V(\Phi(t))\ge0$ throughout the future
interval.  Then the future
proper-time endpoint is infinite.  The spacetime is future timelike
and null geodesically complete on $X$, and on every quotient
$\Gamma\backslash X$ for which $\Gamma<L$ acts freely and properly
discontinuously, in each of the following cases:
\begin{enumerate}[label=\textup{(\roman*)}]
\item $\mathcal V\equiv0$ and $2\le n\le9$;
\item $2\le n\le5$;
\item $n\ge3$ and
\[
 \int_{t_0}^{\infty}
       \mathcal V(\Phi(t))^{\beta_n/2}\,\dd t=\infty,
 \qquad \beta_n=\frac{(n-1)^2}{4n}.
\]
\end{enumerate}
For any future causal geodesic with affine parameter $\lambda$,
\[
 E(t):=\frac{\dd t}{\dd\lambda}
 \le E(t_0)\left(\frac{\theta_0}{\theta(t)}\right)^{\beta_n},
 \qquad
 \beta_2=0,\quad \beta_n=\frac{(n-1)^2}{4n}\quad(n\ge3),
\]
and
\[
 \lambda(t)-\lambda(t_0)\ge
 \frac{1}{E(t_0)\theta_0^{\beta_n}}
 \int_{t_0}^{t}\theta(s)^{\beta_n}\,\dd s.
\]
Thus divergence of the last integral is an all-dimensional completeness
criterion.  When $\mathcal V\equiv0$, one may replace $\beta_n$ by
$c_2=0$, $c_3=1/3$ and $c_n=(\sqrt n-1)/2$ for $n\ge4$.
If $\mathcal V(\Phi(t))\ge V_*>0$, then
$\theta(t)\ge\sqrt{2nV_*/(n-1)}$, so each geodesic has uniformly
bounded energy in every dimension.  In particular, for a connected,
simply connected four-dimensional Lie group $G\not\simeq
\mathrm{SU}(2)\times\mathbb R$, every homogeneous vacuum development on
$G$ which is expanding at one slice is future timelike and null
geodesically complete, as conjectured by G\"odeke and Rendall.
The real-scalar equations
\eqref{eq:main-einstein-normalization} correspond to
$(\mathcal T,\mathsf G)=(\R,\dd\phi^2)$; a constant map with
$\mathcal V=0$ gives vacuum.
\end{theorem}

\begin{proof}
Proposition~\ref{prop:nonnegative-future-completeness} proves future
existence, the energy and affine-length estimates, and \textup{(ii)}.
Corollary~\ref{cor:decaying-potential-completeness} proves \textup{(iii)}.
Proposition~\ref{prop:massless-future-completeness} proves the improved
exponents and \textup{(i)}, including the ninth-dimensional endpoint.
Proposition~\ref{prop:positive-floor-future-completeness} proves
the positive-floor assertion.  Corollary~\ref{cor:four-spatial-vacuum}
proves the final vacuum statement.  The proofs use invariant tensors on
$L/H$ and do not require a simply transitive action or diagonal metrics.
\end{proof}

We use the normalization
\begin{equation}
 \Ric(g)=\dd\phi\otimes\dd\phi+
 \frac{2}{n-1}\mathcal V(\phi)g,
 \qquad \Box_g\phi=\mathcal V'(\phi)
 \label{eq:main-einstein-normalization}
\end{equation}
in $n+1$ spacetime dimensions.

\begin{corollary}[Geodesic lift to nonlinear sigma models]
\label{cor:target-geodesic-lift}
Let $(g,\phi)$ solve the Einstein-massless-scalar equations
\eqref{eq:main-einstein-normalization} with $\mathcal V=0$, and let
$\gamma:\mathbb R\to(\mathcal T,\mathsf G)$ be a unit-speed geodesic in
a complete Riemannian target.  Then
\[
 \Phi=\gamma\circ\phi
\]
together with the same metric $g$ solves
\[
 \Ric(g)=\Phi^*\mathsf G,\qquad D^\mu\partial_\mu\Phi=0.
\]
Consequently every massless-scalar Heisenberg development constructed
below, including its discrete quotients, gives a nonlinear sigma-model
development with exactly the same completeness, inextendibility and
horizon properties.  The statement is valid for noncompact and closed
target geodesics alike.
\end{corollary}

\begin{proof}
Completeness of the target makes $\gamma$ defined on all of $\mathbb R$.
Since $|\dot\gamma|_{\mathsf G}=1$,
$\Phi^*\mathsf G=\dd\phi\otimes\dd\phi$.  The chain rule and
$D_{\dot\gamma}\dot\gamma=0$ give
\[
 D^\mu\partial_\mu(\gamma\circ\phi)
 =\dot\gamma(\phi)\,\Box_g\phi=0.
\]
Thus the stress tensor and the spacetime metric are unchanged, so all
causal and extension conclusions transfer verbatim.
\end{proof}

Write $H_{2m+1}$ for the simply
connected Heisenberg group with invariant coframe satisfying
\[
 \dd\omega^0=-\sum_{r=1}^m\omega^{2r-1}\wedge\omega^{2r},
 \qquad \dd\omega^i=0\quad(i>0).
\]
The later Einstein propositions construct exact developments on these
groups and their discrete quotients, determine the Kasner transition and
the admissible exponential potentials, and realize the strict gain from
central allocation within vacuum and massless-scalar solutions.  The
fixed-quotient vacuum classification culminates in
Theorem~\ref{thm:vacuum-bianchi-alternative}.  By
Corollary~\ref{cor:target-geodesic-lift}, every massless-scalar
construction also takes values along any complete unit-speed target
geodesic.  The exceptional extensions are analyzed further through their
chronology-violating region and the sharp horizontal-length threshold for
closed timelike curves.

\section{Future completeness of homogeneous Einstein-scalar systems}
\label{sec:homogeneous-future}

We prove Theorem~\ref{thm:homogeneous-future} on homogeneous spaces
with arbitrary isotropy and for homogeneous maps into a complete
Riemannian target.  The real-scalar and vacuum equations used later
are special cases.  The Hamiltonian constraint controls finite-time
continuation; a pointwise bound for the least expansion eigenvalue
then controls every causal geodesic.  This extends the setting of
the homogeneous completeness criteria in
\cite{Rendall1995Matter,GoedekeRendall2010,Lee2005}.

Throughout this section $X=L/H$, $(\mathcal T,\mathsf G)$ and
$\mathcal V$ satisfy the geometric and smoothness hypotheses of
Theorem~\ref{thm:homogeneous-future}.  A homogeneous solution means
$g=-\dd t^2+h_t$, with $h_t$ $L$-invariant and $\Phi=\Phi(t)$,
solving all equations~\eqref{eq:main-sigma-normalization}.  Maximality
refers to continuation as such a solution in proper time, without
imposing an additional sign condition on a possible continuation.
The target metric is positive definite; the argument does not apply
to scalar fields with negative kinetic energy.

\begin{proposition}[Nonnegative potentials and affine length]
\label{prop:nonnegative-future-completeness}
Let $(g,\Phi)$ be a maximal homogeneous solution as above, with
$\mathcal V(\Phi(t))\ge0$, $\theta(t_0)=\theta_0>0$ and $R(h_t)\le0$
throughout the future interval.  Its future proper-time endpoint is
infinite.  Put
\[
 \beta_2=0,\qquad \beta_n=\frac{(n-1)^2}{4n}\quad(n\ge3).
\]
Every future causal geodesic,
parametrized affinely by $\lambda$ and starting at time $t_0$,
satisfies
\begin{equation}
 E(t):=\frac{\dd t}{\dd\lambda}(t)
 \le E(t_0)\left(\frac{\theta_0}{\theta(t)}\right)^{\beta_n}
 \le E(t_0)\bigl(1+\theta_0(t-t_0)\bigr)^{\beta_n}.
 \label{eq:future-geodesic-energy}
\end{equation}
Moreover,
\begin{equation}
 \lambda(t)-\lambda(t_0)\ge
 \frac{1}{E(t_0)\theta_0^{\beta_n}}
 \int_{t_0}^{t}\theta(s)^{\beta_n}\,\dd s.
 \label{eq:future-affine-integral}
\end{equation}
Consequently the solution is future timelike and null geodesically
complete in every dimension for which
\begin{equation}
 \int_{t_0}^{\infty}\theta(t)^{\beta_n}\,\dd t=\infty.
 \label{eq:theta-affine-criterion}
\end{equation}
For $2\le n\le5$, the solution is future timelike and null
geodesically complete on $X$ and every quotient allowed in
Theorem~\ref{thm:homogeneous-future}.
\end{proposition}

\begin{proof}
Write $K=\tfrac12\dot h$, $\Sigma=K-\theta h/n$,
$\psi=\dot\Phi$, and $R=R(h)$.  Spatial tensor norms are taken with
respect to $h$, and $|\psi|$ with respect to $\mathsf G$.
The Hamiltonian constraint and the normal evolution equation give
\begin{align}
 \theta^2&=|K|^2+|\psi|^2+2\mathcal V-R,\notag\\
 \dot\theta&=-|K|^2-|\psi|^2+\frac{2}{n-1}\mathcal V
 =-\frac n{n-1}(|\Sigma|^2+|\psi|^2)+\frac R{n-1}.
 \label{eq:higher-future-identities}
\end{align}
Thus $-\theta^2\le\dot\theta\le0$, and
\begin{equation}
 \frac{\theta_0}{1+\theta_0(t-t_0)}\le\theta(t)\le\theta_0.
 \label{eq:future-expansion-bounds}
\end{equation}
In particular, the lower differential inequality prevents the initially
positive expansion from reaching zero at finite time.
The constraint bounds $|K|$ and $|\psi|$ by $\theta_0$.
Consequently
\begin{equation}
 e^{-2\theta_0(t-t_0)}h_{t_0}\le h_t
 \le e^{2\theta_0(t-t_0)}h_{t_0},\qquad
 \operatorname{Length}_{\mathsf G}
       (\Phi|_{[t_0,t]})\le\theta_0(t-t_0).
 \label{eq:homogeneous-sigma-finite-slab}
\end{equation}

We justify continuation without choosing a global invariant frame.
Fix $o=eH$.  Evaluation at $o$ identifies invariant symmetric
two-tensors with the finite-dimensional vector space
$\operatorname{Sym}^2(T_o^*X)^H$; invariant metrics form its open
positive cone.  The Ricci tensor depends smoothly on this cone.
Indeed, local sections of $L\to L/H$ express a metric and its first
two spatial derivatives smoothly in its value at $o$; matrix
inversion is smooth on positive forms.  The evolution equations are
the smooth finite-dimensional system
\begin{equation}
 \dot h=2K,\qquad
 \dot K=-\Ric(h)-\theta K+2Kh^{-1}K+
                 \frac{2\mathcal V(\Phi)}{n-1}h,
 \qquad D_t\psi=-\theta\psi-
                       \operatorname{grad}_{\mathsf G}\mathcal V.
 \label{eq:homogeneous-sigma-ode}
\end{equation}
Here $\dot\Phi=\psi$, and $Kh^{-1}K$ denotes contraction of the
adjacent indices.  All spatial equations preserve invariant tensors.
The Hamiltonian and momentum constraints propagate: the contracted
Bianchi identity, together with the map equation, is a homogeneous
linear first-order system for the normal Einstein defects.  These
defects are invariant scalars and one-forms, so their spatial
covariant derivatives are linear maps between finite-dimensional
invariant tensor spaces, with smooth metric-dependent coefficients.
Their propagation is therefore a linear
ordinary differential system, whose solution with zero initial
defects is zero.  Hence an extension of
\eqref{eq:homogeneous-sigma-ode} extends the full field equations.

If the future endpoint $T$ were finite,
\eqref{eq:homogeneous-sigma-finite-slab} would keep $h,h^{-1},K$
in a compact subset of their coefficient domain.  The target path
lies in the closed ball of radius $\theta_0(T-t_0)$ about
$\Phi(t_0)$.  This ball is compact by completeness and the
Hopf-Rinow theorem.  The vectors $\psi$ lie in the compact disk
bundle $|\psi|\le\theta_0$ over it.  The full state thus remains in
a compact subset of the domain of the smooth system
\eqref{eq:homogeneous-sigma-ode}, which extends past $T$.
This contradicts maximality.  Future proper time is infinite in
every dimension.

If a trace-free symmetric endomorphism has least eigenvalue $-a$,
the remaining eigenvalues have sum $a$ and squared sum at least
$a^2/(n-1)$.  Therefore
$\lambda_{\min}(h^{-1}\Sigma)\ge-\sqrt{(n-1)/n}|\Sigma|$.
Set $d=-\dot\theta\ge0$.  By
\eqref{eq:higher-future-identities} and $R\le0$,
\[
 \lambda_{\min}(h^{-1}K)
 \ge\frac{\theta-(n-1)\sqrt d}{n}
 \ge-\beta_n\frac d\theta.
\]
For $n\ge3$, the last inequality is equivalent to
$(\theta-(n-1)\sqrt d/2)^2\ge0$.  If $n=2$, then
$d\le\theta^2$ makes the preceding lower bound nonnegative, which
proves the sharper value $\beta_2=0$.
For a causal geodesic let $w=\dd x/\dd t$.  Causality gives
$|w|_h\le1$, and its time equation yields
\[
 K(w,w)\ge\min\{0,\lambda_{\min}(h^{-1}K)\},\qquad
 \frac{\dd}{\dd t}\log E=-K(w,w)
 \le \beta_n\frac d\theta=-\beta_n\frac{\dot\theta}{\theta}.
\]
Integration proves~\eqref{eq:future-geodesic-energy}.
Taking reciprocals and integrating
$\dd\lambda/\dd t=E^{-1}$ gives
\eqref{eq:future-affine-integral} and
\eqref{eq:theta-affine-criterion}.

Every invariant Riemannian metric on connected $X$ is complete.
To see this, the compact unit sphere at $o$ gives a uniform local
existence time for unit-speed geodesics starting there.  Transitive
isometries give the same existence time at every point, so every
Riemannian geodesic continues indefinitely.  In particular,
$h_{t_0}$ is complete and its closed bounded balls are compact.
If a future causal geodesic had bounded $t$ and finite affine
length, the finite-slab metric comparison and the energy estimate
would bound its spatial length and its full tangent relative to
$h_{t_0}$.  Its state would remain in a compact subset of the
tangent bundle of a finite slab, so the geodesic equation would
extend it.  For $t\to\infty$, $\beta_n<1$ when $2\le n\le5$, and
\[
 \lambda(t)-\lambda(t_0)
 \ge\frac{(1+\theta_0(t-t_0))^{1-\beta_n}-1}
 {E(t_0)\theta_0(1-\beta_n)}\longrightarrow\infty.
\]
This proves causal completeness on $X$.  The action of $\Gamma$
preserves every $h_t$ and fixes time and $\Phi$.  It defines a
Lorentzian covering, and a quotient geodesic lifts to $I\times X$.
Completeness of the lift proves completeness of the quotient.
\end{proof}

\begin{proposition}[Potentials bounded below by a positive constant]
\label{prop:positive-floor-future-completeness}
Under the hypotheses of
Proposition~\ref{prop:nonnegative-future-completeness}, suppose that
\begin{equation}
 \mathcal V(\Phi(t))\ge V_*>0\quad\hbox{for }t\ge t_0.
 \label{eq:positive-potential-floor}
\end{equation}
Then future timelike and null geodesic completeness holds in every
spatial dimension $n\ge2$, on $X$ and all the stated quotients.
With $\theta_*:=\sqrt{2nV_*/(n-1)}$, one has
\begin{equation}
 E(t)\le E(t_0)\left(\frac{\theta_0}{\theta_*}\right)^{\beta_n},
 \qquad
 \lambda(t)-\lambda(t_0)
 \ge\frac{t-t_0}{E(t_0)}
       \left(\frac{\theta_*}{\theta_0}\right)^{\beta_n}.
 \label{eq:positive-floor-affine-length}
\end{equation}
\end{proposition}

\begin{proof}
Future existence was proved without a dimension restriction.
The constraint gives
\[
 \frac{n-1}{n}\theta^2
 =|\Sigma|^2+|\psi|^2+2\mathcal V-R\ge2V_*,
\]
so $\theta\ge\theta_*$.  Substitution into
\eqref{eq:future-geodesic-energy} bounds $E$ as stated.
Integration of $\dd\lambda/\dd t=1/E$ gives the affine-length
bound.  Bounded-$t$ geodesics and quotient lifts are treated in the
preceding proof.
\end{proof}

\begin{corollary}[Slowly decaying potentials in arbitrary dimension]
\label{cor:decaying-potential-completeness}
Under the hypotheses of
Proposition~\ref{prop:nonnegative-future-completeness}, let $n\ge3$ and
$\beta_n=(n-1)^2/(4n)$.  If
\begin{equation}
 \int_{t_0}^{\infty}
       \mathcal V(\Phi(t))^{\beta_n/2}\,\dd t=\infty,
 \label{eq:potential-affine-criterion}
\end{equation}
then the spacetime is future timelike and null geodesically complete on
$X$ and on every allowed discrete quotient.  In particular, it is
complete if, for all sufficiently large $t$,
\begin{equation}
 \mathcal V(\Phi(t))\ge C(1+t)^{-\alpha},
 \qquad C>0,\qquad
 0\le\alpha\le\frac{8n}{(n-1)^2}.
 \label{eq:polynomial-potential-floor}
\end{equation}
No positive uniform lower bound for the potential is required.
\end{corollary}

\begin{proof}
The Hamiltonian constraint and $R\le0$ give
\[
 \frac{n-1}{n}\theta^2
 =|\Sigma|^2+|\dot\Phi|_{\mathsf G}^2
   +2\mathcal V(\Phi)-R
 \ge2\mathcal V(\Phi).
\]
Since $\theta>0$,
$\theta^{\beta_n}\ge
[2n/(n-1)]^{\beta_n/2}\mathcal V(\Phi)^{\beta_n/2}$.
Thus \eqref{eq:potential-affine-criterion} implies
\eqref{eq:theta-affine-criterion}.  Under
\eqref{eq:polynomial-potential-floor}, the integrand in
\eqref{eq:potential-affine-criterion} is bounded below by a positive
multiple of $(1+t)^{-\alpha\beta_n/2}$; its integral diverges in the
stated range, including the logarithmic endpoint.
\end{proof}

The positive lower bound prevents the mean curvature from tending
to zero.  No convergence of the field or of a rescaled spatial
metric is needed.  The hypothesis can also be imposed only along
the future field trajectory.  Stronger asymptotic conclusions need
additional properties of the potential; the quadratic real-scalar
case is treated in Appendix~\ref{app:massive-future}.

\begin{proposition}[Vacuum and massless fields through nine spatial dimensions]
\label{prop:massless-future-completeness}
Under the hypotheses of
Proposition~\ref{prop:nonnegative-future-completeness}, suppose that
$\mathcal V=0$.  Estimate~\eqref{eq:future-geodesic-energy} then
holds, in every spatial dimension, with $\beta_n$ replaced by
\begin{equation}
 c_2=0,\qquad c_3=\frac13,\qquad
 c_n=\frac{\sqrt n-1}{2}\quad(n\ge4).
 \label{eq:massless-future-exponent}
\end{equation}
For $2\le n\le9$ the spacetime is future timelike and null
geodesically complete on $X$ and all the stated quotients.
\end{proposition}

\begin{proof}
For $\mathcal V=0$, \eqref{eq:higher-future-identities} gives
\begin{equation}
 d=|K|^2+|\psi|^2,\qquad
 \frac{\theta^2}{n}\le d\le\theta^2,\qquad
 |\Sigma|^2\le d-\frac{\theta^2}{n}.
 \label{eq:massless-expansion-constraint}
\end{equation}
Hence
\[
 \lambda_{\min}(h^{-1}K)
 \ge\frac{\theta-\sqrt{(n-1)(nd-\theta^2)}}{n}.
\]
For $y=d/\theta^2\in[1/n,1]$, the smallest nonnegative
coefficient in the resulting bound
$\lambda_{\min}(h^{-1}K)\ge-c_nd/\theta$ is
\[
 \max_{1/n\le y\le1}
 \frac{\sqrt{(n-1)(ny-1)}-1}{ny}.
\]
Put $z=\sqrt{(n-1)(ny-1)}\in[0,n-1]$.  The expression becomes
\[
 f_n(z)=\frac{(n-1)(z-1)}{z^2+n-1},\qquad
 f_n'(z)=\frac{(n-1)(-z^2+2z+n-1)}{(z^2+n-1)^2}.
\]
For $n=2,3$, the maximum occurs at $z=n-1$ and equals $0,1/3$.
For $n\ge4$, it occurs at $z=1+\sqrt n$ and equals
$(\sqrt n-1)/2$.  The preceding geodesic argument therefore
applies with $c_n$.  It gives divergent affine length for $n\le8$
because $c_n<1$.  At $n=9$, $c_9=1$ and
\[
 \lambda(t)-\lambda(t_0)
 \ge\frac{\log(1+\theta_0(t-t_0))}{E(t_0)\theta_0}
 \longrightarrow\infty.
\]
The finite-slab and covering arguments complete the proof.
\end{proof}

\begin{proposition}[Energy-condition completeness for general matter]
\label{prop:energy-condition-future}
Let $g=-\dd t^2+h_t$ be a spatially homogeneous solution of
$\operatorname{Ein}(g)=T$ on $[t_0,\infty)\times X$, where
$\dim X=n\ge2$.  Suppose that $\theta(t_0)>0$ and $R(h_t)\le0$.
With $\nu=\partial_t$, set
\[
 \rho=T(\nu,\nu),\qquad S=T|_{TX\times TX},\qquad
 \mathfrak q=\frac{(n-2)\rho+\operatorname{tr}_{h_t}S}{n-1},
\]
and assume along the solution that
\begin{equation}
 0\le\mathfrak q\le2\rho.
 \label{eq:matter-energy-window}
\end{equation}
Then every future causal geodesic satisfies
\begin{equation}
 E(t)\le E(t_0)
 \left(\frac{\theta(t_0)}{\theta(t)}\right)^{c_n},
 \label{eq:matter-future-energy}
\end{equation}
with the constants $c_n$ in
\eqref{eq:massless-future-exponent}.  Hence the spacetime is future
timelike and null geodesically complete for $2\le n\le9$, as is every
quotient by a freely and properly discontinuously acting group of common
spatial isometries.  For a perfect fluid, condition
\eqref{eq:matter-energy-window} is precisely
\begin{equation}
 -\frac{n-2}{n}\rho\le p\le\rho.
 \label{eq:perfect-fluid-energy-window}
\end{equation}
The hypothesis that the solution exists for all $t\ge t_0$ is explicit;
future continuation must be established from the evolution law of the
chosen matter model.
\end{proposition}

\begin{proof}
The Hamiltonian constraint and the normal Raychaudhuri equation give
\[
 \theta^2=|K|^2+2\rho-R,
 \qquad d:=-\dot\theta=|K|^2+\mathfrak q.
\]
Condition \eqref{eq:matter-energy-window} and $R\le0$ imply
\[
 \frac{\theta^2}{n}\le d\le\theta^2,
 \qquad
 |\Sigma|^2\le d-\frac{\theta^2}{n}.
\]
The eigenvalue optimization in the proof of
Proposition~\ref{prop:massless-future-completeness} therefore applies
without change and proves \eqref{eq:matter-future-energy}.  Since
$d\le\theta^2$, initially positive $\theta$ remains positive and
$\theta(t)\ge\theta(t_0)/(1+\theta(t_0)(t-t_0))$.  The affine-length
argument diverges for $c_n\le1$, including logarithmically for $n=9$.
The covering argument gives the quotient conclusion.  For
$S=ph_t$, the two inequalities in
\eqref{eq:matter-energy-window} reduce to
\eqref{eq:perfect-fluid-energy-window}.
\end{proof}

The constants~\eqref{eq:massless-future-exponent} optimize the
displayed algebraic estimate; no optimality among all Einstein
evolutions is asserted for $n\ge4$.  The dimensional cutoffs are
sufficient conditions, not assertions of incompleteness above them.
In three dimensions, the vacuum Kasner metric
$-\dd t^2+t^{-2/3}\dd x^2+t^{4/3}(\dd y^2+\dd z^2)$ realizes
$c_3=1/3$: a null geodesic with conserved $x$ momentum $P\ne0$
and zero other momenta has $E=|P|t^{1/3}$ and $\theta=1/t$.

For comparison, \cite[Theorem~3]{GoedekeRendall2010} proves vacuum
completeness for $n<9$ under an oscillation bound on the least
generalized Kasner exponent.  Its Theorem~4 assumes $p_a\ge-1/n$,
and its Theorem~5 treats product groups with additional reflection
symmetries.  The estimate above removes the oscillation restriction
in the overlapping range and includes $n=9$.  It applies as well
to homogeneous spaces with isotropy and to arbitrary finite numbers
of scalar fields with a complete positive-definite target metric.
For potentials with a positive lower bound, it complements the
stronger three-dimensional real-scalar asymptotic results of
\cite[Theorems~1-2]{Rendall2004Scalar}.

\begin{corollary}[Four spatial dimensions]\label{cor:four-spatial-vacuum}
Let $G$ be a connected, simply connected four-dimensional Lie group
not isomorphic to $\mathrm{SU}(2)\times\mathbb R$.
Every maximal homogeneous vacuum development on $G$ that is
expanding at one slice is future timelike and null geodesically
complete.  The same conclusion holds on every discrete left quotient.
\end{corollary}

\begin{proof}
The four-dimensional Lie-group discussion in
\cite[Section~5]{GoedekeRendall2010} shows that every left-invariant
Riemannian metric on $G$ has nonpositive scalar curvature.
Apply Proposition~\ref{prop:massless-future-completeness} with $n=4$.
\end{proof}

This gives the conclusion proposed in
\cite[Section~5]{GoedekeRendall2010}.

\begin{corollary}[Spaces without positive invariant scalar curvature]
\label{cor:structural-homogeneous-future}
Let $X=L/H$ be a connected homogeneous space of dimension
$2\le n\le9$ such that $R(h)\le0$ for every $L$-invariant Riemannian
metric $h$ on $X$.  Every maximal expanding homogeneous vacuum or
Einstein-massless-sigma-model development on $X$, with complete
Riemannian target, is future timelike and null geodesically complete.
The same holds on every quotient by common spatial isometries acting
freely and properly discontinuously.
\end{corollary}

\begin{proof}
The structural hypothesis supplies $R(h_t)\le0$ throughout the
development.  Proposition~\ref{prop:massless-future-completeness}
applies directly.
\end{proof}

\section{Boundary localization on homogeneous quotients}
\label{sec:quotient}

\subsection{Causal length and timelike filling}

A timelike homotopy with fixed endpoints is a continuous map from a
closed parameter rectangle into the spacetime whose curve slices are
piecewise $C^1$, future-directed timelike curves with the same endpoints.
This is the convention used in future one-connectedness and in the local
filling statements below.

Fix a base point $o\in Y$.  The least spatial scale relative to $k$ is
\begin{equation}
 m(t)^2=\min\{(\pi^*h_t)_o(v,v): |v|_k=1\}.
 \label{eq:m-def}
\end{equation}
Upstairs homogeneity makes the same lower bound valid at every point, and it
descends to $X$.  Define the full causal spatial length by
\begin{equation}
 \Lambda_k(T)=\sup_c L_k(\operatorname{pr}_X\circ c),
 \label{eq:causal-length}
\end{equation}
where $c$ ranges over future-directed causal curves contained in
$(0,T]\times X$.  The causal inequality gives
\begin{equation}
 \Lambda_k(T)\le \int_0^T\frac{N(t)}{m(t)}\,\dd t.
 \label{eq:causal-length-bound}
\end{equation}
In particular, finiteness of the improper integral on some past interval
implies $\Lambda_k(T)\to0$.

Every product slice is Cauchy.  Indeed, $t$ is strictly increasing on a
nonconstant future causal curve.  If the upper endpoint of its $t$-range
were an interior value $t_*<t_+$, then on a compact interval about $t_*$
the causal speed would satisfy $|\dot x|_k\le C$.  Thus $x(t)$ would have
a limit by completeness of $k$.  Appending a vertical timelike segment at
$(t_*,\lim x(t))$ would extend the causal curve.  The same argument applies
at its lower endpoint.  Every inextendible causal curve therefore meets
each product slice exactly once, and the spacetime is globally hyperbolic.

We use the standard notion of future one-connectedness: two piecewise $C^1$
future timelike curves with the same endpoints must be joined, with endpoints
fixed, through future timelike curves.  For quotient localization the
following weaker tail property is sufficient.

\begin{definition}[Axial timelike filling]
The past end of \eqref{eq:intro-spacetime} has \emph{axial timelike filling}
if, for every past-inextendible future timelike curve
$\gamma:(0,s_1]\to\cM$ with $t(\gamma(s))\to0$, there is $s_0>0$ such that,
whenever $0<s<s_+\le s_0$, every future timelike curve from $\gamma(s)$ to
$\gamma(s_+)$ is timelike-homotopic with fixed endpoints to the corresponding
segment of $\gamma$.
\end{definition}

Future one-connectedness on an early slab implies axial filling.  The latter
also descends through a quotient once winding has been excluded.

\begin{lemma}[Descent of axial filling]
\label{lem:filling-descent}
Assume $\Lambda_k(T)\to0$.  If the lifted spacetime
\[
 (0,t_+)\times Y,\qquad -N^2\dd t^2+\pi^*h_t,
\]
has axial timelike filling, then the quotient spacetime has axial timelike
filling.
\end{lemma}

\begin{proof}
Let $\gamma(s)=(T(s),x(s))$ be a past-inextendible future timelike curve with
$T(s)\downarrow0$.  Vanishing causal spatial length makes $x(s)$ a
$k$-Cauchy curve.  Completeness of the descended metric gives
$x(s)\to x_0$.  Choose an evenly covered neighborhood $O$ of $x_0$, a
connected $O_0\Subset O$ containing $x_0$, and one sheet $\widehat O$ over
$O$.  A sufficiently short tail of $\gamma$ lies in $O_0$ and lifts to
$\widehat O$.

Choose the terminal threshold $s_0$ within the lifted filling range and so
small that
\[
 \Lambda_k(T(s_0))<\operatorname{dist}_k(\overline O_0,X\setminus O).
\]
Every such competitor lies in $(0,T(s_0)]\times X$, because $t$ is
strictly increasing along it.  Its spatial length is therefore at most
$\Lambda_k(T(s_0))$, and it remains in $O$.  The lift beginning on the
chosen component $\widehat O$ stays in that component and ends at the
corresponding lifted terminal point, rather than at a deck translate.
Lifted axial filling supplies the
fixed-endpoint timelike homotopy, and projection gives the required homotopy
on $X$.
\end{proof}

\subsection{Local homogeneous transporters}

The following local isometries align nearby points in a fixed covering
sheet.  They transport the compact separator constructed below.  On a Lie
group with trivial quotient, they are the usual left translations; on a
general quotient, their domains and inverse identities retain the
covering-sheet information needed by the proof.

\begin{lemma}[Local homogeneous transporters]
\label{lem:transporters}
Let $O\subset X$ be evenly covered, let
$s:O\to\widehat O\subset Y$ select one sheet, and choose connected domains
\[
 U_0\Subset U_1\Subset U_2\Subset O.
\]
There is a symmetric identity neighborhood $V_*\subset L$ such that, for
$a\in V_*$,
\begin{equation}
 \vartheta_a:U_2\longrightarrow O,\qquad
 \vartheta_a(x)=\pi(a\cdot s(x)),
 \label{eq:transporter}
\end{equation}
is a local isometry of $k$ and every $h_t$, maps $U_1$ into $U_2$, and
obeys
\begin{equation}
 s(\vartheta_a(x))=a\cdot s(x)\quad(x\in U_2),
 \qquad
 \vartheta_{a^{-1}}(\vartheta_a(x))=x\quad(x\in U_1).
 \label{eq:transporter-inverse}
\end{equation}
Moreover, if $C\Subset U_1$ and $V\subset V_*$ is an identity neighborhood,
there is $\eta>0$ such that
\begin{equation}
 x,y\in C,\quad d_k(x,y)<\eta
 \quad\Longrightarrow\quad
 \vartheta_a(x)=y\ \text{for some }a\in V.
 \label{eq:transporter-alignment}
\end{equation}
\end{lemma}

\begin{proof}
Because $s(\overline U_2)\Subset\widehat O$, continuity of the action gives a
symmetric $V_*$ such that
\[
 a^{\pm1}s(\overline U_2)\subset\widehat O,
 \qquad a^{\pm1}s(\overline U_1)\subset s(U_2).
\]
Injectivity of $\pi$ on $\widehat O$ gives the first identity in
\eqref{eq:transporter-inverse}; applying it successively to $a$ and $a^{-1}$
gives the second.  Thus no deck transformation is hidden in the inverse.
The maps in \eqref{eq:transporter} are local isometries because the sheet
map, the $L$-action, and the covering projection are local isometries.

For the last assertion, the map
$L\times Y\to Y\times Y$, $(a,y)\mapsto(a\cdot y,y)$, is a submersion:
transitivity makes the derivative in the $L$ variable surjective onto
the first tangent factor, and the $Y$ variable supplies the second.
It is therefore open.  The image of $V\times Y$ contains an open
neighborhood of the diagonal.  Compactness of $s(\overline C)$ gives
$\eta_{\rm tr}>0$ such that every pair of its points at distance less
than $\eta_{\rm tr}$ belongs to this image.  Such a pair is related by
an element of $V$.  Shrink $\eta$ below both $\eta_{\rm tr}$ and
$\operatorname{dist}_k(C,X\setminus U_2)$.  A minimizing $k$-geodesic between
$x$ and $y$ then stays in $U_2$, lifts to the selected sheet, and has the
same length.  The upstairs transporter projects to the required
$\vartheta_a$.
\end{proof}

\subsection{The boundary chart and localization}

A $C^0$ extension consists of a connected Hausdorff second-countable smooth
manifold $\widetilde\cM$ of the same dimension, a continuous Lorentzian
metric $\widetilde g$, and a smooth isometric embedding
$\iota:(\cM,g)\to(\widetilde\cM,\widetilde g)$ whose image is a proper
open subset.  In particular its boundary is nonempty.  No field equation,
global causal condition, or spatial symmetry is required of the extension.

A point of $\partial\iota(\cM)$ belongs to $\partial^-\iota(\cM)$ if it
is the initial endpoint of a smooth future-directed timelike curve in the
extension whose remaining points lie in $\iota(\cM)$.  The future boundary
$\partial^+\iota(\cM)$ is defined with a terminal endpoint.  Timelikeness
at the endpoint is understood in the extension metric.

\begin{definition}[One-chart localizable past end]
\label{def:one-chart}
The past end is \emph{one-chart localizable} if every $C^0$ extension
$\iota:\cM\to\widetilde\cM$ with nonempty past boundary admits a chart
$\widetilde\varphi:\widetilde U\to Q=(-\eps_0,\eps_0)\times B$, where
$B=(-\eps_1,\eps_1)^d$ and $\eps_0,\eps_1>0$, a Lipschitz
function $f:B\to(-\eps_0,\eps_0)$, a number
$c\in(\sup_Bf,\eps_0)$, an open set $W\Subset Q$, and $q\in\cM$ such that,
with $E_f=\{(s,x)\in Q:s>f(x)\}$,
\begin{align}
 \widetilde\varphi(\iota(\cM)\cap\widetilde U)&\supset E_f,
 \notag\\
 \widetilde\varphi^{-1}(\{(f(x),x):x\in B\})&\subset
 \past\iota(\cM),
 \notag\\
 \iota(I^-(q))&\subset\widetilde\varphi^{-1}(W),&
 \widetilde\varphi(\iota(I^-(q)))&\subset E_f,
 \label{eq:one-chart-data}\\
 \widetilde\varphi^{-1}(\{c\}\times B)&\subset\iota(I^+(q)).
 \notag
\end{align}
The chart is chosen so that $\partial_0$ is future timelike and fixed inner
and outer Euclidean cones bound the causal cones throughout $Q$.
\end{definition}

For the standard regular boundary chart, originating in the boundary-graph
construction of \cite[Proposition 2.2]{SbierskiProof2018}, we use the form
in \cite[Lemma 2.16]{Miethke2024}: a globally hyperbolic smooth spacetime with
a past boundary point has such a chart with $f(0)=0$, achronal boundary
graph, full epigraph in the spacetime image, and metric coefficients
arbitrarily close to Minkowski coefficients.  The terminal curve is smooth and timelike in the extension manifold \cite[Lemma~2.17]{Sbierski2018}.
As in the proof of \cite[Lemma~2.16]{Miethke2024}, first straighten it
to the positive time axis, with its endpoint at the origin.  A linear
normalization preserving that axis makes the metric Minkowskian at the
origin.  Carry out the boundary-graph construction in these adapted
coordinates.  The graph then has $f(0)=0$, and shrinking the chart gives
the required uniform cone bounds.  These are local chart
properties, not the global confinement required in
Definition~\ref{def:one-chart}.  General continuous-metric causality is
treated in \cite{ChruscielGrant2012,Samann2016}.

Only a local time orientation in the chart is needed.  If the extension
is not time orientable, pull it back to its time-orientation double cover
and select the sheet over the time-oriented interior.  This retains the
boundary curve and an isometric copy of the interior, so the local
arguments below and the one-sided completeness obstruction still apply.

Two elementary consequences of the graph property will be used repeatedly.
First, a continuous curve in the spacetime image, contained in the chart
and ending in $E_f$, stays in $E_f$: otherwise the continuous function
$x^0-f(x)$ vanishes at an interior point of the spacetime, contrary to the
boundary-graph property.  Second, a future timelike chart curve starting
in $E_f$ cannot leave it.  The vertical segment from the graph to its
initial point, followed by a first crossing of the graph, would join two
graph points by a timelike curve, contradicting achronality.

For a boundary chart $Q$, let $I_Q^\pm$ denote chronological sets defined by
timelike curves that remain in $Q$, and write $0=(0,0)$ for the origin.

\begin{lemma}[Boundary lens and timelike separator]
\label{lem:boundary-lens}
Let a $C^0$ extension have nonempty past boundary, and let $\gamma$ be a
future timelike terminal curve contained in a tail with axial timelike
filling.  Its product-time coordinate tends to zero and is strictly
increasing toward the future.  After restricting the curve and a regular
boundary chart, reparametrize it by the boundary-chart coordinate $x^0$
and denote that parameter again by $s$.  Decreasing the inherited axial
filling threshold $s_0$ if necessary, one may arrange
\begin{equation}
 \begin{aligned}
 \widetilde\varphi(\iota(\gamma(s)))&=(s,0),&
 0&<s_-<s_+<s_b<c<\eps_0,\quad s_b\le s_0,\\
 q^\pm&=\gamma(s_\pm),&
 y^\pm&=(s_\pm,0),\qquad b=(s_b,0).
 \end{aligned}
 \label{eq:lens-axis-data}
\end{equation}
and find a relatively compact open lens $W\Subset Q$ with the following
properties, where $\sup_B f<c$:
\begin{align}
 \widetilde\varphi^{-1}\!\left(
 I_Q^+((s,0))\cap I_Q^-(y^+)\right)
 &=\iota\!\left(I^+(\gamma(s))\cap I^-(q^+)\right),
 &&0<s<s_+,
 \label{eq:local-global-diamond}\\
 \overline{I_Q^+(0)\cap I_Q^-(y^+)}&\subset W,&
 \{c\}\times B&\subset I_Q^+(y^+).
 \label{eq:lens-properties}
\end{align}
Every point of $E_f\cap\partial W\cap I_Q^-(y^+)$ can be joined to
$b$ by a future timelike curve in $E_f\cap(Q\setminus W)$.  Moreover, with
\begin{equation}
 A=\bigcup_{0<s<s_+}\bigl(I^+(\gamma(s))\cap I^-(q^+)\bigr),
 \qquad K=\overline A^{\,\cM}\setminus I^-(q^-),
 \label{eq:boundary-lens-separator}
\end{equation}
$K$ timelike-separates $\gamma((0,s_-))$ from $I^+(q^+)$: every future
timelike curve from any point of the first set to any point of the second
meets $K$.
\end{lemma}

\begin{proof}
Choose $\alpha>1$ such that every nonzero vector in
$C_\alpha^+=\{(v^0,v):v^0\ge\alpha|v|,\ v^0>0\}$ is future timelike
throughout a sufficiently small regular boundary chart.  Put
$C_\alpha^-=-C_\alpha^+$.  Choose $0<c<\eps_0$ and then narrow $B$ so that
$\sup_{\overline B}f<c$ and $\alpha\sup_B|x|<c/2$.
Choose $r_B>0$ with $\overline B_{\rm E}(0,r_B)\subset B$, and positive
$s_a,s_b$ such that $s_b<c/2$, $s_a<\eps_0$, and
$(s_a+s_b)/(2\alpha)<r_B$.  With $a=(-s_a,0)$, define
\[
 W=(a+\mathring C_\alpha^+)\cap(b+\mathring C_\alpha^-),
 \qquad b=(s_b,0).
\]
Its closure lies in
$[-s_a,s_b]\times\overline B_{\rm E}(0,(s_a+s_b)/(2\alpha))\Subset Q$,
and the origin is an interior point of $W$.  The outer cone bound implies
that the closure of $I_Q^+(0)\cap I_Q^-((s,0))$ has Euclidean diameter
tending to zero with $s$.  Choose $s_+>0$ small enough to place this
closure in $W$, with $s_+<s_b$, and choose $s_-\in(0,s_+)$.
Straight segments from $y^+$ to the top slab have slope in
$\mathring C_\alpha^+$, proving the second assertion in
\eqref{eq:lens-properties}.

Every point on the lower cone surface can be followed along its outward
generator to the ridge of $\partial W$ and then along an inward upper
generator to $b$.  From a point on the upper surface only the second
segment is needed.  At the lower vertex choose any generator.  All these
segments have future tangents in $C_\alpha^+$ and lie on $\partial W$.
They are strictly timelike for the extension metric.  If their initial
point belongs to $E_f$, the preceding achronality argument keeps them in
$E_f$.  This proves the escape assertion.

For the local-to-global inclusion in \eqref{eq:local-global-diamond},
concatenate local timelike curves through the chosen point.  They start
at $(s,0)\in E_f$, hence remain in $E_f$ and pull back to the spacetime.
Conversely, concatenate any two interior timelike curves through a point
of the global diamond.  Axial filling gives a fixed-endpoint timelike
homotopy $H:[0,1]^2\to\cM$ between this curve and the axis segment.  Put
\[
 \mathcal S=\{\vartheta\in[0,1]:
 \iota(H(\vartheta,[0,1]))\subset\widetilde\varphi^{-1}(Q)\}.
\]
The set $\mathcal S$ is nonempty and open, by compactness of each curve
image and continuity of $H$.  It is also closed:
every such curve has its image, including its endpoints, in
\[
 S^*=\overline{I_Q^+(0)\cap I_Q^-(y^+)}\Subset W.
\]
If parameters in $\mathcal S$ converge, pointwise continuity and closedness of the
compact set $\widetilde\varphi^{-1}(S^*)$ in the Hausdorff extension put
the limiting curve in the same set.  Thus $\mathcal S$ is consequently
all of $[0,1]$, proving the reverse inclusion.  Compactness here comes
from the fixed conical chart, not from a limit-curve assertion about the
extension.

Finally let $\sigma:[0,1]\to\cM$ be future timelike, starting on
$\gamma((0,s_-))$ and ending in $I^+(q^+)$.  Define
\[
 u_\pm=\sup\{u\in[0,1]:\sigma(u)\in I^-(q^\pm)\}.
\]
These suprema satisfy $0<u_-<u_+<1$.  Indeed, global
hyperbolicity gives causal simplicity and
\[
 \overline{I^-(q^-)}=J^-(q^-)\subset I^-(q^+),
\]
where the inclusion is the smooth interior push-up property.  For
$u\in(u_-,u_+)$, $\sigma(u)$ is in the future of its initial axis point
and in $I^-(q^+)$, but not in $I^-(q^-)$.  Thus $\sigma(u)\in K$.
\end{proof}

\begin{proposition}[Boundary localization on homogeneous quotients]
\label{thm:quotient-localization}
Assume
\begin{equation}
 \Lambda_k(T)\longrightarrow0\qquad(T\downarrow0)
 \label{eq:vanishing-length}
\end{equation}
and suppose that the lifted spacetime has axial timelike filling.  Then the
past end of \eqref{eq:intro-spacetime} is one-chart localizable.  The proof
uses local isometries on $X$ only; no lift of the candidate extension is
required.
\end{proposition}

\begin{proof}
Suppose that a $C^0$ extension has nonempty past boundary.  Let
$\gamma(s)=(T(s),X(s))$ be a terminal timelike curve obtained from a regular
boundary graph chart.  Its time coordinate tends to zero at the past
endpoint.  Indeed, a positive interior limiting time would give a uniform
spatial speed bound on a compact time interval; completeness of $k$ would
then give an interior endpoint of $\gamma$.  Its image in the extension
would coincide with the prescribed boundary endpoint by Hausdorffness,
which is impossible.  Lemma~\ref{lem:filling-descent} supplies axial filling
on the quotient.  By \eqref{eq:vanishing-length}, $X(s)\to x_0$.  Since $t$
is a time function, $T$ is strictly increasing and has a continuous inverse
on $(0,T(s_b)]$ after the final choice of the retained endpoint.

Choose an evenly covered connected neighborhood $O$ and $r>0$ such that
$\overline B_k(x_0,4r)\subset O$.  Restrict the chart and terminal curve far
enough toward the past that $X(s)\in B_k(x_0,r)$ and
$\Lambda_k(T(s))<r$ throughout the retained interval.  Apply
Lemma~\ref{lem:boundary-lens} to this restriction, thereby fixing all
the data in \eqref{eq:lens-axis-data}-\eqref{eq:lens-properties}.  In
particular,
\begin{equation}
 X(s)\in B_k(x_0,r),\qquad
 \Lambda_k(T(s_b))<r
 \label{eq:short-terminal}
\end{equation}
for every retained parameter $s\le s_b$.  Every timelike curve between two
retained axis points then remains in $B_k(x_0,2r)$.  Points in the early
chronological past used in the separator construction lie in the same ball,
and curves from such points to $\gamma(s_b)$ remain in $B_k(x_0,3r)$.
Choose connected domains
\begin{equation}
 B_k(x_0,2r)\Subset U_0\Subset U_1,\qquad
 \overline B_k(x_0,3r)\Subset U_1\Subset U_2\Subset O.
 \label{eq:germ-domains}
\end{equation}

Let $A$ and $K$ be the sets in
\eqref{eq:boundary-lens-separator}.  The boundary-lens lemma gives the
local/global diamond identity and shows that $K$ timelike-separates the early
axis from $I^+(q^+)$.  The path-length bound puts
$\operatorname{pr}_X(\overline A^{\cM})\subset\overline B_k(x_0,2r)
\Subset U_0$: use a timelike connector from a point of $A$ to $q^+$,
whose spatial length is less than $r$, and then take limits in $\cM$.
To see compactness, choose
$s_\flat<s_-$ and a spatial neighborhood $O_-$ of $X(s_\flat)$ with
$\{T(s_\flat)\}\times O_-\subset I^-(q^-)$.  Shrink an identity neighborhood
$V$ so that $\vartheta_a(X(s_\flat))\in O_-$ for $a\in V$.  The alignment part of
Lemma~\ref{lem:transporters} gives $\rho>0$ for
$C=\overline B_k(x_0,3r)$; reduce $\rho$ so that $\rho<r$.
Every point of $B_k(X(s),\rho)$ is in $C$, and the entire retained axis
segment is in $U_1$.  Applying the spacetime local isometry
\[
 \Theta_a(t,x)=(t,\vartheta_a(x))
\]
to the retained axis segment shows that
\begin{equation}
 \{T(s)\}\times B_k(X(s),\rho)\subset I^-(q^-)
 \qquad(0<s\le s_\flat).
 \label{eq:translated-tube}
\end{equation}
Choose $0<T_1<T_2<T(s_\flat)$ with $2\Lambda_k(T_2)<\rho$.  Comparing a curve
from the early axis to a point $(t,x)\in A$ with the corresponding axis
segment gives
$d_k(x,X(T^{-1}(t)))\le2\Lambda_k(T_2)<\rho$ whenever $t\le T_2$.
If $p_j=(t_j,x_j)\in A$ converges to $p=(t,x)\in\cM$, then $t>0$.
When $t\le T_1$, eventually $t_j\le T_2$, and the preceding
$2\Lambda_k(T_2)$ estimate passes to the limit.  Hence
$d_k(x,X(T^{-1}(t)))<\rho$, so $p$ lies in the open tube
\eqref{eq:translated-tube} and cannot belong to $K$.  Consequently
\[
 K\subset[T_1,T(s_+)]\times\overline B_k(x_0,2r).
\]
Hopf-Rinow and closedness of $K$ prove compactness.  To place $\iota(K)$
in the lens, take $p_j\in A$ converging to $p\in K$.  Their chart images
lie in $S^*$ by the local/global diamond identity.  A subsequence converges
in this compact set, and uniqueness of limits in the Hausdorff extension
identifies its image with $\iota(p)$.  Hence
$\widetilde\varphi(\iota(K))\subset S^*\Subset W$.

It remains to propagate this separator.  By compactness, shrink a symmetric
$V\Subset V_*$ so that
\begin{equation}
 \Theta_{a^{\pm1}}(K)\subset W_{\cM},
 \qquad \Theta_a(\gamma(s_b))\in I^+(q^+)
 \quad(a\in V),
 \label{eq:separator-propagation}
\end{equation}
where
\[
 W_{\cM}:=\iota^{-1}\!\left(
 \widetilde\varphi^{-1}(W)\cap\iota(\cM)\right)
\]
is the part of the lens lying in $\cM$.  Let $\eta$ be the
alignment radius supplied by Lemma~\ref{lem:transporters} for
$C=\overline B_k(x_0,3r)$ and the final neighborhood $V$ in
\eqref{eq:separator-propagation}, reduced below $r$, and choose $s_*<s_-$ with
$2\Lambda_k(T(s_*))<\eta$.  Put $q=\gamma(s_*)$.  For
$p=(t_p,x_p)\in I^-(q)$, comparison with the axis segment at the same time
gives
\[
 d_k(x_p,X(T^{-1}(t_p)))\le2\Lambda_k(T(s_*))<\eta.
\]
Thus some $a_p\in V$ satisfies
$\vartheta_{a_p}(x_p)=X(T^{-1}(t_p))$.  Here $x_p\in B_k(x_0,2r)$.
Every timelike curve from $p$ to $\gamma(s_b)$ has spatial length less
than $r$, and so stays in $B_k(x_0,3r)\Subset U_1$.
The map $\Theta_{a_p}$ is therefore defined along its entire image.
The transformed curve starts on $\gamma((0,s_-))$ and ends in
$I^+(q^+)$ by \eqref{eq:separator-propagation}; it meets
$K$.  The exact inverse identity \eqref{eq:transporter-inverse} shows that the
original curve meets $W_{\cM}$.  This is precisely the step at which sheet
control is essential.

Suppose $p\in I^-(q)$ lies outside $W_{\cM}$.  Follow a past timelike
curve from $q$ to $p$ and let $z\in\cM$ be its first exit from
$W_{\cM}$.  Compactness of $\overline W$ and uniqueness of limits put
$\iota(z)$ in the chart, with chart coordinates $z_Q\in\partial W$.
The segment from $z$ to $q$ is in the spacetime image, stays in the chart,
and ends in $E_f$.  It therefore lies in $E_f$.  Concatenation with the
axis segment $q\ll q^+$ gives
$z_Q\in E_f\cap\partial W\cap I_Q^-(y^+)$.
The escape assertion of Lemma~\ref{lem:boundary-lens} now supplies a
future timelike curve from $z_Q$ to $b$ in $E_f\cap(Q\setminus W)$.
It pulls back to a curve from $z\in I^-(q)$ to $\gamma(s_b)$ avoiding
$W_{\cM}$, contrary to the separation just proved.
Thus $I^-(q)\subset W_{\cM}$.  For any point of $I^-(q)$, a timelike
connector to $q$ stays in $I^-(q)$ apart from its endpoint; the
spacetime-image argument places its chart image in $E_f$.
Finally the top-slab curves from $y^+$ remain in $E_f$ by achronality,
and $q\ll q^+$, proving all the requirements of
Definition~\ref{def:one-chart}.
\end{proof}

\section{Metric consequences of one-chart localization}
\label{sec:lipschitz}

Let $(X,h)$ be a Riemannian $d$-manifold.  For $U\subset X$, define the
$d$-dimensional Lipschitz covering content by
\begin{equation}
 \mathfrak L_d(U,h)=\inf\left\{\Lip(F):
 F:[-1,1]^d\longrightarrow(X,d_h),\quad
 U\subset F([-1,1]^d)\right\},
 \label{eq:lipschitz-content}
\end{equation}
with value $+\infty$ if there is no admissible map.  For $\rho>0$, let
$\mathsf P_\rho(U;h)$ be the supremum of the cardinalities of finite
$\rho$-separated subsets of $U$, using the ambient distance of $X$, and set

\begin{equation}
 \mathscr P_h(U)=\sup_{0<\rho\le1}\rho^d\mathsf P_\rho(U;h).
 \label{eq:normalized-packing}
\end{equation}

\begin{lemma}[Consequences of a Lipschitz cover]
\label{lem:lipschitz-certificates}
There is a dimensional constant $C_d$ such that every nonempty Borel set
$U\subset X$ satisfies
\begin{align}
 \diam_h(U;X)&\le2\sqrt d\,\mathfrak L_d(U,h),
 \label{eq:lipschitz-diameter}\\
 \Vol_h(U)&\le C_d\mathfrak L_d(U,h)^d,
 \label{eq:lipschitz-volume}\\
 \mathscr P_h(U)&\le C_d(1+\mathfrak L_d(U,h))^d.
 \label{eq:lipschitz-packing}
\end{align}
\end{lemma}

\begin{proof}
Suppose $U\subset F([-1,1]^d)$ and $\Lip(F)\le L$.  The first estimate is
immediate.  If $L=0$, the image is a single point and all the assertions
follow directly.  Suppose henceforth that $L>0$.
The metric area formula \cite{Kirchheim1994} gives
$\mathcal H_h^d(F([-1,1]^d))\le C_dL^d$, proving
\eqref{eq:lipschitz-volume}.  If $E\subset U$ is $\rho$-separated, choose one
preimage of each point.  The chosen points are $\rho/L$-separated in the
cube, and a Euclidean grid count gives
\[
 \#E\le C_d(1+L/\rho)^d.
\]
Multiply by $\rho^d$, take the supremum for $0<\rho\le1$, and then take the
infimum over $F$.
\end{proof}

\begin{proposition}[Uniform metric bounds in a localized chart]
\label{prop:chart-bounds}
Fix the extension and chart data $Q,f,c,W,q$ in
Definition~\ref{def:one-chart}.  There
is $C<\infty$ such that, for every product slice
$\Sigma_t=\{t\}\times X$ meeting $I^-(q)$, a smooth map
\begin{equation}
 F_t:[-1,1]^d\longrightarrow(\Sigma_t,d_{h_t})
 \label{eq:uniform-map}
\end{equation}
has image containing $I^-(q)\cap\Sigma_t$ and satisfies $\Lip(F_t)\le C$.
Consequently, uniformly in $t$,
\begin{align}
 \diam_{h_t}(I^-(q)\cap\Sigma_t;X)&\le C,
 \notag\\
 \Vol_{h_t}(I^-(q)\cap\Sigma_t)&\le C,
 \notag\\
 \mathsf P_\rho(I^-(q)\cap\Sigma_t;h_t)&\le C(1+\rho^{-1})^d,
 \qquad 0<\rho\le1.
 \label{eq:localized-bounds}
\end{align}
\end{proposition}

\begin{proof}
Identify $\cM$ with its image in the extension and write the boundary chart
as $Q=(-\eps_0,\eps_0)\times B$.  Its metric is uniformly close to the
Minkowski metric, $\partial_0$ is future timelike, and the epigraph of the
boundary graph lies in the spacetime image.  By one-chart localizability a
horizontal slab $\{c\}\times B$ lies to the future of $q$.

For each boundary-chart point $x\in B$, the vertical timelike line between
the boundary graph and the top slab is past-inextendible in $\cM$.
Write $(t_1,x'(x))$ for its top-slab endpoint in product coordinates and
continue it by the product vertical curve $t\mapsto(t,x'(x))$,
$t_1\le t<t_+$.  The concatenation
is a future-inextendible piecewise smooth timelike curve of $\cM$.  Since every $\Sigma_t$ is Cauchy, the
resulting inextendible curve meets $\Sigma_t$ once.  Choose
$p\in I^-(q)\cap\Sigma_t$.  Every top-slab point lies in $I^+(q)$, hence
in $I^+(p)$.  An intersection on or beyond that slab would contradict
achronality of $\Sigma_t$.  The unique intersection is therefore on the
chart-vertical portion.  If
$\tau=t\circ\iota^{-1}\circ\widetilde\varphi^{-1}$ on $E_f$, then
$\dd\tau(\partial_0)>0$.  Hence $\tau(u,x)=t$ has a unique smooth solution
$u=u_t(x)$ by the implicit-function theorem.  Thus
$\Sigma_t$ contains a graph
\[
 \Psi_t(x)=(u_t(x),x),\qquad x\in B,
\]
whose image contains $I^-(q)\cap\Sigma_t$.

Fix inner and outer Euclidean cones valid throughout $Q$.  A tangent vector
$(D u_t(x)v,v)$ to the spacelike graph cannot lie in the fixed inner
timelike cone, so $|D u_t(x)|\le L_Q$ with one chart-dependent constant.
Uniform coefficient bounds for the continuous extension metric then give
\begin{equation}
 \Psi_t^*(g|_{\Sigma_t})\le C_Q\delta_{\rm E}
 \label{eq:graph-quadratic-bound}
\end{equation}
independently of $t$, where $\delta_{\rm E}$ is the standard Euclidean metric
on $B$.  Choose one closed sub-box $B'\Subset B$ containing the spatial
projection of $\overline W$ in its interior.  Affine rescaling of $B'$ gives
\eqref{eq:uniform-map}.  Lemma~\ref{lem:lipschitz-certificates} proves the
three estimates after increasing the constant once.
\end{proof}

\begin{proposition}[Lifting localized slices]
\label{prop:lifted-slices}
Let $Q,f,c,W,q$ be the localized chart data.  Suppose $U\Subset X$ is
connected, evenly covered by $\pi:Y\to X$, and
$(0,T_*]\times U\subset I^-(q)$.  For any fixed component
$\widehat U$ of $\pi^{-1}(U)$ there is a constant $C$ such that
\[
 \mathfrak L_d(\widehat U,\pi^*h_t)\le C,
 \qquad \diam_{\pi^*h_t}(\widehat U;Y)\le2\sqrt d\,C,
 \qquad 0<t\le T_*.
\]
\end{proposition}

\begin{proof}
Identify $\Sigma_t$ with $X$.  The map $F_t$ constructed above is an
embedding of a cube: the graph map is injective, the extension embedding
is injective, and projection of a product slice onto $X$ is a
diffeomorphism.  Its image contains $U$ in the image of the interior of
the cube, because the spatial projection of $\overline W$ was placed
strictly inside the chosen box.  The restriction
$F_t|_{F_t^{-1}(U)}$ is a homeomorphism onto $U$, so $F_t^{-1}(U)$
is connected.
Choose $x_t\in F_t^{-1}(U)$ and the unique point
$\widehat x_t\in\widehat U$ above $F_t(x_t)$.  Since the cube is simply
connected, $F_t$ has a unique lift
$\widehat F_t:[-1,1]^d\to Y$ through $\pi$ with
$\widehat F_t(x_t)=\widehat x_t$.
Its differential has the same norm as that of $F_t$, since $\pi$ is a
local isometry.  Integrating along straight cube segments gives
$\Lip(\widehat F_t)\le C_0$ with the same uniform constant.
The connected set $\widehat F_t(F_t^{-1}(U))$ lies in the fixed component
$\widehat U$.  It projects onto all of $U$; since
$\pi|_{\widehat U}$ is a diffeomorphism, it equals $\widehat U$.
Consequently $\diam_{\pi^*h_t}(\widehat U;Y)\le2\sqrt d C_0$, proving
both asserted bounds without quotient shortcuts.
\end{proof}

\begin{corollary}[Metric growth on the cover]
\label{cor:cover-transfer}
Under the hypotheses of Proposition~\ref{thm:quotient-localization}, suppose
that every nonempty relatively compact open $\widehat U\subset Y$
satisfies
\[
 \limsup_{t\downarrow0}\mathfrak L_d(\widehat U,\pi^*h_t)=\infty.
\]
Then the quotient past end is $C^0$-inextendible.  Compactness of $X$ is
not required.  In particular the hypothesis follows from divergence,
on every such open set, of diameter, volume, or normalized packing
as in Lemma~\ref{lem:lipschitz-certificates}.
\end{corollary}

\begin{proof}
A past extension supplies $q$ and the localized chart.  Choose
$(T_*,x)\in I^-(q)$ and a connected evenly covered neighborhood
$U\Subset X$ with $\{T_*\}\times\overline U\subset I^-(q)$.
Vertical timelike segments imply
$(0,T_*]\times U\subset I^-(q)$.  Choose $U$ small enough that a
component of its inverse image is relatively compact in $Y$.
Proposition~\ref{prop:lifted-slices} contradicts the assumed growth.
\end{proof}

This argument does not lift the extension and does not estimate the
ambient diameter downstairs.  In particular, shortcuts winding around
the quotient do not invalidate the conclusion.  It is the embedding
property of a localized slice parametrization, not a general lifting
property of arbitrary Lipschitz covers, that supplies connected preimages.

The preceding propositions convert localized metric growth into an
inextendibility statement.  Its local formulation is important on a
quotient, where different spatial regions need not be related by global
isometries.

\begin{proposition}[Quotient metric-size obstruction]
\label{thm:quotient-obstruction}
Under the hypotheses of Proposition~\ref{thm:quotient-localization}, suppose
there is a dense set $\mathcal D\subset X$ such that, for every
$x\in\mathcal D$ and every neighborhood $V$ of $x$, some nonempty
relatively compact open $U\Subset V$ satisfies
\begin{equation}
 \limsup_{t\downarrow0}\mathfrak L_d(U,h_t)=\infty.
 \label{eq:local-content-divergence}
\end{equation}
Then the past end is $C^0$-inextendible.  If the spacetime is also future
timelike geodesically complete, it is globally $C^0$-inextendible.
\end{proposition}

\begin{proof}
If a past extension existed, Proposition~\ref{thm:quotient-localization} and
Proposition~\ref{prop:chart-bounds} would give a point $q$ and a uniform
Lipschitz bound for $I^-(q)\cap\Sigma_t$.  Openness of $I^-(q)$ and the
product structure give $T_*>0$ and a nonempty open set $V\Subset X$ with
$(0,T_*]\times V\subset I^-(q)$.  Choose $x\in\mathcal D\cap V$ and then
$U\Subset V$ as in \eqref{eq:local-content-divergence}.  For every
$t\le T_*$, restriction of the map in \eqref{eq:uniform-map} covers $U$, so
$\mathfrak L_d(U,h_t)\le C$, a contradiction.

Future timelike completeness excludes a future boundary by the theorem of
Galloway-Ling-Sbierski \cite{GLS2018}.  Every proper extension has a past
or future boundary \cite[Lemma~2.17]{Sbierski2018}, which gives the global conclusion.
\end{proof}

Two elementary tests are useful on compact quotients.

\begin{lemma}[Invariant density and circle calibration]
\label{lem:quotient-tests}
Suppose $X=\Gamma\backslash Y$ carries a family of descended invariant
metrics $h_t$ and a fixed descended invariant metric $k$.

\begin{enumerate}[label=\textup{(\alph*)}]
\item If $\dd\Vol_{h_t}=J(t)\dd\Vol_k$ and
$\limsup_{t\downarrow0}J(t)=\infty$, then
\eqref{eq:local-content-divergence} holds.

\item Let $0\ne\alpha$ be a closed one-form with integral periods and let
$f:X\to\R/\mathbb Z$ be its path-integration map, defined after choosing
a base point, so $f^*(\dd\theta)=\alpha$.
Equip $\R/\mathbb Z$ with the quotient length metric of circumference one and
write
$\operatorname{osc}_{\R/\mathbb Z}(f|_U)
=\sup_{x,y\in U}d_{\R/\mathbb Z}(f(x),f(y))$.
For every nonempty $U\subset X$,
\begin{equation}
 \diam_{h_t}(U;X)\ge
 \frac{\operatorname{osc}_{\R/\mathbb Z}(f|_U)}
 {\|\alpha\|_{L^\infty(h_t)}}.
 \label{eq:circle-calibration}
\end{equation}
\end{enumerate}
\end{lemma}

\begin{proof}
For (a), every nonempty open $U$ has
$\Vol_{h_t}(U)=J(t)\Vol_k(U)$, and Lemma~\ref{lem:lipschitz-certificates}
applies.  For (b), every piecewise $C^1$ curve $c$ from $x$ to $y$ satisfies
\[
 L_{\R/\mathbb Z}(f\circ c)
 \le\|\alpha\|_{L^\infty(h_t)}L_{h_t}(c).
\]
Minimizing over $c$ and then taking the supremum over $x,y\in U$ proves
\eqref{eq:circle-calibration}.
\end{proof}

\subsection{A collapsing Kasner torus}
\label{subsec:irrational-kasner}

The distinction between lifted and quotient diameter already occurs in
vacuum Bianchi~I geometry.  The following example is an application of
Corollary~\ref{cor:cover-transfer}, independent of the two-step estimates
below.  It complements the simply connected Kasner result
\cite{Miethke2024} by making the effect of the lattice explicit.

\begin{proposition}[Continuous inextendibility with collapsing spatial diameter]
\label{prop:irrational-kasner}
Put $\vartheta=(1+\sqrt5)/2$, $c_\vartheta=\sqrt{1+\vartheta^2}$, and define
the constant one-forms
\[
 \eta_1=\frac{\dd x+\vartheta\dd y}{c_\vartheta},\qquad
 \eta_2=\frac{-\vartheta\dd x+\dd y}{c_\vartheta}
\]
on $\mathbb T^3=\R^3/\mathbb Z^3$.  The vacuum spacetime
\begin{equation}
 \cM=(0,\infty)\times\mathbb T^3,\qquad
 g=-\dd t^2+h_t,\qquad
 h_t=t^{-2/3}\eta_1^2+t^{4/3}(\eta_2^2+\dd z^2)
 \label{eq:irrational-kasner}
\end{equation}
is globally $C^0$-inextendible.  Nevertheless, there are constants
$0<c<C<\infty$ such that
\begin{equation}
 c t^{1/6}\le\diam(\mathbb T^3,h_t)\le C t^{1/6},\qquad
 \Vol(\mathbb T^3,h_t)=t,
 \qquad 0<t\le1.
 \label{eq:irrational-kasner-collapse}
\end{equation}
On the universal cover, every nonempty relatively compact open set
$\widehat U\subset\R^3$ instead satisfies
\begin{equation}
 \diam_{\pi^*h_t}(\widehat U;\R^3)
 \ge c_{\widehat U}t^{-1/3}
 \label{eq:irrational-kasner-lifted}
\end{equation}
for some $c_{\widehat U}>0$.
\end{proposition}

\begin{proof}
The linear coordinates with differentials $\eta_1,\eta_2,\dd z$ put the
lifted metric in Kasner form with exponents $(-1/3,2/3,2/3)$.  Their sum
and sum of squares are both one, so $\Ric(g)=0$.  Relative to the standard
flat metric $k$ on $\mathbb T^3$, the least spatial scale on $0<t\le1$
is $m(t)=t^{2/3}$.  Thus
\[
 \Lambda_k(T)\le\int_0^Tt^{-2/3}\,\dd t=3T^{1/3}\longrightarrow0.
\]
The lifted spacetime is future one-connected.  Indeed, parametrize two
future timelike curves with common endpoints by $t$ on $[t_0,t_1]$ and
write their spatial parts as $X_0(t),X_1(t)\in\R^3$.
The reparametrization is itself a fixed-endpoint timelike homotopy, as
shown explicitly after \eqref{eq:two-step-spacetime}.  The interpolation
$X_s=(1-s)X_0+sX_1$ preserves endpoints and is timelike, because
\[
 |\dot X_s(t)|_{h_t}
 \le(1-s)|\dot X_0(t)|_{h_t}+s|\dot X_1(t)|_{h_t}<1.
\]
The affine function $q(x,y,z)=(x+\vartheta y)/c_\vartheta$ on the cover
satisfies $\dd q=\eta_1$ and
$|\dd q|_{(\pi^*h_t)^{-1}}=t^{1/3}$.  Every curve from $P$ to $Q$ in
the cover therefore has length at least $t^{-1/3}|q(P)-q(Q)|$.
Since $q$ is nonconstant on every nonempty open set, this proves
\eqref{eq:irrational-kasner-lifted}.  Corollary~\ref{cor:cover-transfer}
excludes a past $C^0$ extension.

To prove the upper diameter bound, let $F_0=0$, $F_1=1$ and
$F_{n+2}=F_{n+1}+F_n$.  The integer vectors
$v_n=(F_{n+1},-F_n)$ obey
\[
 \det(v_n,v_{n+1})=(-1)^{n+1},\qquad
 \eta_1(v_n)=\frac{(-1)^n\vartheta^{-n}}{c_\vartheta},\qquad
 \eta_2(v_n)=-\frac{\vartheta^{n+1}}{c_\vartheta}.
\]
These identities follow directly from the Fibonacci recurrence and
$\vartheta^2=\vartheta+1$.  Consequently $v_n,v_{n+1}$ form a basis
of $\mathbb Z^2$, and
\[
 |v_n|_{h_t}^2
 =c_\vartheta^{-2}\bigl(t^{-2/3}\vartheta^{-2n}
                +t^{4/3}\vartheta^{2n+2}\bigr).
\]
For $0<t\le1$, choose the integer $n\ge0$ with
$\vartheta^n\le t^{-1/2}<\vartheta^{n+1}$.  The last formula, for
$n$ and $n+1$, gives $|v_n|_{h_t}+|v_{n+1}|_{h_t}\le C t^{1/6}$.
Every point of $\mathbb T^3$ has a representative in the centered
parallelotope generated by $(v_n,0),(v_{n+1},0),(0,0,1)$.
The last generator has length $t^{2/3}\le t^{1/6}$.
Translation invariance and the triangle inequality give the upper bound
in \eqref{eq:irrational-kasner-collapse}.

For the lower bound, the $(x,y)$ factor is a flat two-torus of area
$t^{1/3}$.  If its diameter is $D_t$, projection of the Euclidean disk
of radius $D_t$, measured in its lifted flat metric, covers the torus.
Since this projection is a local isometry, the area formula gives
$t^{1/3}\le\pi D_t^2$.  The full metric is the Riemannian product of
this torus and a circle, so its diameter is at least $D_t$.
Finally $\eta_1\wedge\eta_2=\dd x\wedge\dd y$, and the three scale
factors have product $t$, proving the volume assertion.

Theorem~\ref{thm:homogeneous-future}\textup{(i)} gives future causal
completeness, since the metric is expanding and Ricci-flat with flat
spatial slices.  The completeness obstruction \cite{GLS2018} excludes
a future boundary, proving the global assertion.
\end{proof}

The compact slices in this example converge to a point in diameter,
although no continuous spacetime extension exists.  The obstruction is
detected by fixed open sets in the cover, not by the ambient size of the
compact slice.  The irrational direction is essential to the displayed
diameter collapse: the lattice bases controlling the shortest
representatives vary with $t$.

The behavior for a general slope is determined in
Appendix~\ref{app:kasner-arithmetic}.  Every irrational slope has collapse
subsequences, while a Diophantine approximation constant determines whether
the diameter tends to zero, has a finite positive upper limit, or has
unbounded excursions.  Rational slopes instead have a divergent
power-law asymptotic.  In every case the local vacuum geometry and
continuous-inextendibility conclusion are unchanged.

\section{Timelike filling on two-step groups}
\label{sec:two-step}

Let $-\infty\le t_-<t_+\le\infty$.  Let $\mathfrak n$ be a finite-dimensional
nilpotent Lie algebra of step exactly two, let $Z=Z(\mathfrak n)$, and put
$V=\mathfrak n/Z$.  The bracket induces
\begin{equation}
 \mathsf b:\Lambda^2V\longrightarrow Z,
 \qquad \mathsf b(\overline X,\overline Y)=[X,Y].
 \label{eq:quotient-bracket}
\end{equation}
Choose a linear section $s:V\to\mathfrak n$.  The exponential map identifies
the simply connected group $\mathsf N$ with $V\oplus Z$, where
\begin{equation}
 (v,z)(v',z')=(v+v',z+z'+\tfrac12\mathsf b(v,v')),
 \qquad \Theta=\dd z-\tfrac12\mathsf b(v,\dd v).
 \label{eq:group-law}
\end{equation}

\begin{lemma}[Schur decomposition]
\label{lem:schur}
Every positive-definite inner product $h_t$ on $\mathfrak n$ has a unique
representation
\begin{equation}
 h_t=H_t(\dd v,\dd v)+
 A_t(\Theta+L_t\dd v,\Theta+L_t\dd v),
 \label{eq:schur-form}
\end{equation}
where $H_t$ and $A_t$ are positive-definite forms on $V$ and $Z$,
respectively, and $L_t:V\to Z$.  The forms $H_t$ and $A_t$ are independent
of the section.  If $s'=s+K$, then
$L_t'=L_t+K$ and $\Theta'=\Theta-K\dd v$, so the last term in
\eqref{eq:schur-form} is unchanged.
\end{lemma}

\begin{proof}
Set $A_t=h_t|_Z$ and define $L_t$ by
$A_t(L_tv,z)=h_t(sv,z)$.  Then
\[
 H_t(v,w)=h_t(sv,sw)-A_t(L_tv,L_tw).
\]
For $v\ne0$,
$H_t(v,v)=h_t(sv-L_tv,sv-L_tv)>0$.  Expanding the right-hand side of
\eqref{eq:schur-form} recovers all entries of $h_t$.  The stated
transformation rule follows by replacing $s$ with $s+K$, and proves section
independence.
\end{proof}

Consider the orthogonal invariant spacetime
\begin{equation}
 g=-N(t)^2\dd t^2+h_t
 \quad\text{on}\quad (t_-,t_+)\times\mathsf N,
 \qquad N>0,
 \label{eq:two-step-spacetime}
\end{equation}
with $\partial_t$ future directed.  A future timelike curve can be
reparametrized by $t$.  We assume $N$ and $t\mapsto h_t$ smooth throughout
the open interval.  To justify the reparametrization, write a curve on
$[0,1]$ as $\Gamma(r)$ with time coordinate $t(r)$ and endpoints $a,b$.
The maps
$\psi_\theta(r)=(1-\theta)r+\theta t^{-1}(a+(b-a)r)$ are increasing
piecewise $C^1$ diffeomorphisms with positive derivatives.
Then $\Gamma\circ\psi_\theta$ is a fixed-endpoint timelike homotopy to
the coordinate-time parametrization.  We shall work with curves of the form
$\Gamma(t)=(t,v(t),z(t))$.

Fix $T\in(t_-,t_+)$.  For $I=[a,b]\Subset(t_-,T)$ and
$\xi\in W^{1,2}_0(I;V)$, define its bracket area by
\begin{equation}
 Q_I(\xi)=\frac12\int_a^b\mathsf b(\xi,\dot\xi)\,\dd t.
 \label{eq:bracket-area}
\end{equation}
When $Q_I(\xi)\ne0$, put
\begin{equation}
 J_I(\xi)=\int_a^b
 \frac{H_t(\dot\xi,\dot\xi)}
 {N(t)\sqrt{A_t(Q_I(\xi),Q_I(\xi))}}\,\dd t,
 \qquad
 \kappa_{\rm an}(T)=
 \sup_{I\Subset(t_-,T)}\sup_{Q_I(\xi)\ne0}J_I(\xi)^{-1}.
 \label{eq:kappa-an}
\end{equation}
We use the convention $\sup\varnothing=0$.  The quotient metric $H_t$, the
central metric $A_t$, and $Q_I$ are intrinsic, so $\kappa_{\rm an}$ does not
depend on the section.  It is also unchanged by an orientation-preserving
time reparametrization.

A convenient sufficient bound is obtained from a fixed inner product on
$Z$.  Let $|\cdot|_*$ be its norm, set
\begin{equation}
 a_+(t)=\sup_{z\ne0}\frac{\sqrt{A_t(z,z)}}{|z|_*},
 \qquad
 \mathcal E_I(\xi)=\int_a^b
 \frac{H_t(\dot\xi,\dot\xi)}{N(t)a_+(t)}\,\dd t,
 \label{eq:computable-energy}
\end{equation}
and define
\begin{equation}
 \kappa_{\rm Sch}(T)=
 \sup_{I\Subset(t_-,T)}\sup_{0\ne\xi\in W^{1,2}_0(I;V)}
 \frac{|Q_I(\xi)|_*}{\mathcal E_I(\xi)}.
 \label{eq:kappa-schur}
\end{equation}
Since $\sqrt{A_t(Q,Q)}\le a_+(t)|Q|_*$,
\begin{equation}
 \kappa_{\rm an}(T)\le\kappa_{\rm Sch}(T).
 \label{eq:kappa-comparison}
\end{equation}

The constant one half enters through the following finite-dimensional
estimate.

\begin{lemma}[Pointwise timelike convexity]
\label{lem:convexity}
Let $U,W$ be finite-dimensional inner-product spaces.  Suppose
$(u_j,w_j)\in U\oplus W$, $j=0,1$, satisfy
$|u_j|^2+|w_j|^2<1$.  For $0\le\theta\le1$, set
\[
 \alpha=\theta(1-\theta),\quad
 \overline u=(1-\theta)u_0+\theta u_1,\quad
 \overline w=(1-\theta)w_0+\theta w_1,\quad
 e=|u_1-u_0|^2.
\]
If $|D|\le\kappa\alpha e$ with $0\le\kappa\le\tfrac12$, then
\begin{equation}
 |\overline u|^2+|\overline w+D|^2<1.
 \label{eq:convexity-conclusion}
\end{equation}
\end{lemma}

\begin{proof}
For $0<\theta<1$, the parallelogram identity gives
\[
 |\overline u|^2+|\overline w|^2
 <1-\alpha e-\alpha|w_1-w_0|^2.
\]
Put $r=|\overline w|$.  Then $r<1$ and $\alpha e<1-r^2$.  Therefore
\begin{align*}
 |\overline u|^2+|\overline w+D|^2
 &<1-\alpha|w_1-w_0|^2\\
 &\quad+\alpha e(-1+2\kappa r+\kappa^2\alpha e).
\end{align*}
For $\kappa\le1/2$, the last bracket is at most
\[
 -1+r+\tfrac14(1-r^2)
 =-\tfrac14(1-r)(3-r)<0.
\]
The endpoint cases have $D=0$ and are immediate.
\end{proof}

\subsection{Optimal allocation of the central correction}\label{subsec:optimal-allocation}
For a compact interval $I=[a,b]$ and a nonzero based loop
$\xi\in W^{1,2}_0(I;V)$, put
\[
 f_\xi(t)=\frac{H_t(\dot\xi,\dot\xi)}{N(t)},\qquad
 \kappa_{\rm opt}(T)=\sup_{I\Subset(t_-,T)}\sup_{\xi\ne0}k_I(\xi).
\]
\begin{equation}\label{eq:kappa-opt}
 k_I(\xi)=\sup_{0\ne\lambda\in Z^*}
 \frac{|\lambda(Q_I(\xi))|}
 {\displaystyle\int_I f_\xi(t)
          \sqrt{A_t^{-1}(\lambda,\lambda)}\,\dd t}.
\end{equation}
Here $A_t^{-1}$ is the dual quadratic form.  The denominator is positive
for every nonzero $\lambda$, since $f_\xi\ge0$ is not zero almost
everywhere.  These quantities are independent of a central basis, a
horizontal section, and an increasing time reparametrization.

\begin{lemma}[Allocation of a central correction]
\label{lem:central-allocation}
Let $A_t$ be a smooth positive metric on a finite-dimensional space $Z$
over a compact interval $I$, and let $f\in L^1(I)$ be nonnegative with
$\int_I f>0$.  For $q\in Z$ define
\[
 k(q;f,A)=\sup_{\lambda\ne0}
 \frac{|\lambda(q)|}{\int_I f(t)
                  \sqrt{A_t^{-1}(\lambda,\lambda)}\,\dd t}.
\]
For $k\ge0$ there exists a measurable $c:I\to Z$ satisfying
\[
 \int_I c(t)\,\dd t=q,\qquad |c(t)|_{A_t}\le kf(t)
 \quad\hbox{almost everywhere}
\]
if and only if $k\ge k(q;f,A)$.  Whenever the condition holds, one may
choose $c(t)=f(t)d(t)$ with $d$ smooth on $I$ and
$|d(t)|_{A_t}\le k$.  Thus, for each fixed triple $(q,f,A)$, the
infimum over admissible uniform bounds is attained; no attainment of the
later supremum over intervals and loops is asserted.
\end{lemma}

\begin{proof}
Necessity follows by pairing the integral with any covector and applying
the dual Cauchy-Schwarz inequality.  If $q=0$, take $c=0$.
Otherwise the condition forces $k>0$.  Fix auxiliary Euclidean
coordinates solely for this proof and define
\[
 h(\lambda)=\int_I f(t)
                    \sqrt{A_t^{-1}(\lambda,\lambda)}\,\dd t,
 \qquad
 K=\{x\in Z:\lambda(x)\le h(\lambda)\ \hbox{for all }\lambda\in Z^*\}.
\]
Uniform positive definiteness on $I$ implies
$a|\lambda|\le h(\lambda)\le b|\lambda|$ for constants $a,b>0$.
Thus $K$ is a compact, convex, centrally symmetric body with nonempty
interior.  The hypothesis says $y=q/k\in K$.

Uniform positive definiteness of $A_t$ makes the relevant difference
quotients uniformly dominated on the auxiliary unit sphere.  Hence, for
$\lambda\ne0$, differentiation under the integral gives
\[
 \nabla h(\lambda)=\int_I f(t)
 \frac{A_t^{-1}\lambda}{\sqrt{A_t^{-1}(\lambda,\lambda)}}\,\dd t.
\]
For every covector $\mu$, pointwise dual Cauchy-Schwarz gives
$\mu(\nabla h(\lambda))\le h(\mu)$, whereas
$\lambda(\nabla h(\lambda))=h(\lambda)$.  Consequently $h$ is the
support function of $K$.  If $y\in\partial K$, choose a nonzero
supporting covector $\lambda$.  Differentiability of $h$ implies that
its supporting face consists of the single point $\nabla h(\lambda)$:
indeed $\mu(y)\le [h(\lambda+s\mu)-h(\lambda)]/s$ for $s>0$,
and the inequality with $-\mu$ gives equality after $s\downarrow0$.
Thus
\[
 d(t)=k\frac{A_t^{-1}\lambda}
                 {\sqrt{A_t^{-1}(\lambda,\lambda)}}
\]
is smooth, has $A_t$-norm $k$, and satisfies $\int_I fd=q$.

If $y$ lies in the interior of $K$, consider instead
\[
 F(\lambda)=\int_I f(t)
                  \sqrt{1+A_t^{-1}(\lambda,\lambda)}\,\dd t.
\]
Because $y$ is interior, there is $\delta>0$ such that
$h(\lambda)-\lambda(y)\ge\delta|\lambda|$.  Since $F\ge h$,
the smooth function $F(\lambda)-\lambda(y)$ is coercive and attains a
minimum at some $\lambda_0$.  Its first derivative vanishes, so
\[
 y=\int_I f(t)
       \frac{A_t^{-1}\lambda_0}
            {\sqrt{1+A_t^{-1}(\lambda_0,\lambda_0)}}\,\dd t.
\]
Taking $d(t)$ to be $k$ times the integrand after the factor $f(t)$
gives a smooth function with $|d(t)|_{A_t}<k$ and $\int_I fd=q$.
This proves sufficiency, including the boundary and optimal cases.
\end{proof}

For $Q\ne0$ the direction-preserving choice
\[
 c(t)=\frac{f_\xi(t)Q}{J_I(\xi)\sqrt{A_t(Q,Q)}}
\]
has integral $Q$ and norm $|c(t)|_{A_t}=f_\xi(t)/J_I(\xi)$.
The minimality part of Lemma~\ref{lem:central-allocation} therefore
gives $k_I(\xi)\le J_I(\xi)^{-1}$.
Taking suprema gives
\begin{equation}\label{eq:optimal-anisotropic-comparison}
 \kappa_{\rm opt}(T)\le\kappa_{\rm an}(T)\le\kappa_{\rm Sch}(T).
\end{equation}
If $A_t=a(t)^2A_0$, duality of the fixed central norm instead gives
\[
 k_I(\xi)=\frac{|Q_I(\xi)|_{A_0}}
 {\int_I H_t(\dot\xi,\dot\xi)/(N(t)a(t))\,\dd t}
 =J_I(\xi)^{-1},
\]
with both expressions zero when $Q_I(\xi)=0$.  Thus
$\kappa_{\rm opt}=\kappa_{\rm an}$ for a conformal central metric.

\begin{proposition}[Timelike filling for arbitrary Schur metrics]
\label{thm:schur-filling}
If $\kappa_{\rm opt}(T)\le\tfrac12$, then
$(t_-,T)\times\mathsf N$ is future one-connected.  No additional bound on
$L_t$ or on the condition number of $A_t$ is required.  By
\eqref{eq:optimal-anisotropic-comparison}, either
$\kappa_{\rm an}(T)\le\tfrac12$ or
$\kappa_{\rm Sch}(T)\le\tfrac12$ also suffices.
\end{proposition}

\begin{proof}
Let $\Gamma_j(t)=(t,v_j(t),z_j(t))$, $j=0,1$, be two piecewise $C^1$
future timelike curves with common endpoints on $I=[a,b]\Subset(t_-,T)$.
Refine their decompositions to a common subdivision and put
$\xi=v_1-v_0\in W^{1,2}_0(I;V)$.  Define the central body controls
\begin{equation}
 \zeta_j=\dot z_j-\tfrac12\mathsf b(v_j,\dot v_j),
 \qquad \eta_j=\zeta_j+L_t\dot v_j.
 \label{eq:body-controls}
\end{equation}
Timelikeness is equivalent to
\begin{equation}
 \frac{H_t(\dot v_j,\dot v_j)}{N^2}
 +\frac{A_t(\eta_j,\eta_j)}{N^2}<1.
 \label{eq:timelike-controls}
\end{equation}

Let $Q=Q_I(\xi)$.  For $Q\ne0$, Lemma~\ref{lem:central-allocation}
gives a smooth $d:I\to Z$ with
$\int_I f_\xi d\,\dd t=Q$ and
$|d(t)|_{A_t}\le k_I(\xi)\le\kappa_{\rm opt}(T)$.  Set
\begin{equation}
 c_\theta(t)=\theta(1-\theta)N(t)e_\xi(t)d(t),\qquad
 e_\xi(t)=\frac{H_t(\dot\xi,\dot\xi)}{N(t)^2}.
 \label{eq:distributed-correction}
\end{equation}
If $Q=0$, set $c_\theta=0$.  Define
\begin{align}
 v_\theta&=(1-\theta)v_0+\theta v_1,
 \label{eq:v-theta}\\
 \eta_\theta&=(1-\theta)\eta_0+\theta\eta_1+c_\theta,
 \label{eq:eta-theta}\\
 z_\theta(t)&=z_0(a)+\int_a^t\left(
 \eta_\theta-L_r\dot v_\theta
 +\tfrac12\mathsf b(v_\theta,\dot v_\theta)\right)\dd r.
 \label{eq:z-theta}
\end{align}
The correction has integral $\theta(1-\theta)Q$.  On the other hand,
bilinearity and antisymmetry give
\begin{equation}
 \frac12\int_a^b\mathsf b(v_\theta,\dot v_\theta)\,\dd t
 -(1-\theta)\frac12\int_a^b\mathsf b(v_0,\dot v_0)\,\dd t
 -\theta\frac12\int_a^b\mathsf b(v_1,\dot v_1)\,\dd t
 =-\theta(1-\theta)Q.
 \label{eq:bracket-defect}
\end{equation}
The two terms cancel.  Since the $L_t$ term is affine in $\dot v_\theta$,
$z_\theta(b)$ is the common terminal coordinate.  The initial endpoint is
fixed by construction.

At a fixed time, use $H_t$ and $A_t$ as inner products and put
\[
 u_j=\dot v_j/N,\qquad w_j=\eta_j/N.
\]
The interpolating normalized controls are
\[
 u_\theta=(1-\theta)u_0+\theta u_1,\qquad
 w_\theta=(1-\theta)w_0+\theta w_1+D,\qquad D=c_\theta/N.
\]
If $Q\ne0$, the allocated correction gives
\begin{equation}
 |D|_{A_t}=\theta(1-\theta)e_\xi(t)|d(t)|_{A_t}
 \le\theta(1-\theta)\kappa_{\rm opt}(T)e_\xi(t),
 \label{eq:correction-bound}
\end{equation}
and the estimate is trivial for $Q=0$.  Since
$e_\xi=|u_1-u_0|_{H_t}^2$, Lemma~\ref{lem:convexity} proves that every
interpolating curve is future timelike.  The formulas depend continuously on
$(\theta,t)$ and respect one-sided tangents at subdivision points.  They
therefore define the required fixed-endpoint timelike homotopy.
Concatenating this coordinate-time homotopy with the two fixed-endpoint
reparametrization homotopies used above to make $t$ the common parameter
gives a homotopy of the original curves.
Only a homotopy for each fixed pair of curves is required; no continuous
choice of $d$ as the pair varies is assumed.
\end{proof}

\begin{corollary}[Tensorization of optimal filling and mixed quotients]
\label{cor:optimal-product}
For $1\le j\le r$, let $\mathsf N_j$ be a simply connected two-step
nilpotent group with invariant Schur data
$(H_{j,t},A_{j,t},L_{j,t})$ on a common interval and with the same lapse
$N$.  On $\mathsf N=\prod_{j=1}^r\mathsf N_j$ equip the spatial slices
with the orthogonal product metric $h_t=\bigoplus_{j=1}^r h_{j,t}$.
Then
\begin{equation}
 \kappa_{\rm opt}^{\mathsf N}(T)
 =\max_{1\le j\le r}\kappa_{\rm opt}^{(j)}(T).
 \label{eq:optimal-product}
\end{equation}
In particular, future one-connectedness is preserved under finite
orthogonal products with the same sharp sufficient threshold $1/2$.

Choose the product auxiliary metric and let $m_j(t)$ be the least scale
of $h_{j,t}$.  Suppose, for every $j$, that
\[
 \kappa_{\rm opt}^{(j)}(T)\le\frac12,
 \qquad
 \int_{t_-}^{T}\frac{N(t)}{m_j(t)}\,\dd t<\infty.
\]
If the quotient metric on the abelianization of at least one factor has
unbounded maximal eigenvalue as $t\to t_-$, then the product spacetime is
past $C^0$-inextendible on $\mathsf N$ and on
$\Gamma\backslash\mathsf N$ for every discrete subgroup
$\Gamma<\mathsf N$.  The subgroup need not split as a product or
preserve the factors.  Future timelike geodesic completeness promotes
the conclusion to global $C^0$-inextendibility.
\end{corollary}

\begin{proof}
Write $\xi=(\xi_1,\ldots,\xi_r)$,
$Q=(Q_1,\ldots,Q_r)$ and
$\lambda=(\lambda_1,\ldots,\lambda_r)$, and put
\[
 f_j=\frac{H_{j,t}(\dot\xi_j,\dot\xi_j)}{N},\qquad
 a_j=\sqrt{A_{j,t}^{-1}(\lambda_j,\lambda_j)},\qquad
 \kappa_*=\max_j\kappa_{\rm opt}^{(j)}(T).
\]
The defining dual inequality in each factor gives
\[
 |\lambda(Q)|
 \le\sum_j|\lambda_j(Q_j)|
 \le\kappa_*\int_I\sum_j f_ja_j\,\dd t
 \le\kappa_*\int_I
       \Big(\sum_jf_j\Big)\Big(\sum_ja_j^2\Big)^{1/2}\,\dd t.
\]
The last integral is the dual denominator for the product Schur metric,
so $\kappa_{\rm opt}^{\mathsf N}(T)\le\kappa_*$.  Loops and covectors
supported in one factor give the reverse inequality and prove
\eqref{eq:optimal-product}.

For the product auxiliary metric,
\[
 m_h(t)=\min_jm_j(t),\qquad
 \frac{N}{m_h}=\max_j\frac{N}{m_j}
 \le\sum_j\frac{N}{m_j}.
\]
Thus the product cone integral is finite.  Its abelianization metric is
the orthogonal direct sum of the factor metrics, so its maximal
eigenvalue is the maximum of their maximal eigenvalues.
Theorem~\ref{thm:two-step-obstruction} now applies to every discrete left
quotient, independently of whether $\Gamma$ respects the factors.
\end{proof}

\begin{remark}
The value $1/2$ is sharp for Lemma~\ref{lem:convexity}.  For unit vectors
$e_U,e_W$, take $u_0=-a e_U$, $u_1=a e_U$ and
$w_0=w_1=\sqrt{1-a^2-a^4}\,e_W$, with $a>0$ small.  At
$\theta=1/2$, the admissible correction $D=\kappa a^2 e_W$ yields
$|\overline w+D|=1+(\kappa-1/2)a^2+O(a^4)>1$ for any
$\kappa>1/2$ and sufficiently small $a$.
This does not assert that $1/2$ is necessary for
future one-connectedness of every particular spacetime; it is the optimal
constant for this pointwise construction.
\end{remark}

\section{Inextendibility on covers and compact quotients}
\label{sec:obstruction}

Fix auxiliary inner products on $V$ and $Z$ and let
\begin{equation}
 k=|\dd v|_V^2+|\Theta|_Z^2
 \label{eq:auxiliary-metric}
\end{equation}
be the corresponding complete left-invariant metric.  All eigenvalues,
determinants, and operator norms in this section refer to these fixed inner
products and the fixed section.  Define the intrinsic least scale
\begin{equation}
 m_h(t)^2=\min_{|X|_k=1}h_t(X,X).
 \label{eq:intrinsic-scale}
\end{equation}
The following section-dependent estimate is useful for explicit examples:
\begin{equation}
\begin{aligned}
 \underline\lambda(t)&=
 \min\{\lambda_{\min}(H_t),\lambda_{\min}(A_t)\},\\
 C_L(t)&=\max\{2,1+2\|L_t\|^2\},&
 m_{\rm Sch}(t)&=\sqrt{\underline\lambda(t)/C_L(t)}.
\end{aligned}
 \label{eq:m-schur}
\end{equation}

\begin{lemma}[Cone width and metric size]
\label{lem:schur-size}
The Schur metric satisfies
\begin{equation}
 h_t\ge m_h(t)^2k\ge m_{\rm Sch}(t)^2k.
 \label{eq:schur-lower-bound}
\end{equation}
For every nonempty open $U\Subset\mathsf N$, constants $c_U,C_U>0$
independent of $t$ satisfy
\begin{align}
 \diam_{h_t}(U;\mathsf N)&\ge
 c_U\sqrt{\lambda_{\max}(H_t)},
 \label{eq:horizontal-size}\\
 \Vol_{h_t}(U)&=C_U\sqrt{\det H_t\det A_t}.
 \label{eq:schur-volume}
\end{align}
\end{lemma}

\begin{proof}
For a body vector $(x,z)\in V\oplus Z$,
\[
 h_t(x,z)\ge\underline\lambda(t)
 \bigl(|x|_V^2+|z+L_tx|_Z^2\bigr).
\]
The elementary inequality
\[
 |x|_V^2+|z|_Z^2
 \le(1+2\|L_t\|^2)|x|_V^2+2|z+L_tx|_Z^2
\]
proves \eqref{eq:schur-lower-bound}.  The quotient map
$\mathsf N\to V=\mathsf N/Z$ is length decreasing from $h_t$ to $H_t$.
A fixed coordinate box inside $U$ projects onto a Euclidean ball in $V$;
two points separated in a maximal-eigenvalue direction give
\eqref{eq:horizontal-size}.  Finally, the coframe changes
$(\dd v,\dd z)\mapsto(\dd v,\Theta)$ and
$(\dd v,\Theta)\mapsto(\dd v,\Theta+L_t\dd v)$ are triangular with
determinant one.  The Gram determinant is therefore
$\det H_t\det A_t$, proving \eqref{eq:schur-volume}.
\end{proof}

\begin{proposition}[An equivalent Schur cone test]
\label{prop:trace-cone}
Relative to the fixed auxiliary inner products, put
\[
 D_h(t)=\operatorname{tr}(H_t^{-1})+
 \operatorname{tr}(A_t^{-1})+
 \operatorname{tr}(L_tH_t^{-1}L_t^*),\qquad d=\dim\mathfrak n.
\]
Then
\begin{equation}
 \sqrt{D_h(t)/d}\le m_h(t)^{-1}\le\sqrt{D_h(t)}.
 \label{eq:trace-cone}
\end{equation}
Consequently the intrinsic cone integral is finite if and only if
$\int_{t_-}^T N\sqrt{D_h}\,\dd t$ is finite.
\end{proposition}

\begin{proof}
The Gram matrix and its inverse in $V\oplus Z$ are
\[
 G_t=\begin{pmatrix}H_t+L_t^*A_tL_t&L_t^*A_t\\A_tL_t&A_t\end{pmatrix},
 \quad
 G_t^{-1}=\begin{pmatrix}
 H_t^{-1}&-H_t^{-1}L_t^*\\
 -L_tH_t^{-1}&A_t^{-1}+L_tH_t^{-1}L_t^*
 \end{pmatrix}.
\]
Thus $\operatorname{tr}(G_t^{-1})=D_h(t)$.  For a positive $d$ by $d$
matrix the largest eigenvalue lies between its trace divided by $d$ and
its trace.  Since $m_h^{-2}=\lambda_{\max}(G_t^{-1})$, this proves
\eqref{eq:trace-cone} and the equivalence of the integrals.
\end{proof}

The coupling is weighted here by the horizontal inverse metric, rather
than by the smaller of the horizontal and central eigenvalues.  This
distinction gives a strictly larger range of admissible unbounded
couplings in the power-law applications below.

Let
$\operatorname{Ann}[\mathfrak n,\mathfrak n]\subset\mathfrak n^*$ be the
annihilator of the commutator algebra.  Every
$\alpha$ in this space integrates on the simply connected group to a
homomorphism $\ell_\alpha:\mathsf N\to\R$.  If $\Gamma<\mathsf N$ is a
lattice, call $\alpha$ \emph{$\Gamma$-rational} if, after a nonzero
rescaling, $\ell_\alpha(\Gamma)\subset\mathbb Z$.  The corresponding
left-invariant closed one-form then descends with integral periods to
$\Gamma\backslash\mathsf N$.

\begin{proposition}[The metric on the abelianization]
\label{prop:abelianization-growth}
Let $\mathfrak a=\mathfrak n/[\mathfrak n,\mathfrak n]$, fix an auxiliary
inner product on $\mathfrak a$, and let $\overline h_t$ be the quotient
metric induced by $h_t$:
\[
 \overline h_t(\overline X,\overline X)
 =\inf_{Z\in[\mathfrak n,\mathfrak n]}h_t(X+Z,X+Z).
\]
Every nonempty open $U\Subset\mathsf N$ admits $c_U>0$, independent of $t$,
such that
\begin{equation}
 \diam_{h_t}(U;\mathsf N)
 \ge c_U\sqrt{\lambda_{\max}(\overline h_t)}.
 \label{eq:abelianization-diameter}
\end{equation}
Identifying $\mathfrak a^*$ with
$\operatorname{Ann}[\mathfrak n,\mathfrak n]$, one has
\begin{equation}
 \overline h_t^{-1}
 =h_t^{-1}\big|_{\operatorname{Ann}[\mathfrak n,\mathfrak n]}.
 \label{eq:abelianization-dual}
\end{equation}
Consequently
\begin{equation}
 \limsup_{t\to t_-}\lambda_{\max}(\overline h_t)=\infty
 \label{eq:abelianization-growth}
\end{equation}
is equivalent to the least eigenvalue of this restricted dual form tending
to zero along a sequence.  It follows either from
$\limsup\lambda_{\max}(H_t)=\infty$ or from
\begin{equation}
 h_{t_j}^{-1}(\alpha,\alpha)\longrightarrow0
 \quad\text{for some fixed }0\ne\alpha\in
 \operatorname{Ann}[\mathfrak n,\mathfrak n],\quad t_j\to t_-.
 \label{eq:fixed-closed-growth}
\end{equation}
\end{proposition}

\begin{proof}
The map $P:\mathsf N\to\mathfrak a$, $P(\exp X)=\overline X$, is a
surjective homomorphism and is length decreasing for these metrics.
The open set $P(U)$ contains a ball of fixed positive radius in the
auxiliary inner product.  Choose two points in that ball separated by
a fixed multiple of a unit maximal eigenvector of $\overline h_t$, and
choose preimages in $U$.  Projection of every connecting curve gives
\eqref{eq:abelianization-diameter}.

For fixed $t$, the lift to the $h_t$-orthogonal complement of
$[\mathfrak n,\mathfrak n]$ is an isometry for the quotient metric.
A covector annihilating the commutator has the same dual norm on that
complement, proving \eqref{eq:abelianization-dual}.  Reciprocity of
positive-matrix eigenvalues proves the equivalence.  Since $H_t$ is the
further quotient of $\overline h_t$ by $Z/[\mathfrak n,\mathfrak n]$,
$\lambda_{\max}(H_t)\le C\lambda_{\max}(\overline h_t)$ for a fixed
comparison constant.  Finally,
\[
 \lambda_{\min}(\overline h_t^{-1})
 \le h_t^{-1}(\alpha,\alpha)/|\alpha|^2
\]
proves the fixed-covector implication.
\end{proof}

\begin{theorem}[Continuous inextendibility on two-step quotients]
\label{thm:two-step-obstruction}
Suppose $\kappa_{\rm opt}(T)\le\tfrac12$ and
\begin{equation}
 \int_{t_-}^{T}\frac{N(t)}{m_h(t)}\,\dd t<\infty.
 \label{eq:schur-cone-integral}
\end{equation}
The spacetime on $\mathsf N$, and on $\Gamma\backslash\mathsf N$ for
every discrete subgroup $\Gamma<\mathsf N$, is past
$C^0$-inextendible if at least one of the following holds as $t\to t_-$:
\begin{enumerate}[label=\textup{(\roman*)}]
\item $\limsup\lambda_{\max}(\overline h_t)=\infty$;
\item $\limsup\sqrt{\det H_t\det A_t}=\infty$.
\end{enumerate}
The first condition includes horizontal growth and the fixed-covector
test \eqref{eq:fixed-closed-growth}.  No rationality condition relative
to the discrete subgroup is required.  Future timelike geodesic
completeness implies global $C^0$-inextendibility.  The stronger condition
$\int_{t_-}^T N/m_{\rm Sch}<\infty$ is sufficient for the cone hypothesis.
With a strict margin in the filling bound, the conclusion is stable under
the two-sided metric comparisons of Proposition~\ref{thm:stability}; it
also extends to invariant shifts under the hypotheses of
Proposition~\ref{prop:shift}.
\end{theorem}

\begin{proof}
An increasing time reparametrization reduces the left endpoint to zero.
The factors $N\,\dd t$, $H_t(\dot\xi,\dot\xi)\dd t/N$, and the bracket
area are invariant under this change; hence so are the cone and filling
conditions.  Equation \eqref{eq:schur-lower-bound} and the cone integral
give vanishing full causal spatial length.  Proposition~\ref{thm:schur-filling}
gives future one-connectedness on the lifted past tail.

In case \textup{(i)}, Proposition~\ref{prop:abelianization-growth} gives
unbounded diameter along a sequence approaching the past endpoint
for every fixed nonempty open set upstairs.  Corollary~\ref{cor:cover-transfer}
then excludes a past boundary on every discrete quotient, including the
trivial quotient; this is why no rational direction is required.  In case
\textup{(ii)}, equation \eqref{eq:schur-volume} gives unbounded volume along a sequence approaching the past endpoint
on each nonempty coordinate open set, both upstairs and downstairs.
Lemma~\ref{lem:quotient-tests}(a) and
Proposition~\ref{thm:quotient-obstruction} give the same conclusion.
The complete invariant auxiliary metric makes every product slice Cauchy,
so the spacetime is globally hyperbolic and time oriented.  Future
timelike completeness excludes a future boundary by the one-sided
obstruction of \cite{GLS2018}; for a non-time-orientable candidate
extension one uses the time-orientation double cover described above.
This yields global inextendibility.  The stability and shift assertions
are Propositions~\ref{thm:stability} and~\ref{prop:shift}.
\end{proof}

Growth in the full abelianization is strictly more general than horizontal
growth and the fixed-covector test when the center contains an abelian direct
factor. Here is an explicit example, requiring no field equation.
On $\mathsf N=\mathbb H_3\times\mathbb R^2$ use coordinates
$(x,y,z,u,v)$ and put
\[
 \Theta=\dd z-\tfrac12(x\dd y-y\dd x),\qquad
 \theta(t)=\frac1{\log(1/t)},\qquad
 \begin{pmatrix}\eta_1\cr\eta_2\end{pmatrix}
 =\begin{pmatrix}\cos\theta&\sin\theta\cr
 -\sin\theta&\cos\theta\end{pmatrix}
 \begin{pmatrix}\dd u\cr\dd v\end{pmatrix}.
\]
Fix $0<a<b<1$. For $0<t<T<e^{-2}$ consider
\begin{equation}
 g=-\dd t^2+\dd x^2+\dd y^2+\Theta^2
       +t^{-2a}\eta_1^2+t^{2b}\eta_2^2.
 \label{eq:rotating-abelian-example}
\end{equation}
Here $H_t$ is constant, $L_t=0$, and the volume density is
$t^{b-a}\to0$. Relative to
$k=\dd x^2+\dd y^2+\Theta^2+\dd u^2+\dd v^2$, the least spatial scale
is $m_h(t)=t^b$, so the cone integral is finite. Every bracket area lies
in the $z$-axis, where $A_t$ is constant. For $I=[r,s]$ and
$\xi\in W^{1,2}_0(I;\mathbb R^2)$, Cauchy-Schwarz and the Dirichlet
Poincar\'e inequality give
\[
 |Q_I(\xi)|\le\tfrac12\|\xi\|_{L^2}\|\dot\xi\|_{L^2}
 \le\frac{s-r}{2\pi}\int_I|\dot\xi|^2\,\dd t.
\]
Hence $\kappa_{\rm an}(T)\le T/(2\pi)<1/2$. Since
$\lambda_{\max}(\overline h_t)=t^{-2a}$,
Theorem~\ref{thm:two-step-obstruction} proves past
$C^0$-inextendibility on every discrete quotient.

Neither constant $H_t$ nor the vanishing volume density yields the
horizontal or volume growth test. Nor does any fixed closed invariant covector:
for $\alpha=A\dd x+B\dd y+C\dd u+D\dd v$,
\[
 h_t^{-1}(\alpha,\alpha)=A^2+B^2
 +t^{2a}(C\cos\theta+D\sin\theta)^2
 +t^{-2b}(-C\sin\theta+D\cos\theta)^2.
\]
If $(C,D)\ne(0,0)$, the last term tends to infinity: for $D\ne0$ its
parenthesis tends to $D$, and for $D=0$ it is asymptotic to
$-C/\log(1/t)$. Otherwise a nonzero $\alpha$ has constant positive norm.
Thus no fixed covector satisfies \eqref{eq:fixed-closed-growth}, even along
a sequence. The varying maximal direction in the full abelianization
detects the obstruction.

Quantitative persistence under two-sided metric comparison and the
removal of invariant shifts are proved in Appendix~\ref{sec:stability}.
The power-law applications below also include the comparison estimates
needed for the Einstein asymptotics.

\subsection{Optimal allocation beyond the area direction}

The optimal constant can satisfy the obstruction criterion when its
area-direction upper bound is infinite.  The following example proves
this while retaining both the cone and metric-growth hypotheses.

\begin{proposition}[Strict enlargement by optimal central allocation]
\label{prop:strict-optimal-allocation}
On $\mathsf N=H_3\times\mathbb R$, choose a horizontal basis
$e_1,e_2$ and a central basis $z_1,z_2$ with
$\mathsf b(e_1,e_2)=z_1$.  In the corresponding Schur decomposition set
\begin{equation}
 N=1,\qquad H_t=tI_2,\qquad L_t=0,\qquad
 A_t=\begin{pmatrix}
 1&t^{-5/2}\\ t^{-5/2}&t^{-5}+t^{-4}
 \end{pmatrix},\qquad 0<t<1.
 \label{eq:strict-optimal-metric}
\end{equation}
Then, for every $0<T\le1$,
\begin{equation}
 \kappa_{\rm an}(T)=\infty,\qquad
 \kappa_{\rm opt}(T)
 \le\frac{\sqrt T}{\pi}
       \left[1+4\log\!\left(\frac2{\sqrt T}\right)\right]
 \longrightarrow0\quad(T\downarrow0).
 \label{eq:strict-optimal-bound}
\end{equation}
The spacetime has a finite past cone integral and divergent invariant
volume density.  In particular it is past $C^0$-inextendible on every
discrete left quotient, although the condition
$\kappa_{\rm an}\le1/2$ fails on every past tail.
\end{proposition}

\begin{proof}
The matrix $A_t$ is positive definite, since its leading entry is one
and its determinant is $t^{-4}$.  Direct inversion gives
\begin{equation}
 A_t^{-1}=\begin{pmatrix}
 1+t^{-1}&-t^{3/2}\\ -t^{3/2}&t^4
 \end{pmatrix},\qquad
 \sqrt{\det H_t\det A_t}=t^{-1},\qquad
 D_h(t)=\frac3t+1+t^4.
 \label{eq:strict-optimal-cone-volume}
\end{equation}
For $t\le1$, Proposition~\ref{prop:trace-cone} therefore yields
\[
 \int_0^T\frac{\dd t}{m_h(t)}
 \le\int_0^T\sqrt{\frac3t+1+t^4}\,\dd t
 \le2\sqrt{5T}<\infty.
\]
Equation~\eqref{eq:strict-optimal-cone-volume} also gives the required
volume divergence.

We first record the elementary weighted planar area formula.  If
$w>0$ is smooth on $I=[a,b]$ and
$Q(\xi)=\frac12\int_I(\xi_1\dot\xi_2-\xi_2\dot\xi_1)\,\dd t$,
then
\begin{equation}
 \sup_{0\ne\xi\in W^{1,2}_0(I;\mathbb R^2)}
 \frac{|Q(\xi)|}{\int_Iw(t)|\dot\xi|^2\,\dd t}
 =\frac1{4\pi}\int_I\frac{\dd t}{w(t)}.
 \label{eq:weighted-planar-area}
\end{equation}
Indeed, the change of variable $u=\int_a^t w(r)^{-1}\,\dd r$
preserves area and transforms the denominator to
$\int_0^\ell|\xi'(u)|^2\,\dd u$, where
$\ell=\int_Iw^{-1}$.
Subtracting the mean of the closed loop does not change its area.
The periodic Wirtinger inequality and Cauchy-Schwarz then bound its
absolute area by $\ell(4\pi)^{-1}\int_0^\ell|\xi'|^2$.
A circle traversed once at constant speed, translated so that its
endpoints are zero, attains equality.

Every bracket area is $Q(\xi)z_1$, and $A_t(z_1,z_1)=1$.
Formula~\eqref{eq:weighted-planar-area} with $w=t$ gives
\[
 \sup_{\xi:Q(\xi)\ne0}J_I(\xi)^{-1}
 =\frac1{4\pi}\log\frac ba.
\]
Taking intervals with fixed $b<T$ and $a\downarrow0$ proves
$\kappa_{\rm an}(T)=\infty$.

For $\lambda=(\lambda_1,\lambda_2)\in Z^*$, the numerator in
\eqref{eq:kappa-opt} vanishes when $\lambda_1=0$.  Otherwise scale
$\lambda$ to $(1,\sigma)$.  The dual central norm is
\[
 A_t^{-1}((1,\sigma),(1,\sigma))
 =1+\bigl(t^{-1/2}-\sigma t^2\bigr)^2.
\]
Interchanging the suprema over covectors and loops, and applying
\eqref{eq:weighted-planar-area}, gives the exact expression
\begin{align}
 \kappa_{\rm opt}(T)
 &=\frac1{4\pi}\sup_{\sigma\in\mathbb R}
   \int_0^T\frac{\dd t}
    {t\sqrt{1+(t^{-1/2}-\sigma t^2)^2}}\notag\\
 &=\frac1{2\pi}\sup_{\sigma\in\mathbb R}
   \int_{T^{-1/2}}^\infty
   \frac{\dd x}{x\sqrt{1+(x-\sigma/x^4)^2}}.
 \label{eq:strict-optimal-exact}
\end{align}
The passage to $(0,T)$ uses monotone convergence along increasing
compact intervals; all the integrands are nonnegative.

Put $a=T^{-1/2}\ge1$ and denote the last integral by $I(a,\sigma)$.
If $\sigma\le0$, its integrand is at most $x^{-2}$, so
$I(a,\sigma)\le a^{-1}$.  Suppose $\sigma>0$ and write
$r=\sigma^{1/5}$.  For $x\ge2r$,
\[
 x-r^5/x^4\ge\frac{31}{32}x,\qquad
 \int_{\max(a,2r)}^\infty
 \frac{\dd x}{x\sqrt{1+(x-r^5/x^4)^2}}
 \le\frac{32}{31\max(a,2r)}.
\]
In particular this bounds the entire integral by $32/(31a)$ if
$r<a/2$.  For $r\ge a/2$, the portion with $a\le x\le r/2$,
when nonempty, is at most
\[
 \frac{32}{31r^5}\int_0^{r/2}x^3\,\dd x
 =\frac1{62r}\le\frac1{31a}.
\]
On $[r/2,2r]$, the function $F(x)=x-r^5/x^4$ satisfies
$F(r)=0$ and $F'(x)\ge1$, so $|F(x)|\ge|x-r|$.
The middle portion is thus at most
\[
 \frac2r\int_{r/2}^{2r}\frac{\dd x}{\sqrt{1+(x-r)^2}}
 =\frac2r\bigl(\operatorname{arsinh}(r/2)
                         +\operatorname{arsinh}(r)\bigr)
 \le\frac4r\operatorname{arsinh}(r).
\]
The function $\operatorname{arsinh}(r)/r$ decreases for $r>0$:
its derivative has the sign of
$r/\sqrt{1+r^2}-\operatorname{arsinh}(r)<0$.
Since $r\ge a/2$, this last bound is at most
$8\operatorname{arsinh}(a/2)/a$.  Combining the three portions gives,
uniformly in $\sigma$,
\[
 I(a,\sigma)
 \le\frac1a\left[\frac{33}{31}
                   +8\operatorname{arsinh}(a/2)\right]
 \le\frac1a\bigl[2+8\log(2a)\bigr],
\]
where $\operatorname{arsinh}(a/2)\le\log(2a)$ for $a\ge1$.
Together with \eqref{eq:strict-optimal-exact}, this proves
\eqref{eq:strict-optimal-bound}.
For $T\le1/64$ the right side is at most
$(1+4\log16)/(8\pi)<1/2$; its monotonicity in $T$ on that interval
follows by differentiation.  Proposition~\ref{thm:schur-filling}
and Theorem~\ref{thm:two-step-obstruction} therefore apply.
\end{proof}

\begin{remark}[Field-equation scope of the example]
The preceding example separates two geometric hypotheses of the
inextendibility criterion.  No matter model is prescribed.
It violates the timelike convergence condition: for the unit normal
$\partial_t$ and the second fundamental form $K=\frac12\dot h_t$,
\[
 \theta=\operatorname{tr}_{h_t}K
        =\frac{\dd}{\dd t}\log\sqrt{\det h_t}=-\frac1t,
 \qquad
 \operatorname{Ric}_g(\partial_t,\partial_t)
       =-\dot\theta-|K|_{h_t}^2
       =-\frac1{t^2}-|K|_{h_t}^2<0.
\]
Thus it cannot be a vacuum or a minimally coupled massless-scalar
solution with zero cosmological constant.  Its purpose is to show that
optimal central allocation enlarges the geometric theorem even when
all the other obstruction hypotheses are retained.
\end{remark}

\section{Massless Bianchi II cosmologies}
\label{sec:einstein-scalar}

On $H_3$ use the invariant coframe
\begin{equation}
 \sigma^1=\dd z-x\,\dd y,\qquad
 \sigma^2=\dd x,\qquad \sigma^3=\dd y,\qquad
 \dd\sigma^1=-\sigma^2\wedge\sigma^3.
 \label{eq:heisenberg-coframe}
\end{equation}
The massless normalization is
\begin{equation}
 \Ric(g)=\dd\phi\otimes\dd\phi,\qquad \Box_g\phi=0,
 \label{eq:einstein-scalar}
\end{equation}
following from $\frac12\int(R-|\dd\phi|_g^2)\,\dd\Vol_g$.
A homogeneous scalar has stiff-fluid stress tensor
$\rho=p=\dot\phi^2/2$ in proper time.  The vacuum family is
Taub's solution \cite{Taub1951}; Bianchi II integration and
automorphism methods are discussed in
\cite{ChristodoulakisEtAl2001}.  We retain its customary parameters
to describe the horizon and scalar regimes explicitly.

Choose $A,B,C,N_0,k,\beta,\gamma>0$ and $\sigma\in\R$ with
\begin{equation}
 k=\frac{N_0A}{BC},\qquad
 \sigma^2=2\beta\gamma-\frac{k^2}{2}\ge0.
 \label{eq:scalar-parameters}
\end{equation}
On $\R_\tau\times H_3$ set
\begin{align}
 g&=-N^2\dd\tau^2+a^2(\sigma^1)^2+
 b^2(\sigma^2)^2+c^2(\sigma^3)^2,\notag\\
 a&=A\cosh(k\tau)^{-1/2},&
 b&=B e^{\beta\tau}\cosh(k\tau)^{1/2},\notag\\
 c&=C e^{\gamma\tau}\cosh(k\tau)^{1/2},&
 N&=N_0e^{(\beta+\gamma)\tau}\cosh(k\tau)^{1/2},\notag\\
 \phi&=\sigma\tau+\phi_0.
 \label{eq:scalar-family}
\end{align}
These fields descend to every discrete left quotient.

\begin{proposition}[Massless scalar parameters and continuous inextendibility]
\label{prop:einstein-scalar}
\label{prop:scalar-parameter-regions}
The fields~\eqref{eq:scalar-family} solve~\eqref{eq:einstein-scalar}
and are future timelike and null geodesically complete.
Put
\begin{equation}
 \delta=\frac{\beta-\gamma}{k},\qquad
 \nu=\frac{\sigma}{k},\qquad
 u=\frac{\beta+\gamma}{k}=\sqrt{1+\delta^2+2\nu^2}.
 \label{eq:scalar-dimensionless}
\end{equation}
The past endpoint has finite proper time.  Choosing $t=0$ there,
\begin{equation}
 a_i(t)=c_i t^{p_i}(1+O(t^{4p_1})),\qquad
 \phi(t)=q_\phi\log t+\phi_*+O(t^{4p_1}),
 \label{eq:scalar-kasner-asymptotics}
\end{equation}
with the same errors after two applications of $t\partial_t$, where
\begin{equation}
 (p_1,p_2,p_3)
 =\frac{(1,u+\delta-1,u-\delta-1)}{2u-1},\qquad
 q_\phi=\frac{2\nu}{2u-1},
 \label{eq:scalar-kasner-exponents}
\end{equation}
and $c_i>0$.  In particular,
\begin{equation}
 \sum_i p_i=1,\qquad \sum_i p_i^2+q_\phi^2=1.
 \label{eq:scalar-kasner-relations}
\end{equation}
The condition
\begin{equation}
 \min\{\beta,\gamma\}<k/2
 \label{eq:scalar-expanding}
\end{equation}
is equivalent to
\begin{equation}
 |\delta|>\nu^2.
 \label{eq:scalar-parameter-region}
\end{equation}
Under these equivalent conditions, exactly one horizontal exponent is
negative, and every
discrete left quotient is globally $C^0$-inextendible, without
conditions on a candidate extension.  If $|\delta|<\nu^2$, all
three exponents are positive; if $|\delta|=\nu^2>0$, precisely one
horizontal exponent vanishes.  At $(\delta,\nu)=(0,0)$ the
exponents are $(1,0,0)$, the symmetric vacuum horizon datum.
\end{proposition}

\begin{proof}
We give the exact conversion to the general Heisenberg family,
whose field equations and completeness are proved independently
in Propositions~\ref{prop:h5-exact-massless} and
\ref{prop:heisenberg-scattering}.
Set $\alpha=\beta+\gamma-k/2>0$, $d=N_0/(\sqrt2\alpha)$, and
$\widehat\tau=\alpha\tau+\log d$.  Define
\[
 (p_1,p_2,p_3)=\frac{(k/2,\beta-k/2,\gamma-k/2)}{\alpha},
 \qquad q_\phi=\frac\sigma\alpha,\qquad \chi=2p_1=\frac k\alpha.
\]
The parameter constraint gives the Kasner relations and
$\chi>0$.  With
\[
 (\widehat A_1,\widehat A_2,\widehat A_3)
 =(\sqrt2 A d^{-p_1},\,B d^{-p_2}/\sqrt2,
                                      \,C d^{-p_3}/\sqrt2),
 \qquad \phi_* =\phi_0-q_\phi\log d,
\]
one has
\[
 \frac{\widehat A_1}{\widehat A_2\widehat A_3}
   =2\chi d^{-\chi},\qquad
 F=1+e^{2k\tau}=1+d^{-2\chi}e^{2\chi\widehat\tau},
 \qquad \frac N\alpha=e^{\widehat\tau}F^{1/2}.
\]
The identity
$\cosh(k\tau)=\tfrac12e^{-k\tau}(1+e^{2k\tau})$
now identifies~\eqref{eq:scalar-family} exactly with
\eqref{eq:h5-exact-family} for $m=1$, relabeling its central
index $0$ as $1$.  Its proved conclusions give the field equations,
proper-time range and differentiated asymptotics, with
$2\chi=4p_1$.  The formula
$(\beta+\gamma)^2=(\beta-\gamma)^2+k^2+2\sigma^2$
gives~\eqref{eq:scalar-dimensionless} and hence the displayed
exponents.  Finally
$2\min\{\beta,\gamma\}/k=u-|\delta|$; squaring
$u<1+|\delta|$ gives precisely $\nu^2<|\delta|$.
The sign of $u-|\delta|-1$ gives all remaining sign statements.
The negative-horizontal-exponent conclusion follows from the
past obstruction in Proposition~\ref{prop:h5-exact-massless}
and future completeness.  The origin is the exceptional vacuum
corner of Proposition~\ref{prop:heisenberg-vacuum-dichotomy}.
\end{proof}

\begin{corollary}[The Taub vacuum specialization]
\label{cor:taub-vacuum-specialization}
Set $\sigma=0$ in~\eqref{eq:scalar-family}, so that
$\beta\gamma=k^2/4$.  Every discrete left quotient of this vacuum
spacetime is future timelike and null geodesically complete.
If $\beta\ne\gamma$, it is globally $C^0$-inextendible.
If $\beta=\gamma=k/2$, its past end admits the analytic horizon
extensions constructed in
Proposition~\ref{prop:heisenberg-vacuum-dichotomy}.
\end{corollary}

\begin{proof}
The vacuum equations and future completeness follow from
Proposition~\ref{prop:einstein-scalar}.  Since $\beta,\gamma>0$ and
$\beta\gamma=k^2/4$, unequal $\beta,\gamma$ lie on opposite sides of
$k/2$.  Thus~\eqref{eq:scalar-expanding} holds.  Equivalently, the past
Kasner exponents have $p_1>0$, exactly one negative horizontal exponent,
and $\max_i p_i<1$.  Their sum is one, so
$1+p_1-p_2-p_3=2p_1>0$.  The differentiated asymptotics in
\eqref{eq:scalar-kasner-asymptotics} imply the metric comparison in
Proposition~\ref{prop:asymptotic-transfer}, which excludes a past
$C^0$ extension on the cover and every discrete quotient.  Future
timelike completeness excludes a future boundary by~\cite{GLS2018}.
In the equal-parameter case the past exponents are $(1,0,0)$;
Proposition~\ref{prop:heisenberg-vacuum-dichotomy}, with $m=1$, supplies
the stated analytic extensions directly.
\end{proof}

For example, $k=1$, $\beta=1/4$, $\gamma=2$,
$\sigma=1/\sqrt2$ and $N_0A=BC$ give
$p=(2/7,-1/7,6/7)$ and $q_\phi=4/(7\sqrt2)$.
Every discrete quotient is globally $C^0$-inextendible.

Figure~\ref{fig:scalar-parameter-range} displays the distinction
between the negative-exponent region and the remaining scalar
sector.  Away from the origin, all $p_i<1$ and
$\kappa_{\rm an}(t)=O(t^{2/(2u-1)})\to0$ by the one-bracket
power estimate.  Thus the missing hypothesis for
$|\delta|\le\nu^2$ is spatial metric-size growth.  Indeed,
$abc=(ABC/N_0)N$ and
\begin{equation}
 abc\sim\frac{ABC\alpha}{N_0}\,t,
 \qquad \alpha=k(u-1/2).
 \label{eq:scalar-volume-collapse}
\end{equation}
The volume alternative supplies no additional cases.  Neither
the all-positive sector nor its nonvacuum boundary is decided
by this criterion; curvature divergence alone does not settle
continuous metric extendibility.

The scalar-energy limitation is explicit.  If $p_j<0$, then
\begin{equation}
 1-q_\phi^2=\sum_i p_i^2
 \ge p_j^2+\frac{(1-p_j)^2}{2}>\frac12,
 \qquad q_\phi^2<\frac12.
 \label{eq:negative-exponent-scalar-bound}
\end{equation}
For $H_{\rm av}=\frac13\sum_i\dot a_i/a_i$ and
$\Omega_\phi=\rho_\phi/(3H_{\rm av}^2)$, in units $8\pi G=1$,
\begin{equation}
 H_{\rm av}=\frac{1+o(1)}{3t},\qquad
 \rho_\phi=\frac{q_\phi^2+o(1)}{2t^2},\qquad
 \lim_{t\downarrow0}\Omega_\phi=\frac32q_\phi^2.
 \label{eq:scalar-density-limit}
\end{equation}
This also holds for the differentiated asymptotics with a
subleading potential below.  The obstruction therefore requires
$\lim\Omega_\phi<3/4$, whereas a limit above $3/4$ forces all
exponents to be positive.  A limit below $3/4$ does not conversely
force a negative exponent.

\begin{figure}[tb]
\centering
\begin{tikzpicture}[x=2.05cm,y=1.55cm,font=\small]
 \begin{scope}
  \clip (-2,-1.5) rectangle (2,1.5);
  \fill[black!12] (-2,-1.41421356237) -- (-2,1.41421356237)
   -- plot[domain=1.41421356237:-1.41421356237,samples=101]
       ({-\x*\x},\x) -- cycle;
  \fill[black!12] (2,-1.41421356237) -- (2,1.41421356237)
   -- plot[domain=1.41421356237:-1.41421356237,samples=101]
       ({\x*\x},\x) -- cycle;
  \draw[thick] plot[domain=-1.5:1.5,samples=100] ({\x*\x},\x);
  \draw[thick] plot[domain=-1.5:1.5,samples=100] ({-\x*\x},\x);
 \end{scope}
 \draw[black!45] (-2,-1.5) rectangle (2,1.5);
 \draw[->,black!55] (-2.08,0) -- (2.22,0) node[right,black] {$\delta$};
 \draw[->,black!55] (0,-1.56) -- (0,1.75) node[above,black] {$\nu$};
 \node[fill=white,inner sep=2pt] at (0,1.03) {unresolved};
 \node[fill=white,inner sep=2pt] at (0,-1.03) {unresolved};
 \node at (-1.38,0.31) {inextendible};
 \node at (1.38,0.31) {inextendible};
 \fill (0,0) circle[radius=1.6pt];
 \draw[->] (2.04,-0.89) -- (0.045,-0.035);
 \node[anchor=west] at (2.12,-0.97) {Taub horizon};
 \foreach \x in {-2,-1,1,2}
  {\draw (\x,0.025) -- (\x,-0.025);
   \node[below,inner sep=2pt] at (\x,-0.025) {$\x$};}
\end{tikzpicture}
\caption{Parameter range of the exact massless Bianchi II family.
The shaded regions $|\delta|>\nu^2$ have globally $C^0$-inextendible
developments.  The origin is the symmetric vacuum solution with analytic
Taub horizons.  The unshaded regions and the curves
$|\delta|=\nu^2$ away from the origin are unresolved by this criterion;
the curves do not represent a proved extendibility transition.}
\label{fig:scalar-parameter-range}
\end{figure}
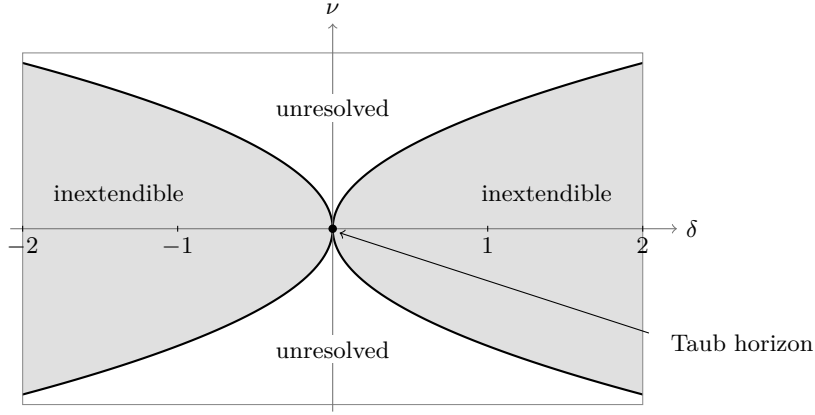

\section{Einstein developments on Heisenberg groups}
\label{sec:scalar-potentials}

\subsection{Scalar potentials and prescribed past asymptotics}

In three spatial dimensions the normalization is
\begin{equation}
 \Ric(g)=\dd\phi\otimes\dd\phi+\mathcal V(\phi)g,
 \qquad \Box_g\phi=\mathcal V'(\phi),
 \label{eq:potential-equations}
\end{equation}
from $\frac12\int(R-|\dd\phi|_g^2-2\mathcal V(\phi))\,\dd\Vol_g$.
Potential-dependent Kasner regimes are treated in
\cite{Ritchie2022,Ringstrom2025}.  The following specializations of
the higher-dimensional construction retain the leading amplitudes
and the differentiated estimates needed for continuous
inextendibility.

\begin{proposition}[Einstein-scalar solutions with a potential]
\label{prop:potential-solutions}
Let $\mathcal V\in C^\infty(\R)$ satisfy
\begin{equation}
 |\mathcal V(s)|+|\mathcal V'(s)|+|\mathcal V''(s)|
 \le C_Ve^{c|s|}\quad(s\in\R),\qquad C_V>0,
 \label{eq:potential-growth}
\end{equation}
with $c\ge0$.  Fix $A_i>0$, $\phi_0\in\R$ and data
\begin{equation}
 \sum_i p_i=1,\qquad \sum_i p_i^2+q^2=1,
 \qquad p_1>0,\qquad c|q|<2.
 \label{eq:potential-data}
\end{equation}
There is a smooth diagonal Bianchi II solution on $(0,t_*)$ with
\begin{equation}
 g=-\dd t^2+\sum_i a_i(t)^2(\sigma^i)^2,\qquad
 a_i=A_it^{p_i}(1+O(t^\mu)),\qquad
 \phi=q\log t+\phi_0+O(t^\mu),
 \label{eq:potential-asymptotics}
\end{equation}
where $\mu=\min\{4p_1,2-c|q|\}>0$, including the same errors
after two applications of $t\partial_t$.  If a horizontal
exponent is negative, the past end is $C^0$-inextendible on
every discrete left quotient, also in any smooth development
containing this tail.  If moreover $\mathcal V\ge0$, the maximal
homogeneous development is globally $C^0$-inextendible and future
causally geodesically complete.  Its scalar stress tensor obeys
the weak and dominant energy conditions.
\end{proposition}

\begin{proof}
Apply Proposition~\ref{prop:h5-einstein} with $m=1$, replacing
its indices $(0,1,2)$ by $(1,2,3)$.  The Kasner sum gives
$\chi_1=1+p_1-p_2-p_3=2p_1$, so all its hypotheses and the
stated error exponent agree exactly.  That proof verifies the
complete Einstein system, differentiated estimates, quotient
obstruction and future conclusion.  Finally
$\rho=\dot\phi^2/2+\mathcal V$ and
$p=\dot\phi^2/2-\mathcal V$ satisfy $\rho\ge|p|$ when
$\mathcal V\ge0$, proving the energy conditions.
\end{proof}

\begin{corollary}[Polynomial potentials and a cosmological constant]
\label{cor:polynomial-potential}
Every nonnegative polynomial potential admits the preceding
globally $C^0$-inextendible, future causally complete developments
for every generalized Kasner datum with $p_1>0$ and a negative
horizontal exponent, and every choice of leading amplitudes.
A past $C^0$-inextendible end also exists for every constant
potential $\mathcal V\equiv\Lambda\in\R$ with those data.
When $q=0$, the latter solutions are vacuum solutions of
$\Ric(g)=\Lambda g$.
\end{corollary}

\begin{proof}
A polynomial and its first two derivatives satisfy
\eqref{eq:potential-growth} with any $c>0$; choose $c|q|<2$.
A constant potential allows $c=0$.  For constant potential and
$q=0$, the harmonic scalar equation $\phi''=0$ and the past
data give $\phi\equiv\phi_0$.
\end{proof}

For example, $p=(2/7,-1/7,6/7)$ and $q=4/(7\sqrt2)$ admit
$\mathcal V=\Lambda+M^2\phi^2/2$ for all $\Lambda,M^2\ge0$.
Negative constant potentials retain the past conclusion.

\begin{corollary}[Finite sums of exponential potentials]
\label{cor:exponential-potential-sector}
Let $\mathcal V(\phi)=\sum_{j=1}^{m}V_j e^{\lambda_j\phi}$,
where $m\ge1$, $V_j>0$ and $\lambda_j\in\R$.  Fix $A_i>0$,
$\phi_0\in\R$ and generalized Kasner data with $p_1>0$.
There is a diagonal Bianchi II solution with
$a_i=A_it^{p_i}(1+o(1))$, $\phi=q\log t+\phi_0+o(1)$ and
$t\dot a_i/a_i\to p_i$, $t\dot\phi\to q$ if and only if
\begin{equation}
 2+\lambda_jq>0\qquad(1\le j\le m).
 \label{eq:one-sided-potential-condition}
\end{equation}
When these hold, the solution satisfies
\eqref{eq:potential-asymptotics} with
$\mu=\min\{4p_1,2+\lambda_1q,\ldots,2+\lambda_mq\}>0$
through two applications of $t\partial_t$; in particular,
\begin{equation}
 t\frac{\dot a_i}{a_i}=p_i+O(t^\mu),\qquad
 t\dot\phi=q+O(t^\mu).
 \label{eq:exponential-differentiated-asymptotics}
\end{equation}
If a horizontal exponent is negative, its maximal homogeneous
development is globally $C^0$-inextendible and future causally
geodesically complete on every discrete left quotient.
\end{corollary}

\begin{proof}
This is Proposition~\ref{prop:heisenberg-exponential-admissibility}
with $n=3$ and $\chi_1=2p_1$.  Positivity of every coefficient
is exactly its necessity hypothesis; the same specialization
gives the differentiated error and global conclusion.
\end{proof}

For one exponential the condition is $\lambda q>-2$, with no
upper restriction: large slopes are permitted when the potential
decays along the past scalar trajectory.  This agrees with the
potential-mode eigenvalue in \cite[equation~(3.12)]{BillyardEtAl1999}.
Equality excludes the stated differentiated Kasner regime, not
other leading balances.  The Bianchi II equilibria in
\cite[Section~III.A, item~3]{BillyardEtAl1999} give classical
examples with a leading-order exponential potential.

\subsection{Existing homogeneous developments}

\begin{corollary}[Nonlinear scalar fields]\label{cor:nonlinear-scalar-existing}
Let $\mathcal V\in C^\infty(\R)$ be nonnegative and suppose that, for some
$C<\infty$ and $0<\alpha<1$,
\begin{equation}\label{eq:potential-ringstrom-growth}
 |\mathcal V(s)|+|\mathcal V'(s)|
 \le C\exp(\sqrt6\,\alpha|s|),\qquad s\in\R.
\end{equation}
Consider the maximal homogeneous proper-time development of expanding,
Bianchi~II initial data without local rotational symmetry for
\eqref{eq:potential-equations}.  The scalar field is spatially homogeneous,
and expanding means positive mean curvature at the initial slice.
If the limiting
expansion-normalized Weingarten map has a negative eigenvalue, then the
development is globally $C^0$-inextendible and future timelike and null
geodesically complete, both on the simply connected Heisenberg group
and on every lattice quotient.
\end{corollary}

\begin{proof}
We give the conversion from the normalized asymptotics in
\cite[Lemma~4.8]{Ringstrom2025} to the spatial metric estimates needed
here.  The momentum constraint gives a common orthonormal eigenframe for the
Weingarten endomorphism and the symmetric class-A structure matrix
\cite[Lemma~2.4]{Ringstrom2025}.  Their diagonal form is preserved by
the homogeneous evolution; integrating the frame coefficients gives
a fixed diagonal invariant coframe $\eta^i$, with
$\dd\eta^1=-c\eta^2\wedge\eta^3$, $c\ne0$, and
$\dd\eta^2=\dd\eta^3=0$
\cite[proof of Proposition~1.31, equation~(3.13)]{Ringstrom2025}.
The bound \eqref{eq:potential-ringstrom-growth} is precisely
$\mathcal V\in\mathfrak P^1_\alpha$ in
\cite[Definition~1.44]{Ringstrom2025}, which is the derivative order
required in Lemma~4.8 of that reference.  Here and in the subsequent
applications of its results we use the numbering of version~1.
Write $h_t=\sum_i a_i(t)^2(\eta^i)^2$ in proper time and put
$\theta=\sum_i\dot a_i/a_i$.  Introduce $\tau$ by
$\dd\tau/\dd t=\theta/3$ and define
$\ell_i=\dot a_i/(\theta a_i)$.
The cited asymptotic result gives constants $\mu>0$, $\theta_\infty>0$
and $\sigma_\pm$ such that, as $\tau\to-\infty$,
\begin{equation}\label{eq:normalized-to-metric}
 \theta=\theta_\infty e^{-3\tau}(1+O(e^{\mu\tau})),
 \qquad \ell_i=p_i+O(e^{\mu\tau}),
\end{equation}
where
\begin{equation}\label{eq:ringstrom-exponent-identification}
 p_1=\frac{1-2\sigma_+}{3},\qquad
 p_2=\frac{1+\sigma_++\sqrt3\sigma_-}{3},\qquad
 p_3=\frac{1+\sigma_+-\sqrt3\sigma_-}{3}.
\end{equation}
Moreover, $\sigma_+<1/2$, and the asymptotic constraint reads
$\sum_i p_i=1$, $\sum_i p_i^2+q_\phi^2=1$, where
$q_\phi=\lim\dot\phi/\theta$.  In particular, $p_1>0$.
The negative eigenvalue is therefore horizontal, and the constraint
implies $\max_i p_i<1$.

Fix the additive constant in proper time by placing the past endpoint
at zero.  Integrating $\dd t/\dd\tau=3/\theta$ and
$\dd\log a_i/\dd\tau=3\ell_i$ gives
\begin{equation}\label{eq:proper-time-ringstrom}
 t=\theta_\infty^{-1}e^{3\tau}(1+O(e^{\mu\tau})),\qquad
 a_i=A_i t^{p_i}(1+O(t^{\mu/3})),\quad A_i>0.
\end{equation}
Consequently $h_t$ has two-sided relative error $O(t^{\mu/3})$ with
respect to $\sum_i A_i^2t^{2p_i}(\eta^i)^2$.  A fixed rescaling
normalizes the nonzero bracket coefficient, retaining fixed positive
metric coefficients.  Proposition~\ref{prop:asymptotic-transfer} now
applies.  A fixed Lie-group isomorphism transports any lattice to a
lattice, so no additional arithmetic hypothesis is needed.
Proposition~\ref{prop:nonnegative-future-completeness} gives future
causal geodesic completeness.  Since the product slices are Cauchy,
\cite{GLS2018} excludes a future extension boundary, while the preceding
argument excludes the past boundary.  This proves global
$C^0$-inextendibility.
\end{proof}

\begin{remark}\label{rem:nonlinear-open-sectors}
Suppose, in addition, that for every integer $m\ge0$ there is $C_m<\infty$
such that
$\sum_{j=0}^m|\mathcal V^{(j)}(s)|\le C_m e^{\sqrt6\alpha|s|}$.
Then the developments in Corollary~\ref{cor:nonlinear-scalar-existing}
contain a nonempty open subset of the non-LRS homogeneous
development classes: the asymptotic-data map is a diffeomorphism by
\cite[Corollary~1.88]{Ringstrom2025}, and a negative horizontal eigenvalue
is an open condition.  Nonemptiness follows, for example, from the
admissible singularity data
\[
 (p_1,p_2,p_3)=\left(\frac12,\frac34,-\frac14\right),
 \qquad q_\phi=\frac{1}{2\sqrt2}.
\]
At any fixed mean-curvature level met by one of these developments,
their regular initial-data classes form a nonempty open subset of
the corresponding homogeneous data space
\cite[Proposition~1.70]{Ringstrom2025}.  This is a finite-dimensional
homogeneous statement, not stability under inhomogeneous perturbations.
Examples of admissible potentials are nonnegative polynomials and
$\mathcal V(s)=V_0e^{\lambda s}$ with $V_0\ge0$ and $|\lambda|<\sqrt6$.
\end{remark}

\subsection{Einstein-scalar developments on Heisenberg groups}
\label{subsec:h5-einstein}

Let $n=2m+1\ge3$ and let $H_n$ have left-invariant coframe
$\omega^0,\ldots,\omega^{2m}$ with
\begin{equation}
 \dd\omega^0=-\sum_{r=1}^m\omega^{2r-1}\wedge\omega^{2r},
 \qquad \dd\omega^i=0\quad(1\le i\le2m).
 \label{eq:h5-coframe}
\end{equation}
Its center and derived algebra both equal $\R e_0$, so its Lie algebra
cannot split as a direct sum of two nonzero nilpotent Lie algebras.
The construction below retains all $m$ brackets.  Generalized Kasner
asymptotics for higher-dimensional Einstein-matter systems have a much
broader Fuchsian theory \cite{DamourHenneauxRendallWeaver2002}; here the
explicit curvature terms also yield a complete massless evolution and
metric inextendibility on every discrete quotient.
The dimensionally trace-reversed field equations are
\begin{equation}
 \Ric(g)=\dd\phi\otimes\dd\phi+
       \frac{2}{n-1}\mathcal V(\phi)g,
 \qquad \Box_g\phi=\mathcal V'(\phi).
 \label{eq:h5-field-equations}
\end{equation}

\begin{proposition}[Prescribed Heisenberg singularity data]
\label{prop:h5-einstein}
Let $\mathcal V\in C^\infty(\R)$ satisfy \eqref{eq:potential-growth}, and choose
$A_a>0$, $\phi_0\in\R$, and real data $(p_0,\ldots,p_{2m},q)$ with
\begin{equation}
 \sum_{a=0}^{2m}p_a=1,\qquad
 \sum_{a=0}^{2m}p_a^2+q^2=1,\qquad c|q|<2,\qquad
 \chi_r=1+p_0-p_{2r-1}-p_{2r}>0\quad(1\le r\le m).
 \label{eq:h5-kasner-data}
\end{equation}
There is a smooth diagonal solution of \eqref{eq:h5-field-equations}
on $(0,t_*)\times H_n$ with
\begin{align}
 g&=-\dd t^2+\sum_{a=0}^{2m}a_a(t)^2(\omega^a)^2,
 &a_a(t)&=A_at^{p_a}(1+O(t^\mu)),\notag\\
 \phi(t)&=q\log t+\phi_0+O(t^\mu),
 &\mu&=\min\{2\chi_1,\ldots,2\chi_m,2-c|q|\}>0.
 \label{eq:h5-asymptotics}
\end{align}
The error estimates hold after one and two applications of
$t\partial_t$.  If a horizontal exponent is negative, its past end is
$C^0$-inextendible on every discrete left quotient of $H_n$.
Such data give a future causally complete, globally $C^0$-inextendible
maximal homogeneous development if $n\in\{3,5\}$ and $\mathcal V\ge0$,
or if $\mathcal V\ge V_*>0$ in any odd spatial dimension.
\end{proposition}

\begin{proof}
Put $a_a=e^{u_a}$, $S=\sum_a u_a$,
$\eta=(\prod_a A_a)^{-1}$ and $N=\eta e^S$.
In the spatial orthonormal frame the brackets are
$[E_{2r-1},E_{2r}]=b_rE_0$, where
$b_r=e^{u_0-u_{2r-1}-u_{2r}}$.
The two-step Ricci formulas
$\Ric^{(h_t)}|_{e_0^\perp}=J_0^2/2$ and
$\Ric^{(h_t)}(E_0,E_0)=-\operatorname{tr}(J_0^2)/4$ give
\begin{equation}
 \begin{aligned}
 \Ric^{(h_t)}(E_0,E_0)&=\tfrac12\sum_r b_r^2,
 &R(h_t)&=-\tfrac12\sum_r b_r^2,\\
 \Ric^{(h_t)}(E_{2r-1},E_{2r-1})
 &=\Ric^{(h_t)}(E_{2r},E_{2r})=-\tfrac12b_r^2.
 \end{aligned}
 \label{eq:h5-spatial-ricci}
\end{equation}
All other entries vanish.  The sign automorphisms
$e_0\mapsto\varepsilon_0e_0$,
$e_{2r-1}\mapsto\varepsilon_re_{2r-1}$,
$e_{2r}\mapsto\varepsilon_0\varepsilon_re_{2r}$ have distinct,
nontrivial characters on the spatial frame and fix time.  They preserve
the diagonal ansatz and force both the off-diagonal spatial Einstein
equations and the momentum equations.

Define
\[
 W_r=S+u_0-u_{2r-1}-u_{2r},\qquad
 \mathcal B_r=\tfrac12\eta^2e^{2W_r},\qquad
 \mathcal P=\eta^2e^{2S}\mathcal V(\phi),\qquad k_n=\frac2{n-1}.
\]
Since $N'/N=S'$, the remaining evolution equations are
\begin{align}
 u_0''&=-\sum_r\mathcal B_r+k_n\mathcal P,
 &u_{2r-1}''=u_{2r}''&=\mathcal B_r+k_n\mathcal P,\notag\\
 \phi''&=-\eta^2e^{2S}\mathcal V'(\phi),
 &S''&=\sum_r\mathcal B_r+nk_n\mathcal P.
 \label{eq:h5-harmonic-evolution}
\end{align}
The Hamiltonian constraint is
\begin{equation}
 \mathcal C=(S')^2-\sum_a(u_a')^2-(\phi')^2
                   -\sum_r\mathcal B_r-2\mathcal P=0.
 \label{eq:h5-hamiltonian}
\end{equation}
In particular $W_r''=-2\mathcal B_r+2\mathcal P$; the other curvature
terms cancel for every $m$.

Write $u_a=p_a\tau+\log A_a+w_a$,
$\phi=q\tau+\phi_0+v$, and $Y=(w_0,\ldots,w_{2m},v)$.
The right-hand side $F(\tau,Y)$ of the equations for $Y$ obeys
\[
 |F(\tau,0)|\le Me^{\mu\tau},\qquad
 |F(\tau,Y)-F(\tau,Z)|\le Le^{\mu\tau}|Y-Z|
 \quad(\tau\le0,\ |Y|,|Z|\le1).
\]
Indeed, the curvature terms have exponents $2\chi_r$, whereas
\eqref{eq:potential-growth} bounds the potential and its derivatives
by a multiple of $e^{(2-c|q|)\tau}$ after multiplication by $e^{2S}$.
On $(-\infty,T]$ use
$\|Y\|_\mu=\sup_{\tau\le T}e^{-\mu\tau}|Y(\tau)|$ and
\[
 (\mathcal TY)(\tau)=
       \int_{-\infty}^{\tau}(\tau-s)F(s,Y(s))\,\dd s.
\]
Then $\|\mathcal T0\|_\mu\le M/\mu^2$ and
\[
 \|\mathcal TY-\mathcal TZ\|_\mu
       \le\frac{Le^{\mu T}}{4\mu^2}\|Y-Z\|_\mu.
\]
Choose a ball of radius $R_0>2M/\mu^2$ and then $T<0$ with
$R_0e^{\mu T}\le1$ and $Le^{\mu T}/(4\mu^2)<1/2$.
The map preserves the ball and is a contraction.  Its smooth fixed point
satisfies $Y,Y',Y''=O(e^{\mu\tau})$.

No separate constraint correction is needed.  The identities
\begin{align*}
 \mathcal B_r'&=2\mathcal B_rW_r',\qquad
 \mathcal P'=2S'\mathcal P+
          \eta^2e^{2S}\mathcal V'(\phi)\phi',\\
 \sum_a u_a'u_a''&=
   \sum_r\mathcal B_r(-u_0'+u_{2r-1}'+u_{2r}')
                         +k_n\mathcal P S'
\end{align*}
give $\mathcal C'=0$ upon differentiation of
\eqref{eq:h5-hamiltonian}; the potential terms cancel because
$(n-1)k_n=2$.  Its limit at $-\infty$ is zero by
\eqref{eq:h5-kasner-data}.  The normal equation follows as well:
\[
 \Ric(\partial_\tau,\partial_\tau)
 =-S''+(S')^2-\sum_a(u_a')^2
 = (\phi')^2-k_n\mathcal P.
\]
Finally $N=e^\tau\exp(\sum_a w_a)$, so
$t=\int_{-\infty}^\tau N(s)\,\dd s
=e^\tau(1+O(e^{\mu\tau}))$ and
$\tau=\log t+O(t^\mu)$.  These estimates and their first two
derivatives give \eqref{eq:h5-asymptotics}.

If a horizontal exponent is negative, all $p_a<1$ by the Kasner
relations.  The remaining hypotheses of
Corollary~\ref{cor:all-two-step-anisotropic-powers} are precisely
$\chi_r>0$, and its metric comparison applies to
$h_t=(1+O(t^\mu))\sum_a A_a^2t^{2p_a}(\omega^a)^2$ in the sense
of quadratic forms.  This proves past inextendibility on all the stated
quotients.  One has $\theta=t^{-1}(1+O(t^\mu))>0$ near zero
and $R<0$ throughout the diagonal evolution.
Theorem~\ref{thm:homogeneous-future}\textup{(ii),(iii)} gives future
causal completeness in the two stated potential regimes.  The slices are Cauchy by completeness
of an invariant auxiliary metric and uniform comparison on compact
time intervals.  Hence there is no future extension boundary
\cite{GLS2018}; the past conclusion proves global inextendibility.
\end{proof}

\begin{proposition}[Exact massless developments in every odd spatial dimension]
\label{prop:h5-exact-massless}
Assume \eqref{eq:h5-kasner-data} with $\mathcal V=0$, taking $c=0$,
and define
\[
 \kappa_r=\frac{A_0}{A_{2r-1}A_{2r}},\qquad
 F_r(\tau)=1+\frac{\kappa_r^2}{4\chi_r^2}e^{2\chi_r\tau}.
\]
For all $\tau\in\R$, the formulas
\begin{align}
 a_0&=A_0e^{p_0\tau}\prod_{r=1}^mF_r^{-1/2},\notag\\
 a_{2r-1}&=A_{2r-1}e^{p_{2r-1}\tau}F_r^{1/2},
 &a_{2r}&=A_{2r}e^{p_{2r}\tau}F_r^{1/2},\notag\\
 N&=e^\tau\prod_{r=1}^mF_r^{1/2},
 &\phi&=q\tau+\phi_0
 \label{eq:h5-exact-family}
\end{align}
solve $\Ric(g)=\dd\phi\otimes\dd\phi$ and $\Box_g\phi=0$.
Their proper-time range is $(0,\infty)$, and every discrete left
quotient is future timelike and null geodesically complete.
If some horizontal $p_i<0$, every such quotient is globally
$C^0$-inextendible.  The case $q=0$ is vacuum.
\end{proposition}

\begin{proof}
The logarithms in \eqref{eq:h5-exact-family} obey $N=\eta e^S$ and
\[
 W_r=\chi_r\tau+C_r-\log F_r,\qquad
 \eta e^{C_r}=\kappa_r,\qquad
 (\log F_r)''=
       \frac{\kappa_r^2e^{2\chi_r\tau}}{F_r^2}=2\mathcal B_r.
\]
Thus \eqref{eq:h5-harmonic-evolution} holds with $\mathcal P=0$.
The preceding constraint calculation proves all Einstein equations.
At the past end, put $\mu=2\min_r\chi_r$.
The potential term $\mathcal P$ vanishes identically and
$N=e^\tau\prod_rF_r^{1/2}$ exactly, so the potential-dependent cap
$2-c|q|$ in \eqref{eq:h5-asymptotics} is absent.  The factors satisfy
$F_r=1+O(e^{\mu\tau})$ with differentiated estimates, whence
$t=e^\tau(1+O(e^{\mu\tau}))$.  This gives the prescribed past
asymptotics with error $O(t^\mu)$ after two applications of
$t\partial_t$.  The coefficients are positive and smooth at every
finite $\tau$, and $N\ge e^\tau$.
Consequently $t=\int_{-\infty}^\tau N(s)\,\dd s$ ranges over
$(0,\infty)$.  Future completeness follows from the asymptotics proved
in the next proposition.  The past obstruction just established,
together with the absence of a future boundary, then gives the global
conclusion.
\end{proof}

\begin{proposition}[Kasner transition and scalar fraction]
\label{prop:heisenberg-scattering}
For \eqref{eq:h5-exact-family}, set $D=1+\sum_r\chi_r$.  The future
proper-time Kasner exponents and scalar slope are
\begin{equation}
 P_0=\frac{p_0-\sum_r\chi_r}{D},\qquad
 P_{2r-1}=\frac{p_{2r-1}+\chi_r}{D},\qquad
 P_{2r}=\frac{p_{2r}+\chi_r}{D},\qquad Q=\frac qD.
 \label{eq:heisenberg-scattering}
\end{equation}
They satisfy $\sum_aP_a=1$, $\sum_aP_a^2+Q^2=1$, and
$1+P_0-P_{2r-1}-P_{2r}=-\chi_r/D$.
For positive constants $A_a^+$ and some $\phi_+$,
\begin{equation}
 a_a(t)=A_a^+t^{P_a}(1+O(t^{-\epsilon})),\qquad
 \phi(t)=Q\log t+\phi_++O(t^{-\epsilon})
 \label{eq:heisenberg-future-asymptotics}
\end{equation}
with two differentiated estimates for every
$0<\epsilon<\min\{1,2\min_r\chi_r/D\}$.
Writing $\Omega_\phi=\rho_\phi/\rho_{\rm crit}$ with
$\rho_{\rm crit}=(n-1)\theta^2/(2n)$, one has
\begin{equation}
 \Omega_{\phi,-}=\frac{nq^2}{n-1},\qquad
 \Omega_{\phi,+}=\frac{\Omega_{\phi,-}}{D^2}.
 \label{eq:heisenberg-scalar-scattering}
\end{equation}
\end{proposition}

\begin{proof}
At $+\infty$, $\log F_r=2\chi_r\tau+
\log(\kappa_r^2/(4\chi_r^2))+O(e^{-2\chi_r\tau})$ with all
derivatives.  Hence $N=C_Ne^{D\tau}(1+O(e^{-2\chi_*\tau}))$,
where $\chi_* =\min_r\chi_r>0$ and $C_N>0$.
Integration and inversion give \eqref{eq:heisenberg-future-asymptotics};
the strict restriction on $\epsilon$ also covers the integration
constant and the possible logarithmic resonance $2\chi_*=D$.
The slopes in \eqref{eq:heisenberg-scattering} follow immediately.

The Kasner identities can be checked without expansion.  On the
logarithmic scale factors use the bilinear form
\[
 B(x,y)=\Bigl(\sum_a x_a\Bigr)\Bigl(\sum_a y_a\Bigr)
                                  -\sum_a x_ay_a.
\]
For $v_r=-e_0+e_{2r-1}+e_{2r}$,
$B(v_r,v_s)=-2\delta_{rs}$ and $B(p,v_r)=\chi_r$.
The maps $x\mapsto x+B(x,v_r)v_r$ are commuting reflections preserving
$B$.  Thus $r=p+\sum_r\chi_rv_r$ has
$\sum_a r_a=D$ and $B(r,r)=B(p,p)=q^2$.
Dividing by $D^2$ proves the two Kasner relations; substitution gives
the stated curvature slopes.  This realizes the curvature-wall
reflection description of cosmological dynamics
\cite{DamourHenneauxNicolai2003} by an exact evolution on $H_n$.
Finally $\theta=t^{-1}+O(t^{-1-\epsilon})$ and
$\dot\phi=Q/t+O(t^{-1-\epsilon})$ at the future end, with the
corresponding past formulas, proving
\eqref{eq:heisenberg-scalar-scattering}.

To verify the completeness assertion in
Proposition~\ref{prop:h5-exact-massless}, let $P_* =\min_aP_a$.
The Kasner relations and Cauchy-Schwarz imply
$P_*\ge-(n-2)/n>-1$.
For $K=\tfrac12\dot h$, \eqref{eq:heisenberg-future-asymptotics} gives
\[
 \lambda_{\min}(h^{-1}K)=\frac{P_*}{t}+O(t^{-1-\epsilon}).
\]
If $E=\dd t/\dd\lambda>0$ along an affinely parametrized future
causal geodesic and $w=\dd x/\dd t$, then $h(w,w)\le1$ and
$(\log E)^{\displaystyle\cdot}=-K(w,w)$.  With
$\rho=\max\{0,-P_*\}<1$ this yields
$E(t)\le Ct^\rho$ at large $t$, so
$\int^\infty\dd t/E(t)=\infty$.
If $t$ instead has a finite upper limit, uniform metric comparison and
the energy equation keep the position and velocity in a compact subset
of the tangent bundle of a finite slab, using completeness of a fixed
invariant metric; the geodesic therefore extends.  This proves both
timelike and null completeness on the cover.  Every quotient geodesic
lifts to the cover, proving the assertion for arbitrary discrete
quotients.
\end{proof}

The family has vacuum examples in every odd spatial dimension.  Take
\[
 p_1=-\frac{n-2}{n},\qquad p_a=\frac2n\ (a\ne1),\qquad q=0.
\]
Then $\chi_1=2(n-1)/n$ and $\chi_r=(n-2)/n$ for $r>1$.
All brackets are nonzero, the spatial Lie algebra is indecomposable,
and every discrete quotient is globally $C^0$-inextendible.
For a nonzero scalar in dimension $5+1$, one may take
\begin{equation}
 (p_0,p_1,p_2,p_3,p_4)=
 \left(\frac15,-\frac1{10},\frac3{10},\frac3{10},\frac3{10}\right),
 \qquad q=\frac{\sqrt{17}}5,
 \qquad (\chi_1,\chi_2)=(1,3/5).
 \label{eq:h5-concrete-scalar-data}
\end{equation}
Here $D=13/5$, so the massless evolution changes the limiting scalar
fraction from $17/20$ to $85/676$.
For these same past data, Proposition~\ref{prop:h5-einstein} also applies
to every nonnegative polynomial potential, including
$\mathcal V=\Lambda+M^2\phi^2/2$ with $\Lambda,M^2\ge0$, since
\eqref{eq:potential-growth} then holds with arbitrarily small positive
$c$.  These potential developments are globally $C^0$-inextendible;
the massless scattering formula is not asserted for them.
Compact quotients are included: in coordinates
$\omega^0=\dd z-\sum_r x_r\dd y_r$ and
$(\omega^{2r-1},\omega^{2r})=(\dd x_r,\dd y_r)$, the integral
Heisenberg group is a lattice.

More generally, a negative exponent implies
\[
 1-q^2\ge p_j^2+\frac{(1-p_j)^2}{n-1}>\frac1{n-1},
 \qquad \Omega_{\phi,-}<1-\frac1{(n-1)^2}.
\]
The higher-dimensional examples thus allow a larger scalar fraction
while retaining the expanding direction needed by the past obstruction.
They do not settle the all-positive-exponent regime in three spatial
dimensions.

\begin{proposition}[Exponential potentials and Heisenberg Kasner asymptotics]
\label{prop:heisenberg-exponential-admissibility}
Let $n=2m+1\ge3$, and consider the field equations
\eqref{eq:h5-field-equations} on the Heisenberg group with coframe
\eqref{eq:h5-coframe}.  Let
\[
 \mathcal V(\phi)=\sum_{j=1}^{J}V_j e^{\lambda_j\phi},
 \qquad J\ge1,\quad V_j>0,\quad \lambda_j\in\mathbb R.
\]
Fix $A_a>0$, $\phi_0\in\mathbb R$, and real numbers
$(p_0,\ldots,p_{2m},q)$ satisfying
\[
 \sum_a p_a=1,\qquad \sum_a p_a^2+q^2=1.
\]
There is a smooth diagonal solution with
\begin{gather}
 a_a(t)=A_at^{p_a}(1+o(1)),\qquad
 \phi(t)=q\log t+\phi_0+o(1),\notag\\
 t\dot a_a(t)/a_a(t)\longrightarrow p_a,\qquad
 t\dot\phi(t)\longrightarrow q
 \label{eq:heisenberg-exponential-kasner-limits}
\end{gather}
as $t\downarrow0$ if and only if
\begin{equation}
 \chi_r:=1+p_0-p_{2r-1}-p_{2r}>0\quad(1\le r\le m),
 \qquad 2+\lambda_jq>0\quad(1\le j\le J).
 \label{eq:heisenberg-exponential-admissibility}
\end{equation}
When these inequalities hold, the solution can be chosen so that the
relative scale-factor errors and the additive scalar error are
$O(t^\mu)$, with the same estimate after one and two applications of
$t\partial_t$, where
\[
 \mu=\min\{2\chi_1,\ldots,2\chi_m,
                    2+\lambda_1q,\ldots,2+\lambda_Jq\}>0.
\]
If a horizontal exponent is negative, the past end is
$C^0$-inextendible on every discrete left quotient.  Under this
negative-horizontal-exponent condition, for $n=3,5$, or in every odd dimension when
$\min_j\lambda_j\le0\le\max_j\lambda_j$, its maximal homogeneous
development is also future causally geodesically complete and globally
$C^0$-inextendible.
\end{proposition}

\begin{proof}
For sufficiency, use the harmonic variables and equations of
Proposition~\ref{prop:h5-einstein}.  Thus
$u_a=p_a\tau+\log A_a+w_a$, $\phi=q\tau+\phi_0+v$,
$S=\sum_a u_a$, $\eta=(\prod_aA_a)^{-1}$, and $N=\eta e^S$.
Write $Y=(w_0,\ldots,w_{2m},v)$ and
$\kappa_r=A_0/(A_{2r-1}A_{2r})$.  The curvature terms are
\[
 \mathcal B_r=\frac{\kappa_r^2}{2}e^{2\chi_r\tau}
 \exp\!\left(2\sum_a w_a+2w_0-2w_{2r-1}-2w_{2r}\right),
\]
and each potential term is
\[
 \mathcal P_j=V_j e^{\lambda_j\phi_0}
       e^{(2+\lambda_jq)\tau}
       \exp\!\left(2\sum_a w_a+\lambda_jv\right).
\]
The scalar equation contains $-\sum_j\lambda_j\mathcal P_j$.
Consequently the right-hand side $F(\tau,Y)$ of the second-order
system for $Y$ satisfies, for $\tau\le0$ and $|Y|,|Z|\le1$,
\[
 |F(\tau,0)|\le M e^{\mu\tau},\qquad
 |F(\tau,Y)-F(\tau,Z)|\le L e^{\mu\tau}|Y-Z|.
\]
On $(-\infty,T]$, equip continuous functions with the norm
$\|Y\|_\mu=\sup_{\tau\le T}e^{-\mu\tau}|Y(\tau)|$ and set
\[
 (\mathcal TY)(\tau)=
       \int_{-\infty}^{\tau}(\tau-s)F(s,Y(s))\,\dd s.
\]
The integral identities
\[
 \int_{-\infty}^{\tau}(\tau-s)e^{\mu s}\,\dd s
 =\frac{e^{\mu\tau}}{\mu^2},\qquad
 \int_{-\infty}^{\tau}(\tau-s)e^{2\mu s}\,\dd s
 =\frac{e^{2\mu\tau}}{4\mu^2}
\]
show that $\|\mathcal T0\|_\mu\le M/\mu^2$ and that its Lipschitz
constant on a fixed ball is at most $Le^{\mu T}/(4\mu^2)$.
Choose the ball radius larger than $2M/\mu^2$ and then choose $T<0$
so that the ball lies in $|Y|\le1$ and this Lipschitz constant is
less than $1/2$.  The contraction theorem gives a fixed point with
$Y,Y',Y''=O(e^{\mu\tau})$.  Smoothness follows from the differential
equations on every finite interval.

The constraint \eqref{eq:h5-hamiltonian} is conserved by these
equations.  All $\mathcal B_r$ and $\mathcal P_j$ tend to zero, while
its quadratic part tends to $1-\sum_a p_a^2-q^2=0$.
It therefore vanishes identically.  The diagonal evolution, the normal
equation, and the mixed equations verified in
Proposition~\ref{prop:h5-einstein} give the full field equations.
Since
\[
 N=e^\tau\exp\!\left(\sum_a w_a\right),\qquad
 t=\int_{-\infty}^{\tau}N(s)\,\dd s
   =e^\tau(1+O(e^{\mu\tau})),
\]
one has $\tau=\log t+O(t^\mu)$.  Moreover
$t\partial_t=(t/N)\partial_\tau$ with
$t/N=1+O(e^{\mu\tau})$ and the corresponding differentiated bounds.
This proves the claimed estimates through two applications of
$t\partial_t$.

For necessity, consider any diagonal solution with
\eqref{eq:heisenberg-exponential-kasner-limits}.  Put
$H_a=\dot a_a/a_a$ and $\theta=\sum_aH_a$.
The proper-time Hamiltonian constraint is
\[
 \theta^2-\sum_aH_a^2-\dot\phi^{\,2}
 =\frac12\sum_{r=1}^m
       \left(\frac{a_0}{a_{2r-1}a_{2r}}\right)^2
        +2\sum_{j=1}^{J}V_j e^{\lambda_j\phi}.
\]
Multiply by $t^2$.  Its left-hand side tends to zero by the
differentiated limits and the generalized Kasner relations.
Every term on the right is positive, and its asymptotic form is
respectively
\[
 \frac{\kappa_r^2}{2}t^{2\chi_r}(1+o(1)),\qquad
 2V_j e^{\lambda_j\phi_0}t^{2+\lambda_jq}(1+o(1)).
\]
All coefficients are strictly positive.  Each term must tend to zero,
which gives every strict inequality in
\eqref{eq:heisenberg-exponential-admissibility}.

A negative horizontal exponent and the Kasner relations imply
$\max_a p_a<1$.  The strict curvature inequalities and the relative
metric estimates therefore give the past conclusion by
Corollary~\ref{cor:all-two-step-anisotropic-powers}.
Finally $\mathcal V\ge0$, the spatial scalar curvature is negative,
and $\theta=t^{-1}(1+O(t^\mu))>0$ near the past end.
For $n=3,5$, Proposition~\ref{prop:nonnegative-future-completeness}
proves future causal completeness.  If the slopes straddle zero, the
potential has a positive lower bound: a zero slope supplies a constant
term, and slopes of both signs force growth at both ends of the scalar
line.  Proposition~\ref{prop:positive-floor-future-completeness} then
applies in every odd dimension.  The Cauchy property of the
product slices and the one-sided completeness obstruction
\cite{GLS2018} then give global $C^0$-inextendibility.
\end{proof}

The equivalence concerns precisely the differentiated generalized
Kasner asymptotics \eqref{eq:heisenberg-exponential-kasner-limits}.
At equality, the constraint excludes these asymptotics but does not
exclude a different leading balance.  In dimension three,
$\chi_1=2p_0$, so positivity of the central Kasner exponent is also
necessary for this regime.

\begin{corollary}[Globally inextendible developments with a positive potential floor]
\label{cor:heisenberg-positive-floor}
Let $n=2m+1\ge3$, and fix the leading amplitudes and Kasner data of
Proposition~\ref{prop:h5-einstein}, with
\[
 \sum_{a=0}^{2m}p_a=1,\qquad
 \sum_{a=0}^{2m}p_a^2+q^2=1,\qquad
 \chi_r:=1+p_0-p_{2r-1}-p_{2r}>0,
\]
and some horizontal $p_i<0$.  Suppose that $\mathcal V$ is smooth,
satisfies~\eqref{eq:positive-potential-floor}, and satisfies the
growth bound~\eqref{eq:potential-growth} with $c|q|<2$.
Then these past data are realized by a maximal homogeneous
Einstein-scalar development which is future causally geodesically
complete and globally $C^0$-inextendible on $H_n$ and every discrete
left quotient.

The same conclusion holds for
\begin{equation}
 \mathcal V(\phi)=\Lambda+
        \sum_{j=1}^{J}V_j e^{\lambda_j\phi},
 \qquad \Lambda>0,\quad V_j>0,
 \label{eq:positive-floor-exponential-sum}
\end{equation}
under the exact inequalities $2+\lambda_jq>0$ for all $j$, without
the two-sided growth restriction $c|q|<2$.  More generally, the
positive finite sum in
Proposition~\ref{prop:heisenberg-exponential-admissibility}, under
its strict inequalities~\eqref{eq:heisenberg-exponential-admissibility},
has the same global conclusion in every odd dimension whenever
\begin{equation}
 \min_j\lambda_j\le0\le\max_j\lambda_j.
 \label{eq:exponential-sum-positive-floor}
\end{equation}
\end{corollary}

\begin{proof}
Proposition~\ref{prop:h5-einstein} constructs the prescribed past
asymptotics in the first case and proves past $C^0$-inextendibility
on every discrete quotient.  The differentiated estimates give
$\theta=t^{-1}(1+O(t^\mu))>0$ near the past end.  The scalar
curvature formula~\eqref{eq:h5-spatial-ricci} remains negative along
the diagonal homogeneous evolution.  Proposition~\ref{prop:positive-floor-future-completeness}
therefore proves future causal completeness in every dimension.
The product slices are Cauchy: on each compact proper-time interval,
the spatial metrics are uniformly comparable with a fixed complete
left-invariant metric, on both the group and its quotients.
Future timelike completeness excludes a future extension boundary
by~\cite{GLS2018}.  Together with the past obstruction, this proves
global $C^0$-inextendibility.

For~\eqref{eq:positive-floor-exponential-sum}, treat $\Lambda$ as
an additional positive exponential term of slope zero.
Proposition~\ref{prop:heisenberg-exponential-admissibility} constructs
the differentiated Kasner asymptotics under precisely the stated
curvature and potential inequalities.  The zero slope contributes
the automatic inequality $2>0$, and $\mathcal V\ge\Lambda$ supplies
the positive lower bound.  The preceding global argument applies.

For a positive finite exponential sum without a separately displayed
constant, condition~\eqref{eq:exponential-sum-positive-floor} is
equivalent to the existence of a positive lower bound.  A zero slope
already supplies a positive constant term.  If there are slopes of
both signs, the potential tends to $+\infty$ at both ends of the
scalar line and attains a strictly positive minimum.  Conversely,
if all slopes have the same strict sign, the potential tends to zero
at one end.  This proves the final assertion by the same construction
and completeness argument.
\end{proof}

For example, $\mathcal V(\phi)=V_0\cosh(\lambda\phi)$, $V_0>0$,
admits the prescribed differentiated Kasner regime exactly when
$\chi_r>0$ for every bracket and $|\lambda q|<2$.
Whenever a horizontal exponent is negative, the resulting maximal
homogeneous developments are future causally complete and globally
$C^0$-inextendible in every odd spatial dimension.
Similarly, $\mathcal V=\Lambda+P$ with $\Lambda>0$ and a
nonnegative polynomial $P$ has this conclusion for all the stated
Kasner data: its first two derivatives satisfy the growth bound for
every $c>0$, so one can choose $c|q|<2$.
These completeness and inextendibility conclusions do not assert
the massless Kasner scattering law or exponential future
isotropization for these general potentials.

\subsection{Einstein solutions with distinct central scales}

\begin{proposition}[Distinct central scales in vacuum and scalar cosmologies]
\label{prop:einstein-optimal-separation}
Let $\mathsf N=H_3\times H_3$ and choose an invariant coframe
$\omega^1,\ldots,\omega^6$ satisfying
\[
 \dd\omega^1=-\omega^2\wedge\omega^3,\qquad
 \dd\omega^4=-\omega^5\wedge\omega^6,
\]
with the other coframe elements closed.  There exists a smooth
diagonal Einstein-massless-scalar solution
\[
 g=-\dd t^2+\sum_{i=1}^6 a_i(t)^2(\omega^i)^2,
 \qquad \operatorname{Ric}_g=\dd\phi\otimes\dd\phi,
\]
whose past asymptotics are
\begin{equation}
 \begin{split}
 a_i(t)&=t^{p_i}\bigl(1+O(t^{3/5})\bigr),\\
 \phi(t)&=q\log t+O(t^{3/5}),\\
 (p_1,\ldots,p_6)&=\frac1{20}(-8,-1,7,8,7,7),
 \qquad q=\frac{\sqrt{31}}{10}.
 \end{split}
 \label{eq:einstein-optimal-separation-data}
\end{equation}
The estimates hold after two applications of $t\partial_t$.
There is also a vacuum solution with the same error order and
Kasner data
\begin{equation}
 p^{\rm vac}=\frac1{20}
       (-8,3-\sqrt{78},3+\sqrt{78},8,7,7),\qquad q=0.
 \label{eq:vacuum-optimal-separation-data}
\end{equation}
For either maximal expanding homogeneous development,
\begin{align}
 cT^{3/10}\le\kappa_{\rm opt}(T)&\le C T^{3/10}
       &&\text{for sufficiently small }T>0,\notag\\
 \kappa_{\rm an}(T)&=\infty
       &&\text{for every }T\text{ in the proper-time domain}.
 \label{eq:einstein-optimal-separation}
\end{align}
Both developments are future timelike and null complete and globally
$C^0$-inextendible on $\mathsf N$ and every discrete left quotient.
The strict separation of the filling conditions, completeness and
inextendibility persist for a relative open neighborhood of either
displayed Kasner datum within the corresponding diagonal asymptotic
family.  The positive error and decay exponents may vary in these
neighborhoods.
\end{proposition}

\begin{proof}
Write $a_i=e^{\alpha_i}$ and use harmonic time $\tau$ with
$\dd t=e^S\dd\tau$, $S=\sum_i\alpha_i$.  Put
\[
 v_1=(-1,1,1,0,0,0),\qquad
 v_2=(0,0,0,-1,1,1),\qquad
 w_r=\mathbf1-v_r,
\]
and
\[
 \mathcal B_r=\frac12 e^{2w_r\cdot\alpha},\qquad r=1,2.
\]
For a diagonal invariant metric, the only nonzero orthonormal
spatial brackets have lengths
$e^{\alpha_1-\alpha_2-\alpha_3}$ and
$e^{\alpha_4-\alpha_5-\alpha_6}$.  In each block, the spatial Ricci
eigenvalues are one half the squared bracket length in the central
direction and minus one half in the two horizontal directions.
Independent sign changes of the four horizontal generators induce
the six distinct nontrivial characters
$\eta_1\eta_2,\eta_1,\eta_2,\eta_3\eta_4,\eta_3,\eta_4$
on the invariant frame.  They preserve the spacetime metric and
force all off-diagonal spatial and normal-spatial Ricci components
to vanish.  Thus the momentum constraint vanishes, and the
Einstein-scalar equations are
\begin{equation}
 \alpha''=\mathcal B_1v_1+\mathcal B_2v_2,
 \qquad \phi''=0,
 \qquad
 (S')^2-\sum_i(\alpha_i')^2
       = (\phi')^2+\mathcal B_1+\mathcal B_2.
 \label{eq:einstein-product-heisenberg-ode}
\end{equation}
Here and in the construction below primes denote $\tau$ derivatives.
The two curvature terms are coupled:
$w_r\cdot v_r=-2$ and $w_r\cdot v_s=1$ for $r\ne s$.
No decoupled formula for the Heisenberg family is used.

The data in \eqref{eq:einstein-optimal-separation-data} satisfy
\[
 \sum_i p_i=1,\qquad \sum_i p_i^2+q^2=1,\qquad
 \chi_1:=w_1\cdot p=\frac3{10},\qquad
 \chi_2:=w_2\cdot p=\frac7{10}.
\]
Set $\alpha=p\tau+u$ and $\phi=q\tau$.  On a sufficiently early
interval $(-\infty,\tau_0]$, solve the Volterra equation
\begin{equation}
 u(\tau)=\frac12\sum_{r=1}^2v_r
   \int_{-\infty}^{\tau}(\tau-s)
     e^{2\chi_rs+2w_r\cdot u(s)}\,\dd s.
 \label{eq:einstein-product-heisenberg-volterra}
\end{equation}
For completeness, take $\mu=3/5$ and the norm
$\|u\|_\mu=\sup_{\tau\le\tau_0}e^{-\mu\tau}|u(\tau)|$.
On a fixed ball in this space, first choose its radius larger than
the norm of the right side at $u=0$, and then decrease $\tau_0$.
The inequalities $2\chi_r\ge\mu$ and
\[
 \int_{-\infty}^{\tau}(\tau-s)e^{\mu s}\,\dd s
       =\frac{e^{\mu\tau}}{\mu^2}
\]
show that the right side preserves this ball.  The mean-value theorem
for the exponential gives a Lipschitz constant at most
$C e^{\mu\tau_0}$ in the same norm, since the difference of two
integrands is bounded by
$C e^{2\mu s}\|u-\widetilde u\|_\mu$.
Thus the map is a contraction for sufficiently negative $\tau_0$.
Differentiating its integral equation gives
$u,u',u''=O(e^{\mu\tau})$, and the differential equation gives
smoothness and the corresponding higher derivative estimates.

The constraint is preserved.  Indeed, the derivative of
\[
 \mathcal C=(S')^2-\sum_i(\alpha_i')^2-q^2
                  -\mathcal B_1-\mathcal B_2
\]
vanishes, because $\sum_i(v_r)_i=1$ and
$\mathcal B_r'=2(w_r\cdot\alpha')\mathcal B_r$.
Its limit as $\tau\to-\infty$ is zero by the Kasner relations, so
$\mathcal C=0$.  Hence the constructed functions solve all the
Einstein-scalar equations.  Since
$S=\tau+O(e^{\mu\tau})$, integration of the lapse from $-\infty$
gives $t=e^\tau(1+O(e^{\mu\tau}))$, and proves
\eqref{eq:einstein-optimal-separation-data} in proper time.

For the filling estimates, the central directions are $z_1,z_2$
dual to $\omega^1,\omega^4$.  On a sufficiently early tail the Schur
forms are uniformly comparable with
\begin{equation}
 H_t^0=\operatorname{diag}
    (t^{-1/10},t^{7/10},t^{7/10},t^{7/10}),
 \qquad
 A_t^0=\operatorname{diag}(t^{-4/5},t^{4/5}),
 \qquad L_t=0.
 \label{eq:einstein-optimal-schur-powers}
\end{equation}
The two nonzero brackets have exponents $3/10$ and $7/10$.
The central-component bound
\eqref{eq:anisotropic-bracket-matrix-bound} and its comparison
estimate therefore give
\[
 \kappa_{\rm opt}(T)\le\kappa_{\rm cen}(T)
       \le C\bigl(T^{3/5}+T^{7/5}\bigr)^{1/2}
       \le C' T^{3/10}.
\]
This order is attained.  Place a unit-area smooth ellipse in the
first horizontal pair on $[T/4,T/2]$, with coordinate amplitudes
$T^{(p_3-p_2)/2}$ and $T^{(p_2-p_3)/2}$.
For the central covector dual to $z_1$, its numerator in
\eqref{eq:kappa-opt} is one and its denominator is at most
\[
 C T^{p_2+p_3-p_1-1}=C T^{-3/10}.
\]
Here the dual central norm is $a_1^{-1}$, and the two amplitudes
balance the horizontal energy terms.  Taking this loop and covector
in the variational supremum proves the lower bound
$\kappa_{\rm opt}(T)\ge cT^{3/10}$.

To prove the second assertion in
\eqref{eq:einstein-optimal-separation}, fix $T>0$ and a compact
interval $J$ in a sufficiently early part of $(0,T)$.
Choose smooth planar based loops of signed area one.  For small
$s>0$ with $[s,2s]$ preceding $J$, place such a loop in the second
horizontal pair on $[s,2s]$, using the rescaled parameter $t/s$.
Place a loop of signed area $\varepsilon$ in the first pair on $J$
by multiplying a fixed loop by $\sqrt\varepsilon$.  Set both loops
equal to zero elsewhere.  Their direct sum is a based loop $\xi$
on a compact interval in $(0,T)$, with bracket area
\[
 Q_I(\xi)=\varepsilon z_1+z_2.
\]
On $J$, the central norm of this area has a positive lower bound
independent of $\varepsilon$, because its $z_2$ component equals one.
The contribution to $J_I(\xi)$ on $J$ is therefore at most
$C\varepsilon$.  On $[s,2s]$, the horizontal energy is at most
$C s^{-3/10}$, whereas
\[
 \sqrt{A_t(Q_I(\xi),Q_I(\xi))}
       \ge c\varepsilon s^{-2/5}.
\]
The early contribution is consequently at most
$C\varepsilon^{-1}s^{1/10}$.  Taking
$\varepsilon=s^{1/20}$ proves
\[
 J_I(\xi)\le C s^{1/20}\longrightarrow0.
\]
Thus $\kappa_{\rm an}(T)=\infty$ for every such $T$, and by
monotonicity for every $T$ in the proper-time domain.

All spatial Kasner exponents are less than one, so the cone integral
is finite.  The horizontal exponent $p_2=-1/20$ gives unbounded
diameter on every fixed open set of the cover.  Theorem
\ref{thm:two-step-obstruction} excludes every past continuous extension
on every discrete quotient.  The expansion is positive near the past
end.  The maximal homogeneous continuation has nonpositive spatial
scalar curvature and six spatial dimensions, so
Theorem~\ref{thm:homogeneous-future} gives future timelike and null
completeness.  This proves global $C^0$-inextendibility.

For the vacuum datum \eqref{eq:vacuum-optimal-separation-data},
\[
 \sum_i p_i^{\rm vac}=1,\qquad
 \sum_i(p_i^{\rm vac})^2
   =\frac{244+2\cdot78}{400}=1.
\]
The central exponents, the second horizontal pair and the sum of
the first horizontal pair are unchanged.  Thus both curvature
exponents remain $3/10$ and $7/10$, and the same Volterra construction
with constant scalar gives a vacuum solution with error $O(t^{3/5})$.
The central-component estimate still gives
$\kappa_{\rm opt}(T)=O(T^{3/10})$, and the balanced first-pair
ellipse gives the same lower bound because
$p_2^{\rm vac}+p_3^{\rm vac}=3/10$.
The early loop in the second pair
is unchanged, while the loop in the first pair remains on a fixed
compact time interval and still costs $O(\varepsilon)$.
Consequently the proof of $\kappa_{\rm an}=\infty$ is unchanged.
All vacuum exponents are less than one, and
$(3-\sqrt{78})/20<0$ is horizontal.  The same cone, past obstruction
and future completeness arguments apply.

Finally, all inequalities used above are strict.  In particular, if
$c_1,c_2$ denote the two central exponents and $r_1,r_2$ the second
pair's horizontal exponents, then
\[
 1+c_2-r_1-r_2>0,
 \qquad \delta:=r_1+r_2-1-c_1>0.
\]
For nearby unequal $r_1,r_2$, use on $[s,2s]$ a unit-area ellipse
whose two coordinate amplitudes are
$s^{(r_2-r_1)/2}$ and $s^{(r_1-r_2)/2}$.
Its energy is $O(s^{r_1+r_2-1})$, and the same argument yields
$J_I\le C(\varepsilon+\varepsilon^{-1}s^\delta)$.
The choice $\varepsilon=s^{\delta/2}$ again gives divergence of
$\kappa_{\rm an}$.  The Volterra construction, cone estimate and
negative horizontal exponent also persist in a sufficiently small
relative neighborhood of either displayed Kasner datum on the
appropriate vacuum or scalar Kasner constraint set.  This is a
statement about the diagonal
asymptotic family, with no assertion of genericity in the full space
of invariant or inhomogeneous initial data.
\end{proof}

\begin{corollary}[Einstein separation in every dimension at least six]
\label{cor:einstein-optimal-all-dimensions}
For every spatial dimension $n\ge6$, there are expanding vacuum
developments and Einstein-massless-scalar developments with
nonconstant scalar field on
\[
 \mathsf N_n=H_3\times H_3\times\mathbb R^{n-6}
\]
which are globally $C^0$-inextendible on every discrete left quotient,
future timelike and null complete, and satisfy
$\kappa_{\rm opt}(T)\to0$ while
$\kappa_{\rm an}(T)=\infty$ on every past tail.
\end{corollary}

\begin{proof}
Take the direct metric product of either preceding development with
Euclidean $\mathbb R^{n-6}$ and leave the scalar field independent
of the new coordinates.  The equation
$\operatorname{Ric}_g=\dd\phi\otimes\dd\phi$ is preserved by this
product, so the result is an Einstein-massless-scalar solution in
every stated dimension.  The additional spatial directions are
central and flat, and every bracket area has zero component in them.
Thus $\kappa_{\rm an}$ is unchanged.  In the dual expression for
$\kappa_{\rm opt}$, additional flat covector components leave the
numerator unchanged and can only increase the denominator; choosing
them to vanish shows that $\kappa_{\rm opt}$ is unchanged as well.
The cone integral remains finite and the expanding horizontal
direction remains present.  Theorem~\ref{thm:two-step-obstruction}
therefore gives past $C^0$-inextendibility on every discrete quotient.

A causal geodesic of the metric product has constant Euclidean
velocity, and its projection to the original spacetime is a causal
geodesic, with the same affine parameter.  The latter projection is
complete by Proposition~\ref{prop:einstein-optimal-separation}.
The product is consequently future causally complete.  Every
geodesic on a discrete quotient lifts to the product, so the same
completeness holds on every quotient.  Future completeness together
with the past obstruction proves the global conclusion.
\end{proof}

\subsection{Vacuum singularities and analytic horizons in every odd dimension}
\label{subsec:heisenberg-horizons}

The exact family also gives a complete alternative within its vacuum
sector.  The statement concerns the solutions of
Proposition~\ref{prop:h5-exact-massless}; the classification of all
expanding invariant vacuum data in Section~\ref{subsec:all-vacuum-data}
is specific to three spatial dimensions.

\begin{proposition}[The Heisenberg vacuum alternative]
\label{prop:heisenberg-vacuum-dichotomy}
Let $m\ge1$, and consider the exact family of
Proposition~\ref{prop:h5-exact-massless} with $q=0$ and
$\chi_r>0$ for every $r$.  On $H_{2m+1}$ and every discrete left quotient,
exactly one of the following cases occurs:
\begin{enumerate}[label=\textup{(\roman*)}]
\item A horizontal exponent is negative, and the spacetime is globally
$C^0$-inextendible.
\item $p_0=1$ and $p_i=0$ for every $i>0$.  The past end has two analytic
Ricci-flat one-sided extensions.  In either extension the added boundary
is a nondegenerate Killing horizon and the past Cauchy horizon of each
original spatial slice.  On a compact quotient its null generators are
closed, and the region beyond it contains closed timelike curves.
\end{enumerate}
\end{proposition}

\begin{proof}
Suppose first that every horizontal exponent is nonnegative, and put
$s_r=p_{2r-1}+p_{2r}$.  The strict bracket inequalities give
$0\le s_r<1+p_0$.  If $s_r>0$, then
\[
 p_{2r-1}^2+p_{2r}^2\le s_r^2<(1+p_0)s_r.
\]
If any $s_r$ were positive, summation and the vacuum Kasner identities
would imply
\[
 1-p_0^2=\sum_{i=1}^{2m}p_i^2
 <(1+p_0)\sum_{r=1}^m s_r
 =(1+p_0)(1-p_0)=1-p_0^2,
\]
a contradiction.  Thus every horizontal exponent vanishes and $p_0=1$.
In every other case a horizontal exponent is negative, and the global
inextendibility conclusion of Proposition~\ref{prop:h5-exact-massless}
applies.

For the exceptional datum use the Heisenberg coordinates
\[
 \omega^0=\dd z-\sum_{r=1}^m x_r\dd y_r,
 \qquad \omega^{2r-1}=\dd x_r,\qquad
 \omega^{2r}=\dd y_r.
\]
Then $\chi_r=2$.  Set
\begin{equation}
 s=e^{2\tau},\quad b_r=\frac{\kappa_r^2}{16},\quad
 F_r(s)=1+b_rs^2,\quad H(s)=\prod_{r=1}^mF_r(s),\quad
 h_r=A_{2r-1}^2\dd x_r^2+A_{2r}^2\dd y_r^2.
 \label{eq:heisenberg-horizon-functions}
\end{equation}
On $s>0$ the exact vacuum metric is
\begin{equation}
 g=-\frac{H(s)}{4s}\dd s^2
   +\frac{A_0^2s}{H(s)}
      \left(\dd z-\sum_r x_r\dd y_r\right)^2
   +\sum_rF_r(s)h_r.
 \label{eq:heisenberg-exceptional-metric}
\end{equation}
For $\epsilon\in\{+1,-1\}$ write $z=w+f_\epsilon(s)$, where
\[
 f_\epsilon(s)=\frac{\epsilon}{2A_0}
 \left(\log s+\int_0^s\frac{H(u)-1}{u}\,\dd u\right).
\]
The integrand is a polynomial, and
$f_\epsilon'(s)=\epsilon H(s)/(2A_0s)$.
With $\eta=\dd w-\sum_r x_r\dd y_r$, the $\dd s^2$ terms cancel and give
\begin{equation}
 \widetilde g_\epsilon
   =\epsilon A_0\,\dd s\,\eta
    +\frac{A_0^2s}{H(s)}\eta^2
    +\sum_rF_r(s)h_r,
 \qquad -\infty<s<\infty.
 \label{eq:heisenberg-horizon-extension}
\end{equation}
Here $\dd s\,\eta=(\dd s\otimes\eta+\eta\otimes\dd s)/2$.
All coefficients are analytic for real $s$, and in the coframe
$(\dd s,\eta,\dd x_1,\dd y_1,\ldots)$ the determinant is
\[
 -\frac{A_0^2}{4}H(s)^2\prod_{i=1}^{2m}A_i^2<0.
\]
The $(s,w)$ block has one negative and one positive eigenvalue, while
the horizontal block is positive definite.  Thus
\eqref{eq:heisenberg-horizon-extension} is Lorentzian everywhere.
Its analytic Ricci tensor vanishes on $s>0$ and hence everywhere.

Every central translation commutes with the left Heisenberg action.
Since the change from $z$ to $w$ is a central translation depending only
on $s$, a discrete left action becomes the same action on $(x,y,w)$,
with $s$ fixed.  It remains free and properly discontinuous on
$\R_s\times H_{2m+1}$.  Equation~\eqref{eq:heisenberg-horizon-extension}
therefore defines both extensions on
$\R_s\times\Gamma\backslash H_{2m+1}$ for every discrete $\Gamma$.

The field $K=\partial_w$ is Killing and satisfies
\begin{equation}
 \widetilde g_\epsilon(K,K)=\frac{A_0^2s}{H(s)},\qquad
 \widetilde g_\epsilon^{-1}(\dd s,\dd s)=-\frac{4s}{H(s)},\qquad
 K^\flat\big|_{s=0}=\frac{\epsilon A_0}{2}\dd s.
 \label{eq:heisenberg-horizon-identities}
\end{equation}
Consequently $s=0$ is a Killing horizon.  The identity
$\dd(g(K,K))=-2\varkappa K^\flat$ on the horizon gives
$\varkappa=-\epsilon A_0$, which is nonzero.
Each horizon point is a timelike boundary point: the smooth vector field
\[
 T=\partial_s-
 \frac{\epsilon H(s)}{A_0\sqrt{1+s^2}}K
\]
has $T(s)=1$ and
$\widetilde g_\epsilon(T,T)
=H(s)(s-\sqrt{1+s^2})/(1+s^2)<0$.
Its integral curve through that point enters the original region in
the future direction.

Fix an original slice $\Sigma_{s_0}$, $s_0>0$.  The region $s>0$ is
globally hyperbolic with the original spatial slices as Cauchy surfaces,
and $s$ is a time function there.  Every future-inextendible causal
curve starting at $0<s<s_0$ stays in $s>0$ until it meets
$\Sigma_{s_0}$; the Cauchy property ensures that it does meet this slice.
The analogous past-directed assertion holds for $s>s_0$.
Every point with $s\le0$, on the other hand, lies on a complete central
$K$-orbit which is causal and avoids $\Sigma_{s_0}$.  Since $K$ is
complete, this orbit is parametrized on all of $\mathbb R$--periodically
when it is closed--and is therefore an inextendible causal curve.  It is
null for $s=0$ and timelike for $s<0$.  Thus, in the full extension,
\[
 D(\Sigma_{s_0})=\{s>0\},\qquad
 H^-(\Sigma_{s_0})=\{s=0\}.
\]

If the quotient is compact, the horizontal projections of $\Gamma$
span $\R^{2m}$; otherwise a nonzero linear functional vanishing on their
span would descend to an unbounded continuous function on the compact
quotient.  The nondegeneracy of the horizontal symplectic form then
provides two elements of $\Gamma$ with nonzero central commutator.
Hence $\Gamma$ meets the center in a nonzero discrete subgroup, and
every central orbit is closed.  By
\eqref{eq:heisenberg-horizon-identities} these orbits are null on the
horizon and timelike on $s<0$, as asserted.
\end{proof}

\begin{proposition}[Chronology violation and its sharp horizontal length]
\label{prop:heisenberg-cover-chronology}
In either analytic extension
\eqref{eq:heisenberg-horizon-extension}, the chronology-violating set
is exactly $\{s<0\}$.  Every point of this set lies on a smooth
closed timelike curve in its constant-$s$ slice.  These assertions
hold on $\mathbb R\times H_{2m+1}$ and on every discrete left quotient.

Fix $s<0$ and put
\begin{equation}
 c(s)=\frac{A_0^2|s|}{H(s)},\qquad
 a_r(s)=F_r(s)A_{2r-1}^2,\qquad
 d_r(s)=F_r(s)A_{2r}^2,\qquad
 B(s)=\max_{1\le r\le m}\frac1{\sqrt{a_r(s)d_r(s)}}.
 \label{eq:heisenberg-chronology-weights}
\end{equation}
For a curve $\gamma$ in this slice define its horizontal length by
\[
 L_{\mathrm h}(\gamma)=\int
 \left[\sum_{r=1}^m
       \bigl(a_r(s)\dot x_r^2+d_r(s)\dot y_r^2\bigr)\right]^{1/2}
 \,\dd u.
\]
Among smooth closed timelike curves on the cover, and among such
curves on a quotient whose lifts to the cover are closed, the exact
infimum of this length is
\begin{equation}
 L_*(s)=\frac{4\pi}{B(s)\sqrt{c(s)}}
 =\frac{4\pi\sqrt{H(s)}}{A_0\sqrt{|s|}}
       \min_r\bigl(F_r(s)A_{2r-1}A_{2r}\bigr).
 \label{eq:heisenberg-chronology-minimum}
\end{equation}
Every such timelike curve satisfies $L_{\mathrm h}>L_*(s)$, and
every length greater than $L_*(s)$ is attained.  At the limiting
length there is a smooth closed null curve.  In particular,
\begin{equation}
 L_*(s)=\frac{4\pi}{A_0}
       \min_r(A_{2r-1}A_{2r})\,|s|^{-1/2}
       \bigl(1+O(s^2)\bigr)
 \qquad(s\uparrow0).
 \label{eq:heisenberg-chronology-asymptotic}
\end{equation}
If the quotient contains a nonzero central lattice element, its purely
central timelike circles have zero horizontal length; hence the
unrestricted quotient infimum is zero.  Formula
\eqref{eq:heisenberg-chronology-minimum} concerns precisely the fixed-slice
curves specified above, with closed lifts in the quotient case.
\end{proposition}

\begin{proof}
The induced metric on a constant-$s$ slice with $s<0$ is
\begin{equation}
 -c\eta^2+\sum_{r=1}^m(a_r\dd x_r^2+d_r\dd y_r^2),
 \qquad \eta=\dd w-\sum_{r=1}^m x_r\dd y_r,
 \label{eq:heisenberg-chronology-slice}
\end{equation}
where the positive coefficients in
\eqref{eq:heisenberg-chronology-weights} are constant on the slice.
Choose an index $r$ at which $B$ is attained, write
$x=x_r$, $y=y_r$, $a=a_r$, $d=d_r$, and set all other horizontal
coordinates equal to zero.  For $0\le u\le2\pi$, set
\begin{equation}
 x(u)=\frac{R\cos u}{\sqrt a},\qquad
 y(u)=\frac{R\sin u}{\sqrt d},\qquad
 w(u)=\frac{R^2\sin(2u)}{4\sqrt{ad}}.
 \label{eq:heisenberg-chronology-loop}
\end{equation}
All derivatives agree at the endpoints.  Direct calculation gives
\[
 \dot w-x\dot y=-\frac{R^2}{2\sqrt{ad}},\qquad
 a\dot x^2+d\dot y^2=R^2,
\]
and hence
\begin{equation}
 \widetilde g_\epsilon(\dot\gamma,\dot\gamma)
 =R^2-\frac{cR^4}{4ad}<0
 \quad\Longleftrightarrow\quad
 R>\frac{2\sqrt{ad}}{\sqrt c}.
 \label{eq:heisenberg-chronology-threshold}
\end{equation}
The timelike field
\[
 T=\partial_s-
 \frac{\epsilon H(s)}{A_0\sqrt{1+s^2}}\partial_w
\]
from the preceding proof fixes the time orientation.  On a
constant-$s$ tangent vector,
\[
 \widetilde g_\epsilon(\dot\gamma,T)
 =\epsilon A_0\left(\frac12-
              \frac{s}{\sqrt{1+s^2}}\right)\eta(\dot\gamma).
\]
Thus the displayed curve is future directed for $\epsilon=1$;
its reverse is future directed for $\epsilon=-1$.
Left translations preserve the metric and $T$ and act transitively
on the slice.  Translating the curve therefore places it through
any prescribed point, and projecting it to a quotient preserves
its causal character and time orientation.

Conversely, a smooth closed timelike curve anywhere in the extension
has a point where $s$ attains its maximum.  Its tangent there has
$\dd s=0$.  If the maximum were positive, the induced slice metric
would be positive definite; if it were zero, that metric would be
positive semidefinite.  Either conclusion contradicts timelikeness.
Every closed timelike curve consequently lies in $\{s<0\}$.
Together with the construction, this proves the exact assertion
about chronology violation, on the cover and on every quotient.

To prove the lower bound, let $\gamma$ be a smooth closed timelike
curve at fixed $s<0$ on the cover.  Write $L=L_{\mathrm h}(\gamma)$
and let $v$ denote its horizontal speed.  Equation
\eqref{eq:heisenberg-chronology-slice} gives
$|\eta(\dot\gamma)|>v/\sqrt c$.  The function
$\eta(\dot\gamma)$ never vanishes and therefore has a fixed sign.
Since all coordinates of the curve close,
\begin{equation}
 \left|\sum_{r=1}^m\oint x_r\,\dd y_r\right|
 =\left|\oint\eta\right|
 =\int|\eta(\dot\gamma)|\,\dd u
 >\frac L{\sqrt c}.
 \label{eq:heisenberg-chronology-area-lower}
\end{equation}
In particular the horizontal projection is nonconstant and $L>0$.

For completeness, introduce weighted horizontal coordinates
\[
 X=(\sqrt{a_1}x_1,\sqrt{d_1}y_1,\ldots,
             \sqrt{a_m}x_m,\sqrt{d_m}y_m)
\]
and let $J$ be block diagonal with blocks
\[
 \frac1{\sqrt{a_rd_r}}
 \begin{pmatrix}0&-1\\1&0\end{pmatrix}.
\]
Then $\|J\|=B$.  Parametrize the nonconstant closed horizontal
projection by arclength on $[0,L]$.  This gives a periodic Lipschitz
curve $X$ with $|X'|=1$ almost everywhere; constant portions can
be omitted without changing its line integrals or length.
Writing $\overline X=L^{-1}\int_0^L X\,\dd u$, integration by
parts and the periodic Wirtinger inequality give
\begin{align}
 \left|\sum_r\oint x_r\,\dd y_r\right|
 &=\frac12\left|\int_0^L
           \langle J(X-\overline X),X'\rangle\,\dd u\right|
       \notag\\
 &\le\frac B2
       \left(\int_0^L|X-\overline X|^2\,\dd u\right)^{1/2}
       \left(\int_0^L|X'|^2\,\dd u\right)^{1/2}
       \notag\\
 &\le\frac{BL}{4\pi}\int_0^L|X'|^2\,\dd u
 =\frac{BL^2}{4\pi}.
 \label{eq:heisenberg-chronology-isoperimetry}
\end{align}
Combining this with
\eqref{eq:heisenberg-chronology-area-lower} and dividing by $L$
proves $L>4\pi/(B\sqrt c)$.
The same proof applies to a quotient curve with a closed lift,
because its horizontal length and causal character are preserved
by the covering map.

The curves in \eqref{eq:heisenberg-chronology-loop} have horizontal
length $2\pi R$.  Since the chosen pair satisfies
$\sqrt{ad}=B^{-1}$,
\eqref{eq:heisenberg-chronology-threshold} realizes every
$L>L_*(s)$.  At $R=2/(B\sqrt c)$ the same curve has a
nonvanishing null tangent and length $L_*(s)$, with the same time
orientation.  Thus the timelike infimum is exact and is not
attained.  Finally, $F_r(s)=1+O(s^2)$ and $H(s)=1+O(s^2)$
give \eqref{eq:heisenberg-chronology-asymptotic}.
\end{proof}

The horizon is therefore the exact boundary of chronology violation
in these extensions.  For closed curves on the universal cover,
the infimum of horizontal lengths diverges as the horizon is approached.
The closed-lift qualification in
\eqref{eq:heisenberg-chronology-minimum} is essential on quotients:
if $\Gamma$ contains a nonzero central element, its central orbits
are closed timelike curves on $s<0$ with horizontal length zero.
This occurs on every compact quotient.  Chronology violation
itself is already present on the cover, whereas the quotient adds
these purely central closed curves.

\section{Expanding vacuum data on compact Bianchi II quotients}
\label{sec:taub}

The preceding Heisenberg alternative concerns the exact diagonal family.
In three spatial dimensions the constraints and evolution identify every
expanding invariant vacuum datum.  We now obtain the complete statement
on a fixed compact quotient; additional coordinates for its exceptional
horizons are given in Appendix~\ref{app:taub-coordinates}.

\subsection{All expanding homogeneous vacuum data}
\label{subsec:all-vacuum-data}

The preceding dichotomy extends from the displayed family to every
expanding homogeneous vacuum Bianchi~II datum.  The diagonalization and
integration below make this passage explicit.  Compact locally homogeneous
vacuum horizons and their additional local symmetries were studied in
\cite{ChruscielRendall1995}; the inextendibility conclusion here uses the
continuous-metric obstruction of Proposition~\ref{prop:einstein-scalar}.

Let $\mathfrak g$ be the Heisenberg Lie algebra, with center
$\mathfrak z$, and fix a lattice $\Gamma<H_3$.  An invariant vacuum
datum is a pair $(h,K)$ of a positive inner product and a symmetric
bilinear form on $\mathfrak g$, regarded as tensors on
$\Gamma\backslash H_3$, satisfying
\begin{equation}
 R(h)+(\operatorname{tr}_hK)^2-|K|_h^2=0,
 \qquad \operatorname{div}_hK-\dd(\operatorname{tr}_hK)=0.
 \label{eq:all-vacuum-constraints}
\end{equation}
We use $K=\tfrac12\partial_t h$ in proper time.  Write
$\mathscr D^+_\Gamma$ for these data with
$\operatorname{tr}_hK>0$, equipped with the finite-dimensional topology
of the invariant tensor coefficients.  No identifications of data by
spatial diffeomorphisms are made in this definition.

\begin{proposition}[Parametrization and integration of expanding vacuum data]
\label{prop:all-expanding-vacuum-data}
Put $H_h=\mathfrak z^{\perp_h}$ and let $n(h)>0$ be determined by
$R(h)=-n(h)^2/2$.  The space $\mathscr D^+_\Gamma$ is parametrized
smoothly and bijectively by
\begin{equation}
 h>0,\qquad s>0,\qquad
 \mathcal A\in\operatorname{Sym}_0(H_h,h),
 \label{eq:all-vacuum-data-parameters}
\end{equation}
where $\operatorname{Sym}_0(H_h,h)$ denotes the two-dimensional space
of trace-free $h$-self-adjoint endomorphisms of $H_h$.  In this
parametrization $L=h^{-1}K$ preserves $\mathfrak z$ and $H_h$, and
\begin{equation}
 L|_{H_h}=\frac{s}{2}\operatorname{Id}+\mathcal A,
 \qquad
 L|_{\mathfrak z}=\alpha\operatorname{Id},\qquad
 \alpha=\frac{n(h)^2-s^2+2|\mathcal A|_h^2}{4s}.
 \label{eq:all-vacuum-data-formula}
\end{equation}
Every maximal homogeneous development of such data is isometric to
a member of \eqref{eq:scalar-family} with
\begin{equation}
 \sigma=0,\qquad \beta,\gamma>0,\qquad
 \beta\gamma=k^2/4.
 \label{eq:all-vacuum-positive-branch}
\end{equation}
Here an invariant change of coframe may replace the lattice by an
isomorphic image lattice.  The locally rotationally symmetric data
are exactly $\mathcal A=0$, and correspond to $\beta=\gamma=k/2$.
\end{proposition}

\begin{proof}
Choose an $h$-orthonormal invariant frame $E_1,E_2,E_3$ with $E_1$
central and
$[E_2,E_3]=nE_1$, where $n=n(h)>0$.  Such a frame exists because
$\mathfrak z$ has dimension one and $H_h$ has dimension two.
Koszul's formula gives
\[
 (\operatorname{div}_hK)(E_1)=0,\qquad
 (\operatorname{div}_hK)(E_2)=nK(E_1,E_3),\qquad
 (\operatorname{div}_hK)(E_3)=-nK(E_1,E_2).
\]
For example, the first term in
$\sum_i(\nabla_{E_i}K)(E_i,E_j)$ arising from
$\nabla_{E_i}E_i$ vanishes, while the two nonzero terms for $j=2$
are $nK(E_1,E_3)/2$ and $nK(E_3,E_1)/2$.
The mean curvature is spatially constant.  Thus the momentum constraint
is equivalent to $K(E_1,E_2)=K(E_1,E_3)=0$, so $L$ preserves the
central line and its orthogonal complement.

Write $L|_{\mathfrak z}=\alpha\operatorname{Id}$ and
$L|_{H_h}=s\operatorname{Id}/2+\mathcal A$, with
$\operatorname{tr}\mathcal A=0$.  The Hamiltonian constraint is
precisely
\begin{equation}
 2\alpha s+\frac{s^2}{2}-|\mathcal A|_h^2=\frac{n^2}{2}.
 \label{eq:all-vacuum-reduced-constraint}
\end{equation}
It excludes $s=0$.  Solving it gives
\eqref{eq:all-vacuum-data-formula} and
\begin{equation}
 \operatorname{tr}_hK=\alpha+s
 =\frac{n^2+3s^2+2|\mathcal A|_h^2}{4s}.
 \label{eq:all-vacuum-expansion-sign}
\end{equation}
Consequently expansion is equivalent to $s>0$.  Conversely every
choice in \eqref{eq:all-vacuum-data-parameters} satisfies both
constraints and has positive mean curvature.  The spaces $H_h$ form
a smooth rank-two vector bundle over the space of positive inner
products, and all operations in
\eqref{eq:all-vacuum-data-formula} are smooth.  The inverse map takes
the horizontal trace and trace-free part of $L$, proving the smooth
parametrization.

An oriented orthogonal change of $E_2,E_3$ diagonalizes
$\mathcal A$ and retains $[E_2,E_3]=nE_1$.
Rescale the dual invariant coframe by a constant so that its sole
nonzero structure equation is
$\dd\sigma^1=-\sigma^2\wedge\sigma^3$.
For example, if $\omega^i$ is dual to $E_i$, take
$\sigma^1=\omega^1/n$, $\sigma^2=\omega^2$ and
$\sigma^3=\omega^3$.
Both initial tensors are diagonal in this coframe.  The diagonal
vacuum equations form an invariant subsystem of the homogeneous
Einstein evolution equations, so uniqueness preserves diagonality.
This also gives a direct version of the simultaneous diagonalization
used in \cite[Lemma~2.4]{Ringstrom2025}.

To integrate that subsystem, set
\[
 h=\sum_{i=1}^3e^{2u_i}(\sigma^i)^2,\qquad
 S=u_1+u_2+u_3,\qquad N=\eta e^S,
\]
where $\eta>0$ is constant and $\dd t=N\dd\tau$.
The equations and the remaining constraint are
\begin{align}
 u_1''&=-\frac{\eta^2}{2}e^{4u_1},&
 u_2''=u_3''&=\frac{\eta^2}{2}e^{4u_1},\notag\\
 (S')^2-\sum_i(u_i')^2-\frac{\eta^2}{2}e^{4u_1}&=0.
 \label{eq:all-vacuum-harmonic-system}
\end{align}
In particular, for constants $k>0$, $\beta$ and $\gamma$,
\begin{equation}
 (u_1')^2+\frac{\eta^2}{4}e^{4u_1}=\frac{k^2}{4},
 \qquad (u_1+u_2)'=\beta,\qquad
 (u_1+u_3)'=\gamma.
 \label{eq:all-vacuum-first-integrals}
\end{equation}
The constant $k$ is strictly positive because $e^{4u_1}>0$.
Equivalently, $y=e^{-2u_1}$ satisfies
$y''=k^2y$ and $(y')^2=k^2y^2-\eta^2$.
Its positive solution has the form
\[
 y=\frac{\eta}{k}\cosh(k(\tau-\tau_c)).
\]
Translate harmonic time by $\tau_c$.  Integration of the other two
first integrals gives positive constants $A,B,C$ such that
\[
 a=A\cosh(k\tau)^{-1/2},\qquad
 b=Be^{\beta\tau}\cosh(k\tau)^{1/2},\qquad
 c=Ce^{\gamma\tau}\cosh(k\tau)^{1/2},
 \qquad A^2=k/\eta.
\]
The lapse is exactly the lapse in \eqref{eq:scalar-family} with
$N_0=\eta ABC$, and
$k=\eta A^2=N_0A/(BC)$.

Writing $v=u_1'$, substitution of
$(u_1',u_2',u_3')=(v,\beta-v,\gamma-v)$ in the constraint gives
\[
 0=2\beta\gamma-2v^2-\frac{\eta^2}{2}e^{4u_1}
   =2\beta\gamma-\frac{k^2}{2}.
\]
Thus $\beta\gamma=k^2/4$.  The two constants have the same sign,
$|\beta+\gamma|\ge k$, and
\[
 N\operatorname{tr}_hK=S'
 =\beta+\gamma+\frac{k}{2}\tanh(k\tau).
\]
Its sign is the common sign of $\beta$ and $\gamma$.  Initial
expansion therefore selects exactly
\eqref{eq:all-vacuum-positive-branch}.
The formulas are smooth for all $\tau\in\mathbb R$.
Their past proper-time length is finite, their future proper-time
length is infinite, and $abc\to0$ at the past endpoint, since
$\beta+\gamma-k/2>0$.
Hence this is the maximal interval of the positive-definite
homogeneous proper-time evolution.

Finally, at any regular slice the difference of the two horizontal
eigenvalues of $L$ is $(\beta-\gamma)/N$.
Thus $\mathcal A=0$ exactly when $\beta=\gamma$, which by the
positive parameter constraint is $\beta=\gamma=k/2$.
In this case rotations of the orthonormal horizontal plane, fixing
the center, preserve the Lie bracket and both initial tensors.
They generate the local rotational symmetry of the development.
Conversely, the positive spatial Ricci eigenspace is precisely
$\mathfrak z$, so every connected isotropy group of $h$ fixes
the central line and acts by rotations on $H_h$.
If $\mathcal A\ne0$, the horizontal eigenvalues differ and their
common stabilizer has no nontrivial connected rotation subgroup.
Thus local rotational symmetry of the initial data is absent.

Every constant invariant coframe used above descends to the original
quotient.  Equivalently, the Lie algebra isomorphism identifying its
normalized brackets with the standard Heisenberg brackets integrates
to a Lie group isomorphism $F:H_3\to H_3$.  It induces a diffeomorphism
\[
 \Gamma\backslash H_3\longrightarrow
 F(\Gamma)\backslash H_3,
\]
and $F(\Gamma)$ is a lattice.  This establishes the claimed
identification without requiring $F$ to preserve $\Gamma$.
\end{proof}

\begin{theorem}[Invariant Bianchi~II horizon-singularity alternative]
\label{thm:vacuum-bianchi-alternative}
\label{cor:open-dense-vacuum-inextendibility}
Fix a lattice $\Gamma<H_3$ and a datum
$(h,K)\in\mathscr D^+_\Gamma$.  Put $L=h^{-1}K$,
$H_h=\mathfrak z^{\perp_h}$, and write
\[
 L|_{H_h}=\frac{s}{2}\operatorname{Id}+\mathcal A,
 \qquad
 \mathcal A\in\operatorname{Sym}_0(H_h,h),
\]
as in Proposition~\ref{prop:all-expanding-vacuum-data}.  Let
\begin{equation}
 n=n(h),\qquad
 \alpha=\frac{n^2-s^2+2|\mathcal A|_h^2}{4s},\qquad
 r=\frac{\sqrt2|\mathcal A|_h}{\sqrt{4\alpha^2+n^2}}.
 \label{eq:invariant-vacuum-anisotropy}
\end{equation}
The maximal homogeneous vacuum development is future timelike and null
geodesically complete.  The following conditions are equivalent:
\begin{enumerate}[label=\textup{(\roman*)}]
\item its past end on $\Gamma\backslash H_3$ admits a continuous
nondegenerate Lorentzian extension;
\item its past end admits an analytic Ricci-flat extension;
\item $\mathcal A=0$;
\item $L|_{H_h}$ is a scalar endomorphism;
\item the datum is locally rotationally symmetric;
\item in past proper time,
\[
 \lim_{t\downarrow0}t^4
 R_{\mu\nu\rho\sigma}R^{\mu\nu\rho\sigma}=0.
\]
\end{enumerate}
When these conditions hold, the one-sided quotient extensions have a
compact nondegenerate Cauchy horizon with closed null generators, and
the universal cover has the analytic bifurcate extension of
Proposition~\ref{prop:heisenberg-bifurcate}.  Otherwise the development
is globally $C^0$-inextendible on both the quotient and its universal
cover, without any field equation or symmetry assumption on a candidate
extension, and
\begin{equation}
 R_{\mu\nu\rho\sigma}R^{\mu\nu\rho\sigma}
 =\mathfrak C(r)t^{-4}+o(t^{-4}),\qquad
 \mathfrak C(r)=
 \frac{32(\sqrt{1+r^2}-1)}{(2\sqrt{1+r^2}-1)^3}>0.
 \label{eq:curvature-residue}
\end{equation}
The extendible locus is a smooth codimension-two submanifold of the
nine-dimensional space $\mathscr D^+_\Gamma$; its complement is open
and dense.  No quotient by spatial diffeomorphisms is taken in this
dimension count.
\end{theorem}

\begin{proof}
In the parametrization of
Proposition~\ref{prop:all-expanding-vacuum-data}, the positive inner
product has six parameters, $s>0$ has one, and
$\mathcal A\in\operatorname{Sym}_0(H_h,h)$ has two.
The locus $\mathcal A=0$ is the zero section in this rank-two
vector bundle over the seven-dimensional space of $(h,s)$.
It is a smooth submanifold of codimension two with empty interior.
Its complement is therefore open and dense.
This topology agrees with every $C^r$ topology restricted to
invariant tensor data on the compact quotient, including $C^\infty$.

If $\mathcal A\ne0$, then $\beta\ne\gamma$ in the positive
vacuum family, so one of $\beta,\gamma$ is less than $k/2$.
Proposition~\ref{prop:einstein-scalar} proves global
$C^0$-inextendibility, including on the image lattice described
above.  Pulling back by the quotient isometry gives the assertion
for the fixed original lattice.  These globally hyperbolic
developments have no proper smooth extension and are therefore maximal
globally hyperbolic developments, in the sense of
\cite{ChoquetBruhatGeroch1969}.
If $\mathcal A=0$, the preceding proposition identifies the
symmetric positive Taub branch, and
Propositions~\ref{prop:heisenberg-bifurcate} and
\ref{prop:heisenberg-vacuum-dichotomy} give the stated extensions and horizons.
Theorem~\ref{thm:homogeneous-future} gives future causal completeness.
The preceding proposition proves
$\mathcal A=0\Longleftrightarrow L|_{H_h}$ is scalar
$\Longleftrightarrow$ local rotational symmetry.  Put
$u=\sqrt{1+r^2}$ and
\[
 p_1=\frac1{2u-1},\qquad
 p_-=\frac{u-r-1}{2u-1},\qquad
 p_+=\frac{u+r-1}{2u-1}.
\]
The differentiated Kasner asymptotics of the exact family give
\[
 \lim_{t\downarrow0}t^4|\operatorname{Riem}(g)|^2
 =4\sum_i p_i^2(p_i-1)^2+4\sum_{i<j}p_i^2p_j^2
 =-16p_1p_-p_+=\mathfrak C(r).
\]
This vanishes exactly when $r=0$, equivalently $\mathcal A=0$.
These observations prove all the asserted equivalences without using
curvature blow-up as the $C^0$ obstruction.
\end{proof}

The openness and density just proved concern invariant tensor data
on a fixed underlying nilmanifold.  They do not assert openness or
density in the full space of inhomogeneous Einstein initial data.

\begin{corollary}[Anisotropy and the exceptional vacuum horizon]
\label{cor:quantitative-vacuum-horizon}
For a datum $(h,s,\mathcal A)$ in
Proposition~\ref{prop:all-expanding-vacuum-data}, let $r$ be given by
\eqref{eq:invariant-vacuum-anisotropy} and put $u=\sqrt{1+r^2}$.
The number $r$ is independent of the regular slice on which it is
computed.  The unordered past Kasner exponents are
\begin{equation}
 p_1=\frac1{2u-1},\qquad
 p_-=\frac{u-r-1}{2u-1},\qquad
 p_+=\frac{u+r-1}{2u-1},
 \label{eq:regular-data-past-exponents}
\end{equation}
where $p_1$ is central.  If $r>0$, then $p_-<0$, and the invariant
horizontal scale factor in that eigendirection satisfies
\[
 \lim_{t\downarrow0}
 \frac{\log a_-(t)}{\log(1/t)}=-p_->0.
\]
The Kretschmann scalar has the limit
\begin{equation}
 \lim_{t\downarrow0}
 t^4 R_{\mu\nu\rho\sigma}R^{\mu\nu\rho\sigma}
 =\frac{32(u-1)}{(2u-1)^3}.
 \label{eq:vacuum-anisotropy-curvature}
\end{equation}
For fixed regular $h$ and $s$, as $\mathcal A\to0$,
\begin{align}
 -p_-&=\frac{2\sqrt2s}{n^2+s^2}|\mathcal A|_h
          +O(|\mathcal A|_h^2),\label{eq:vacuum-horizon-linear-rate}\\
 \frac{32(u-1)}{(2u-1)^3}
 &=\frac{128s^2}{(n^2+s^2)^2}|\mathcal A|_h^2
          +O(|\mathcal A|_h^4).
 \label{eq:vacuum-horizon-quadratic-curvature}
\end{align}
The error bounds are uniform when $(h,s)$ ranges over a compact subset
of the regular data parameters.
Thus arbitrarily small nonzero invariant horizontal anisotropy replaces
the exceptional analytic past horizon by a $C^0$-inextendible end.
\end{corollary}

\begin{proof}
In the harmonic-time integration of
Proposition~\ref{prop:all-expanding-vacuum-data}, let $N$ denote the
lapse at the selected regular slice.  The central logarithmic
derivative is $u_1'=N\alpha$, and $Nn=\eta e^{2u_1}$.
The first integral
\eqref{eq:all-vacuum-first-integrals} therefore gives
\[
 \frac{k}{N}=\sqrt{4\alpha^2+n^2}.
\]
The horizontal eigenvalue difference has absolute value
$\sqrt2|\mathcal A|_h$, whereas the same difference in harmonic
coordinates is $|\beta-\gamma|/N$.  Hence
\begin{equation}
 r=\frac{|\beta-\gamma|}{k}.
 \label{eq:invariant-vacuum-first-integral}
\end{equation}
This proves slice independence.  Since $\beta\gamma=k^2/4$ and both
constants are positive,
$(\beta+\gamma)/k=\sqrt{1+r^2}=u$.
Equation~\eqref{eq:scalar-kasner-exponents} now proves
\eqref{eq:regular-data-past-exponents}.  For $r>0$,
$\sqrt{1+r^2}<1+r$, so $p_-<0$.  The scale-factor assertion follows
from \eqref{eq:scalar-kasner-asymptotics}; the continuous
inextendibility follows from
Theorem~\ref{thm:vacuum-bianchi-alternative}.

For completeness, the exact formulas give differentiated asymptotics
\[
 a_i(t)=c_i t^{p_i}(1+O(t^{4p_1})),\qquad
 \frac{\dot a_i}{a_i}=\frac{p_i}{t}+O(t^{4p_1-1}),\qquad
 \frac{\dd}{\dd t}\left(\frac{\dot a_i}{a_i}\right)
 =-\frac{p_i}{t^2}+O(t^{4p_1-2}),
\]
with $p_1>0$.  The spatial bracket coefficient satisfies
$n(t)=O(t^{2p_1-1})$.  In the orthonormal-frame curvature equations,
the purely spatial intrinsic curvature terms are $O(n(t)^2)$,
and the mixed terms are
$O(n(t)\max_i|\dot a_i/a_i|)$.  Multiplication by $t^2$ makes
both tend to zero.  The remaining leading terms are the Kasner
components.  Consequently
\[
 \lim_{t\downarrow0}t^4
 R_{\mu\nu\rho\sigma}R^{\mu\nu\rho\sigma}
 =4\sum_i p_i^2(p_i-1)^2
   +4\sum_{i<j}p_i^2p_j^2
 =-16p_1p_-p_+.
\]
The last equality follows from
$\sum_i p_i=\sum_i p_i^2=1$.
Since
$p_1p_-p_+=-2(u-1)/(2u-1)^3$,
this proves \eqref{eq:vacuum-anisotropy-curvature}.

Finally, with $h,s$ fixed,
\[
 \alpha=\frac{n^2-s^2}{4s}+O(|\mathcal A|_h^2),\qquad
 \sqrt{4\alpha^2+n^2}
 =\frac{n^2+s^2}{2s}+O(|\mathcal A|_h^2).
\]
Thus
$r=2\sqrt2s|\mathcal A|_h/(n^2+s^2)
+O(|\mathcal A|_h^3)$.
The elementary expansions
$-p_-=r+O(r^2)$ and
$32(\sqrt{1+r^2}-1)/(2\sqrt{1+r^2}-1)^3
=16r^2+O(r^4)$ give
\eqref{eq:vacuum-horizon-linear-rate} and
\eqref{eq:vacuum-horizon-quadratic-curvature}.
\end{proof}

These estimates describe perturbations in the finite-dimensional space
of invariant vacuum data at a regular slice.  The continuous
inextendibility is supplied by the negative horizontal exponent;
curvature blow-up alone would give a weaker obstruction to extension.

\begin{proposition}[Uniform stretching near the exceptional horizon]
\label{prop:uniform-vacuum-horizon}
Let $\mathcal K$ be a compact subset of the regular parameter space
$\{(h,s):h>0,\ s>0\}$ in
Proposition~\ref{prop:all-expanding-vacuum-data}.
There are $\varepsilon,C,\eta>0$ with the following property.
Consider data $(h,s,\mathcal A)$ with $(h,s)\in\mathcal K$ and
$|\mathcal A|_h\le\varepsilon$, and choose past proper time $t$ to
vanish at the singularity or horizon.  Let $t_{\mathrm{ref}}>0$ be
the time of the selected regular slice.  For $\mathcal A\ne0$, choose
$h$-unit horizontal eigenvectors $X_-,X_+$ of $\mathcal A$ for its
negative and positive eigenvalues, respectively, and extend them as
invariant vector fields.  Set
\[
 \ell_\pm(t)=\sqrt{h_t(X_\pm,X_\pm)},\qquad
 B_{\mathrm H}(h,s)=\frac{n(h)}{\sqrt{n(h)^2+s^2}}.
\]
At $\mathcal A=0$, any $h$-orthonormal horizontal pair may be used.
The quantity $B_{\mathrm H}$ is the limiting horizontal length at
the horizon of the development of $(h,s,0)$.  With $r$ and $p_\pm$
as in~\eqref{eq:invariant-vacuum-anisotropy} and
\eqref{eq:regular-data-past-exponents}, one has
\begin{equation}
 \left|\log\frac{\ell_\pm(t)}{B_{\mathrm H}}
       -p_\pm\log\frac{t}{t_{\mathrm{ref}}}\right|
 \le C\left(r+\left(\frac{t}{t_{\mathrm{ref}}}\right)^\eta\right),
 \qquad 0<t\le t_{\mathrm{ref}}.
 \label{eq:uniform-horizon-stretching}
\end{equation}
The constants are independent of the direction of $\mathcal A$.
Moreover, for every $0<\lambda_0<\lambda_1<\infty$, there is a
constant $C_{\lambda_0,\lambda_1}$ such that, when $r>0$,
\begin{align}
 \sup_{\lambda_0\le\lambda\le\lambda_1}
 \left(
 \left|\frac{\ell_-(t_{\mathrm{ref}}e^{-\lambda/r})}
                  {B_{\mathrm H}}-e^\lambda\right|
 +\left|\frac{\ell_+(t_{\mathrm{ref}}e^{-\lambda/r})}
                  {B_{\mathrm H}}-e^{-\lambda}\right|
 \right)
 \le C_{\lambda_0,\lambda_1}
       \left(r+e^{-\eta\lambda_0/r}\right).
 \label{eq:uniform-horizon-crossover}
\end{align}
Thus the joint limit of vanishing invariant anisotropy and approach
to the past end has a nontrivial stretching profile on the scale
$\log(t_{\mathrm{ref}}/t)\asymp r^{-1}$.
\end{proposition}

The exact lapse integral and the uniform parameter estimates proving
this proposition are given in Appendix~\ref{app:uniform-vacuum-proof}.

The normalization in this proposition compares directional lengths
with the horizon length for the projected symmetric datum at the
same regular $h$ and $s$.  It does not compare the diameters of entire
compact slices.  The two limiting factors are $e^\lambda$ and
$e^{-\lambda}$, so their ratio tends to $e^{2\lambda}$.
The estimate resolves the joint proper-time and anisotropy limit
within invariant vacuum data; it makes no assertion about
inhomogeneous perturbations.

\begin{corollary}[Sharp curvature residue and horizon crossover]
\label{cor:curvature-delay-law}
For the function $\mathfrak C$ in \eqref{eq:curvature-residue},
\begin{equation}
 0\le\mathfrak C(r)\le\frac{64}{27}.
 \label{eq:sharp-curvature-residue}
\end{equation}
For finite $r$, equality on the left holds exactly at $r=0$, while
equality on the right holds exactly at $r=3/4$.  The latter value gives,
up to permutation, the past Kasner exponents
$(-1/3,2/3,2/3)$.  Moreover, let
$(h_j,s_j,\mathcal A_j)$ be data with $(h_j,s_j)$ in a fixed compact
parameter set, $0<|\mathcal A_j|_{h_j}\to0$, and write
$r_j=r(h_j,s_j,\mathcal A_j)$,
$\mathfrak C_j=\mathfrak C(r_j)$.  If
$0<t_j\le t_{{\rm ref},j}$ and
\begin{equation}
 \frac{\sqrt{\mathfrak C_j}}4
 \log\frac{t_{{\rm ref},j}}{t_j}\longrightarrow\lambda
 \in(0,\infty),
 \label{eq:curvature-delay-scaling}
\end{equation}
then
\begin{equation}
 \frac{\ell_{-,j}(t_j)}{B_{\rm H}(h_j,s_j)}\longrightarrow e^\lambda,
 \qquad
 \frac{\ell_{+,j}(t_j)}{B_{\rm H}(h_j,s_j)}\longrightarrow e^{-\lambda}.
 \label{eq:curvature-delay-profile}
\end{equation}
Thus an order-one anisotropic deformation occurs at the
nonperturbative proper-time scale
\begin{equation}
 \frac{t}{t_{\rm ref}}
 =\exp\!\left[-\frac{4\lambda}{\sqrt{\mathfrak C}}
          +o(\mathfrak C^{-1/2})\right].
 \label{eq:curvature-delay-time}
\end{equation}
\end{corollary}

\begin{proof}
Set $u=\sqrt{1+r^2}$.  Then
\[
 \mathfrak C=\frac{32(u-1)}{(2u-1)^3},\qquad
 \frac{\dd\mathfrak C}{\dd u}
 =\frac{32(5-4u)}{(2u-1)^4}.
\]
This proves \eqref{eq:sharp-curvature-residue}; its maximum occurs at
$u=5/4$, hence at $r=3/4$, and substitution in
\eqref{eq:regular-data-past-exponents} gives the stated Kasner triple.
As $r\downarrow0$,
\[
 \mathfrak C(r)=16r^2+O(r^4),\qquad
 \frac{\sqrt{\mathfrak C(r)}}4=r+O(r^3).
\]
Condition \eqref{eq:curvature-delay-scaling} is therefore precisely the
double-scaling limit in
Proposition~\ref{prop:uniform-vacuum-horizon}, which proves
\eqref{eq:curvature-delay-profile} and
\eqref{eq:curvature-delay-time}.
\end{proof}

\section{Discussion}

The expanding-end estimate proves the G\"odeke-Rendall conjecture and
places it in a broader dimension-explicit principle.  It applies without
an imposed oscillation bound, without diagonalization, and also to
homogeneous spaces with isotropy and complete Riemannian scalar targets.
The integrated criterion separates the geometric estimate from the
matter dynamics: a positive potential floor yields completeness in every
dimension, while a sufficiently slow decay along the actual trajectory
still forces infinite affine length.  For future-global general matter,
the same argument identifies an explicit normal energy-condition window
through the critical ninth spatial dimension.  It is a completeness
statement, not a cosmic no-hair theorem; the stronger asymptotic
conclusions in \cite{Rendall2004Scalar,Ringstrom2025} use additional
hypotheses.

At the finite end, the main point is that quotient collapse need not hide
growth on the universal cover.  Boundary localization, transporter germs
and lifted spacelike sections transfer such growth to an arbitrary
candidate $C^0$ extension without lifting that extension.  For two-step
groups the central correction is intrinsic and variationally optimal,
and its exact tensorization permits arbitrary mixed discrete subgroups of
finite products.  The Einstein examples on $H_3\times H_3$ show that
distributed correction can strictly improve area-direction allocation
inside the vacuum and canonical scalar equations.  The numerical bound
$1/2$ remains a sufficient timelike-interpolation threshold, not a
necessary characterization of filling or extendibility.

The Bianchi~II application joins both time directions in one fixed-lattice
statement.  The classical codimension-two non-genericity of homogeneous
horizons \cite{Siklos1978,ChruscielRendall1995} becomes an intrinsic
equivalence between local rotational symmetry, vanishing anisotropy,
vanishing renormalized curvature residue, analytic horizon extension and
past $C^0$-extendibility.  Every datum in the open dense complement gives
a future-complete but past globally $C^0$-inextendible spacetime.
Proposition~\ref{prop:uniform-vacuum-horizon} and
Corollary~\ref{cor:curvature-delay-law} quantify the transition: a shear
of size $r$ changes directional lengths by order one only when
$\log(t_{\mathrm{ref}}/t)\asymp r^{-1}$, equivalently at a scale
$t/t_{\mathrm{ref}}\asymp\exp(-4\lambda/\sqrt{\mathfrak C})$.  This
nonuniform limit explains how a regular horizon can be destroyed by
arbitrarily small invariant anisotropy although its curvature residue is
only quadratic in that anisotropy.  The lengths here are local
directional observables, not compact quotient diameters.

Either one-sided analytic extension has chronology-violating region
exactly $\{s<0\}$.  The sharp contractible-loop horizontal threshold
grows as $|s|^{-1/2}$ near the horizon, while compact quotients also have
central timelike loops of zero horizontal length.  These quantitative
claims concern the time-dependent NUT extension and do not duplicate the
qualitative Lorentz-Heisenberg causality theory of
\cite{GuediriLorentz2003,GuediriGlobal2003,GuediriCriterion2008}.  No
uniqueness of continuous extensions is asserted
\cite{SbierskiUniqueness}.

The scope has definite boundaries.  Genericity is only in the
nine-dimensional invariant constraint-data space on a fixed compact
quotient, not in the full inhomogeneous Cauchy-data space or quotient
moduli.  The all-positive-exponent scalar regime remains outside the
two-step obstruction.  Quiescent formation and stability
\cite{FournodavlosRodnianskiSpeck2023,OudeGroenigerPetersenRingstrom,
FrancoGrisalesRingstrom2026,FrancoGrisalesLocalized2026}
do not settle that continuous-metric question.  Exact local isometries
are essential to the present localization argument, and higher-step or
perfect homogeneous algebras require different filling mechanisms.  The
FLRW results \cite{SbierskiFLRW,LingFLRW} concern singularities without
particle horizons, while Le's volume-distance criterion
\cite{LeVDR} imposes additional conditions on candidate extensions; both
directions remain complementary to the local-homogeneity mechanism used
here.

\clearpage

\appendix
\section{Exact filling constants and power-law applications}
\label{sec:applications}

\subsection{The exact scalar-block filling constant}

Choose fixed inner products on $V$ and $Z$ and consider
\begin{equation}
 h_t=b(t)^2|\dd v|_V^2+
 a(t)^2|\Theta+L_t\dd v|_Z^2.
 \label{eq:scalar-block-metric}
\end{equation}
No boundedness of $L_t$ is used in the filling estimate.  Put
\begin{equation}
 \dd\mu=\frac Na\,\dd t,
 \qquad q(t)=\frac{a(t)}{b(t)},
 \qquad
 B_0=\sup_{|v|_V=|w|_V=1}|\mathsf b(v,w)|_Z.
 \label{eq:scalar-data}
\end{equation}
For $T\in(t_-,t_+)$ define
\begin{equation}
 R(T)=\int_{t_-}^T\frac{N(t)a(t)}{b(t)^2}\,\dd t.
 \label{eq:area-time}
\end{equation}

\begin{proposition}[Exact weighted bracket-area constant]
\label{prop:exact-area}
For \eqref{eq:scalar-block-metric}, including arbitrary $L_t$,
\begin{equation}
 \kappa_{\rm an}(T)=\frac{B_0}{4\pi}R(T),
 \label{eq:exact-area-constant}
\end{equation}
with value $+\infty$ if $R(T)=\infty$.
In particular $B_0R(T)\le2\pi$ implies future one-connectedness.
The equality is sharp as an evaluation of the variational constant,
not as a necessary condition for future one-connectedness.
\end{proposition}

\begin{proof}
Fix $I=[A,B]\Subset(t_-,T)$ and change variable by
$u(t)=\int_A^t Na/b^2\,\dd r$, with interval length $\ell=u(B)$.
The bracket area is invariant under this change, whereas
\[
 \mathcal E_I(\xi)=\int_A^B\frac{b^2}{Na}|\dot\xi|^2\,\dd t
 =\int_0^\ell|\xi'(u)|^2\,\dd u.
\]
Since $A_t=a(t)^2I$, $J_I=\mathcal E_I/|Q_I|$ exactly.
For a unit central vector $z$, define the skew-adjoint map $J_z$ by
$\langle J_zv,w\rangle=\langle z,\mathsf b(v,w)\rangle$.
Then $\sup_{|z|=1}\|J_z\|=B_0$.  Subtracting the mean from the closed
loop $\xi$ does not change its area.  The periodic Wirtinger inequality,
obtained term by term from the Fourier series on $[0,\ell]$, gives
\[
 \int_0^\ell|\xi-\bar\xi|^2\,\dd u
 \le\frac{\ell^2}{4\pi^2}\int_0^\ell|\xi'|^2\,\dd u.
\]
Cauchy-Schwarz therefore yields
\[
 |\langle z,Q_I\rangle|
 =\frac12\left|\int_0^\ell
 \langle J_z(\xi-\bar\xi),\xi'\rangle\,\dd u\right|
 \le\frac{B_0\ell}{4\pi}\mathcal E_I(\xi).
\]
Taking the supremum over $z$ proves the upper bound.

Choose orthonormal $e_1,e_2\in V$ with
$|\mathsf b(e_1,e_2)|=B_0$; a maximizing pair can be chosen orthogonal
by antisymmetry and orthogonal projection.  For $r>0$ the based circle
\[
 \xi(u)=r\bigl((\cos(2\pi u/\ell)-1)e_1+
 \sin(2\pi u/\ell)e_2\bigr)
\]
has $Q_I=\pi r^2\mathsf b(e_1,e_2)$ and
$\mathcal E_I=4\pi^2r^2/\ell$.  Thus the bound is attained on each
compact interval.  Exhaustion of $(t_-,T)$ gives
\eqref{eq:exact-area-constant}, also when the integral is infinite.
\end{proof}

\subsection{Separate central components and anisotropic powers}
\label{subsec:central-component-filling}

The direction-preserving correction used to bound
$\kappa_{\rm an}$ remains parallel to the total bracket area.
When the central metric has different time scales, distributing its
components separately gives another computable upper bound for
$\kappa_{\rm opt}$.
Fix a basis $z_1,\ldots,z_r$ of $Z$ and smooth positive functions
$a_\alpha(t)$ such that
\begin{equation}
 A_t\left(\sum_{\alpha=1}^r u_\alpha z_\alpha,
           \sum_{\alpha=1}^r u_\alpha z_\alpha\right)
 \le \sum_{\alpha=1}^r a_\alpha(t)^2u_\alpha^2.
 \label{eq:central-diagonal-majorant}
\end{equation}
This inequality does not require simultaneous diagonalization: for an
arbitrary central metric, one may take
$a_\alpha(t)^2=rA_t(z_\alpha,z_\alpha)$, by Cauchy-Schwarz.
Write $Q_I(\xi)=\sum_\alpha Q_{I,\alpha}(\xi)z_\alpha$ and set
\begin{align}
 E_{I,\alpha}(\xi)
  &=\int_I\frac{H_t(\dot\xi,\dot\xi)}{N(t)a_\alpha(t)}\,\dd t,
 \notag\\
 c_\alpha(T)
  &=\sup_{I\Subset(t_-,T)}\ 
    \sup_{0\ne\xi\in W^{1,2}_0(I;V)}
       \frac{|Q_{I,\alpha}(\xi)|}{E_{I,\alpha}(\xi)},
 &
 \kappa_{\rm cen}(T)&=
       \left(\sum_{\alpha=1}^r c_\alpha(T)^2\right)^{1/2}.
 \label{eq:central-component-constant}
\end{align}
The constants depend on the chosen central basis and majorant.  They
provide another sufficient filling test; no general ordering between
$\kappa_{\rm cen}$ and $\kappa_{\rm an}$ is asserted.

\begin{proposition}[Filling by separate central corrections]
\label{prop:central-component-filling}
If \eqref{eq:central-diagonal-majorant} holds and
$\kappa_{\rm cen}(T)\le1/2$, then $(t_-,T)\times\mathsf N$ is future
one-connected, with no additional restriction on $L_t$.
Moreover, $\kappa_{\rm opt}(T)\le\kappa_{\rm cen}(T)$, so this test
implies the filling hypothesis of Theorem~\ref{thm:two-step-obstruction}.
\end{proposition}

\begin{proof}
Use the notation of the proof of Proposition~\ref{thm:schur-filling}, with
$\xi=v_1-v_0$ and $e_\xi=H_t(\dot\xi,\dot\xi)/N^2$.  If $\xi=0$,
the affine interpolation of the body controls already fixes both
endpoints and is timelike.  Otherwise every $E_{I,\alpha}(\xi)$ is
strictly positive.  Replace the correction in
\eqref{eq:distributed-correction} by
\begin{equation}
 c_\theta(t)=\theta(1-\theta)
 \sum_{\alpha=1}^r
  \frac{Q_{I,\alpha}(\xi)}{E_{I,\alpha}(\xi)}
  \frac{H_t(\dot\xi,\dot\xi)}{N(t)a_\alpha(t)}z_\alpha.
 \label{eq:componentwise-central-correction}
\end{equation}
Its integral is $\theta(1-\theta)Q_I(\xi)$, so
\eqref{eq:bracket-defect} gives the terminal endpoint exactly as before.
By \eqref{eq:central-diagonal-majorant},
\begin{align*}
 \left|\frac{c_\theta(t)}{N(t)}\right|_{A_t}
 &\le\theta(1-\theta)e_\xi(t)
 \left(\sum_{\alpha=1}^r
       \frac{|Q_{I,\alpha}(\xi)|^2}{E_{I,\alpha}(\xi)^2}
       \right)^{1/2}\\
 &\le\theta(1-\theta)e_\xi(t)\kappa_{\rm cen}(T).
\end{align*}
Removing the factor $\theta(1-\theta)$ gives a central control
with integral $Q_I(\xi)$ and norm at most
$f_\xi(t)(\sum_\alpha |Q_{I,\alpha}|^2/E_{I,\alpha}^2)^{1/2}$.
Lemma~\ref{lem:central-allocation} and the supremum over $I,\xi$
therefore give $\kappa_{\rm opt}\le\kappa_{\rm cen}$.
Lemma~\ref{lem:convexity} proves timelikeness.  The correction is
piecewise continuous in $t$, is continuous in $\theta$, and respects the
one-sided tangents on the common subdivision; integration therefore
gives a fixed-endpoint timelike homotopy of the required regularity.
The proof of Theorem~\ref{thm:two-step-obstruction} uses its filling
hypothesis only to obtain future one-connectedness on the cover.  The
same proof consequently applies with the present hypothesis.
\end{proof}

The separate central estimates give a criterion for anisotropic powers
on every two-step algebra.  Choose orthonormal auxiliary bases
$e_1,\ldots,e_m$ of $V$ and $z_1,\ldots,z_r$ of $Z$, and write
\[
 \mathsf b(e_i,e_j)=\sum_{\alpha=1}^r b_{ij}^{\alpha}z_\alpha.
\]
Let $B_i,C_\alpha>0$, and define Schur data in proper time by
\begin{equation}
 H_t^0=\operatorname{diag}(B_i^2t^{2p_i})_{i=1}^m,
 \qquad
 A_t^0=\operatorname{diag}(C_\alpha^2t^{2q_\alpha})_{\alpha=1}^r,
 \qquad L_t^0=L_t.
 \label{eq:fully-anisotropic-schur-powers}
\end{equation}
Denote the associated invariant metric by $h_t^0$.

\begin{corollary}[Anisotropic powers on arbitrary two-step groups]
\label{cor:all-two-step-anisotropic-powers}
Suppose that $L_t$ is smooth and, for some $r_i\ge0$,
\begin{equation}
 \|L_te_i\|\le K_i t^{-r_i},\qquad
 q_\alpha<1\ \text{for all }\alpha,\qquad
 p_i+r_i<1\ \text{for all }i.
 \label{eq:anisotropic-power-cone}
\end{equation}
Assume the bracket inequalities
\begin{equation}
 \chi_{ij}^{\alpha}:=1+q_\alpha-p_i-p_j>0
 \qquad\text{whenever }b_{ij}^{\alpha}\ne0,
 \label{eq:bracketwise-power-conditions}
\end{equation}
and at least one of the following growth alternatives:
\begin{equation}
\begin{aligned}
 &\text{\textup{(i)}}\quad \min_i p_i<0;\\
 &\text{\textup{(ii)}}\quad
       \sum_{i=1}^m p_i+\sum_{\alpha=1}^r q_\alpha<0;\\
 &\text{\textup{(iii)}}\quad
       [\mathfrak n,\mathfrak n]
         =\operatorname{span}\{z_1,\ldots,z_{r_c}\}
       \ \text{for some }0\le r_c<r,
       \quad q_\alpha<0\ \text{for some }\alpha>r_c.
\end{aligned}
\label{eq:anisotropic-power-growth}
\end{equation}
The requirement that the central basis be adapted to the commutator
applies only to alternative \textup{(iii)}.
Then $-\dd t^2+h_t^0$ is past $C^0$-inextendible on the simply
connected group and on every quotient by a discrete subgroup.
The same conclusions hold for every smooth invariant metric
$g=-\dd t^2+h_t$ satisfying, on a past tail,
\begin{equation}
 c_0h_t^0\le h_t\le C_0h_t^0,
 \qquad 0<c_0\le C_0<\infty.
 \label{eq:all-two-step-power-comparison}
\end{equation}
Independent future timelike geodesic completeness gives global
$C^0$-inextendibility.
\end{corollary}

\begin{proof}
First consider $h_t^0$, and take $a_\alpha=C_\alpha t^{q_\alpha}$.
Fix $I=[a,b]\Subset(0,T)$ and extend a based loop $\xi$ by zero to
$(0,T)$.  For each central index $\alpha$, put
\[
 f_i^\alpha(t)=\frac{C_\alpha}{B_i^2}t^{q_\alpha-2p_i},
 \qquad D_i^\alpha=1+q_\alpha-2p_i,
 \qquad u_i(t)=\frac{\dot\xi_i(t)}{\sqrt{f_i^\alpha(t)}}.
\]
Then $E_{I,\alpha}=\sum_i\|u_i\|_{L^2(0,T)}^2$ and, since the
loop has equal endpoints,
\[
 Q_{I,\alpha}=\sum_{i<j}b_{ij}^{\alpha}
                  \int_0^T\xi_i\dot\xi_j\,\dd t,
 \qquad
 \int_0^T\xi_i\dot\xi_j\,\dd t
       =-\int_0^T\xi_j\dot\xi_i\,\dd t.
\]
For a pair with $b_{ij}^{\alpha}\ne0$, one has
$D_i^\alpha+D_j^\alpha=2\chi_{ij}^{\alpha}>0$.  Choose as the inner
variable the index with the larger $D$, denoted $D_{ij}^{\alpha}$;
this number is positive.  Cauchy-Schwarz on the time triangle gives
\begin{align*}
 \left|\int_0^T\xi_i\dot\xi_j\,\dd t\right|
 &\le V_{ij}^{\alpha}(T)\|u_i\|_2\|u_j\|_2,\\
 V_{ij}^{\alpha}(T)
 &=\frac{C_\alpha}{B_iB_j}
   \frac{T^{\chi_{ij}^{\alpha}}}
        {\sqrt{2\chi_{ij}^{\alpha}D_{ij}^{\alpha}}},
 \qquad
 D_{ij}^{\alpha}=\max\{D_i^\alpha,D_j^\alpha\}.
\end{align*}
Indeed, if the inner index is $i$, the square of the triangle norm is
\[
 \int_0^T f_j^\alpha(t)\int_0^t f_i^\alpha(s)\,\dd s\,\dd t
 =\frac{C_\alpha^2}{B_i^2B_j^2}
   \frac{T^{2\chi_{ij}^{\alpha}}}
        {2\chi_{ij}^{\alpha}D_i^\alpha}.
\]
The reversed area identity treats the other ordering.  This argument
does not require both $D_i^\alpha$ and $D_j^\alpha$ to be positive.

Let $S_\alpha(T)$ be the symmetric $m$ by $m$ matrix with zero
diagonal and off-diagonal entries
\[
 (S_\alpha(T))_{ij}=
 \begin{cases}
 |b_{ij}^{\alpha}|V_{ij}^{\alpha}(T),&b_{ij}^{\alpha}\ne0,\\
 0,&b_{ij}^{\alpha}=0.
 \end{cases}
\]
Writing $s_i=\|u_i\|_2$ yields
\[
 |Q_{I,\alpha}|
 \le\frac12s^{\mathsf T}S_\alpha(T)s
 \le\frac12\lambda_{\max}(S_\alpha(T))E_{I,\alpha}.
\]
Taking the suprema over $I$ and $\xi$ therefore gives the explicit
bound
\begin{equation}
 \kappa_{\rm cen}^0(T)
 \le\frac12\left(\sum_{\alpha=1}^r
       \lambda_{\max}(S_\alpha(T))^2\right)^{1/2}
 \longrightarrow0\qquad(T\downarrow0).
 \label{eq:anisotropic-bracket-matrix-bound}
\end{equation}
Thus Proposition~\ref{prop:central-component-filling} supplies the
required timelike filling on a sufficiently early tail.

The trace formula in Proposition~\ref{prop:trace-cone} gives
\[
 D_{h^0}(t)=
 \sum_i B_i^{-2}t^{-2p_i}
 +\sum_\alpha C_\alpha^{-2}t^{-2q_\alpha}
 +\sum_i B_i^{-2}t^{-2p_i}\|L_te_i\|^2.
\]
Condition \eqref{eq:anisotropic-power-cone} makes
$\sqrt{D_{h^0}}$ integrable at zero, hence gives finite causal width.
Under the first alternative in \eqref{eq:anisotropic-power-growth},
the largest horizontal eigenvalue tends to infinity.  Under the second,
the invariant volume density
\[
 \left(\prod_i B_i\right)\left(\prod_\alpha C_\alpha\right)
 t^{\sum_i p_i+\sum_\alpha q_\alpha}
\]
diverges.  In alternative \textup{(iii)}, pass to the full abelianization
$\mathfrak a=\mathfrak n/[\mathfrak n,\mathfrak n]$ and its quotient
metric $\overline h_t^0$.  For the specified index $\alpha>r_c$, the
class $\overline z_\alpha$ is nonzero.  Since $A_t^0$ is diagonal in
the chosen central basis and a pure central vector has no horizontal
component, one has, for arbitrary $L_t$,
\[
 \overline h_t^0(\overline z_\alpha,\overline z_\alpha)
 =\inf_{w\in[\mathfrak n,\mathfrak n]}
       h_t^0(z_\alpha+w,z_\alpha+w)
 =C_\alpha^2t^{2q_\alpha}\longrightarrow\infty.
\]
Thus $\lambda_{\max}(\overline h_t^0)\to\infty$.
This supplies the first growth alternative of
Theorem~\ref{thm:two-step-obstruction}.

The metric-size obstruction, with its filling hypothesis
replaced as in Proposition~\ref{prop:central-component-filling}, proves
the conclusion on the cover and all the stated quotients.

For \eqref{eq:all-two-step-power-comparison}, restriction to $Z$ and
the quotient characterization \eqref{eq:quotient-characterization}
give $A_t\le C_0A_t^0$ and $H_t\ge c_0H_t^0$.
Choose $a_\alpha=\sqrt{C_0}C_\alpha t^{q_\alpha}$ in
\eqref{eq:central-diagonal-majorant}.  Then
\[
 E_{I,\alpha}\ge\frac{c_0}{\sqrt{C_0}}E_{I,\alpha}^0,
 \qquad
 \kappa_{\rm cen}(T)
 \le\frac{\sqrt{C_0}}{c_0}\kappa_{\rm cen}^0(T)\longrightarrow0.
\]
Also $m_h\ge\sqrt{c_0}m_{h^0}$,
$\lambda_{\max}(H_t)\ge c_0\lambda_{\max}(H_t^0)$, and the
volume density is at least $c_0^{(m+r)/2}$ times that of $h_t^0$.
The cone and growth conditions are therefore preserved.  The same
obstruction proves the comparison statement, and future timelike
completeness excludes a future extension boundary.
For alternative \textup{(iii)}, the comparison also passes to the full
abelianization: taking the infimum over commutator translates in
$h_t\ge c_0h_t^0$ gives
$\overline h_t\ge c_0\overline h_t^0$.
Its diverging eigenvalue is therefore preserved, and
Theorem~\ref{thm:two-step-obstruction} applies with the same replacement
of the filling test.

\end{proof}

The bracket inequalities have a direct geometric interpretation.  In
the frame orthonormal for the Schur metric $h_t^0$, the horizontal
vectors and central vectors may be written
\[
 E_i(t)=B_i^{-1}t^{-p_i}\bigl(s(e_i)-L_te_i\bigr),\qquad
 Z_\alpha(t)=C_\alpha^{-1}t^{-q_\alpha}z_\alpha.
\]
Here the subtraction in $E_i$ is central, so the spatial brackets at
fixed time are independent of $L_t$ and satisfy
\[
 [E_i(t),E_j(t)]
 =\sum_\alpha b_{ij}^{\alpha}
       \frac{C_\alpha}{B_iB_j}
       t^{q_\alpha-p_i-p_j}Z_\alpha(t).
\]
Multiplying each bracket coefficient by the proper-time scale $t$
gives $t^{\chi_{ij}^{\alpha}}$ times a constant.  Thus the condition
$\chi_{ij}^{\alpha}>0$ means that these time-normalized spatial
commutators vanish at the singular end.  When Einstein solutions have
an expansion rate of order $t^{-1}$, this is the spatial scaling
associated with a velocity-dominated regime.  The corollary uses this
scaling only to control timelike filling; it does not establish the
Einstein equations, an expansion asymptotic, or perturbative stability
for every metric covered by its hypotheses.

\begin{example}[Different central time scales]
\label{ex:two-central-time-scales}
Take $\mathsf N=H_3\times H_3$ with the only nonzero basic brackets
$\mathsf b(e_1,e_2)=z_1$ and $\mathsf b(e_3,e_4)=z_2$.  Let
\[
 H_t=\operatorname{diag}(t^{-1/2},t^{-1/2},t,t),
 \qquad A_t=\operatorname{diag}(t^{-1},t^{3/2}),
 \qquad L_t=0,
 \qquad 0<t<1.
\]
Thus $(p_1,p_2,p_3,p_4)=(-1/4,-1/4,1/2,1/2)$ and
$(q_1,q_2)=(-1/2,3/4)$.  The two bracket exponents are $1$ and
$3/4$, respectively.  All the cone inequalities hold, and a horizontal
direction expands.  Corollary~\ref{cor:all-two-step-anisotropic-powers}
therefore proves past $C^0$-inextendibility on the cover and every
discrete quotient.

In contrast, the common central majorant in
\eqref{eq:computable-energy} gives $a_+(t)=t^{-1/2}$.
Restricting its variational problem to the $(e_3,e_4)$ plane gives
energy $\int t^{3/2}|\dot\xi|^2\,\dd t$.  The based-circle
calculation after the change of variable
$u=\int t^{-3/2}\,\dd t$ shows
\[
 \kappa_{\rm Sch}(T)\ge
 \frac1{4\pi}\sup_{0<a<b<T}\int_a^b t^{-3/2}\,\dd t
 =+\infty.
\]
The separate central estimate tends to zero nevertheless.  The
example demonstrates why distinct central time scales should be kept
separate; it is a geometric example and no Einstein-matter equation
is imposed.
\end{example}

\subsection{Power-law metrics}

The next corollary allows the horizontal-central coupling to diverge.

\begin{corollary}[Power-law Schur metrics]
\label{cor:power-law}
Let \eqref{eq:scalar-block-metric} be in proper time on
$(0,t_+)\times\mathsf N$.  Suppose that, as $t\downarrow0$,
\begin{equation}
 a(t)\asymp t^{p_Z},\qquad b(t)\asymp t^{p_V},
 \qquad \|L_t\|\le C t^{-q_L},\qquad q_L\ge0.
 \label{eq:power-asymptotics}
\end{equation}
Let $d_V=\dim V$, $d_Z=\dim Z$, and assume
\begin{equation}
 \max\{p_Z,p_V+q_L\}<1,\qquad
 D:=1+p_Z-2p_V>0,
 \label{eq:power-causal-filling}
\end{equation}
together with
\begin{equation}
 p_V<0
 \quad\text{or}\quad
 d_Zp_Z+d_Vp_V<0.
 \label{eq:power-growth}
\end{equation}
Then the past end on the simply connected cover is $C^0$-inextendible.  The
same holds on every compact nilmanifold quotient under either alternative,
with no lattice alignment requirement.  Future timelike
completeness promotes each conclusion to global $C^0$-inextendibility.
\end{corollary}

\begin{proof}
Proposition~\ref{prop:trace-cone} gives
\[
 m_h(t)^{-1}\le C\bigl(a(t)^{-1}+(1+\|L_t\|)/b(t)\bigr).
\]
The first condition in \eqref{eq:power-causal-filling} therefore gives
finite cone width.  Moreover,
\[
 q(t)^2\dd\mu(t)=\frac{a(t)}{b(t)^2}\,\dd t
 \asymp t^{D-1}\dd t.
\]
Equation \eqref{eq:exact-area-constant} gives
$\kappa_{\rm an}(T)=O(T^D)<1/2$ on a sufficiently early tail.  If $p_V<0$,
$\sqrt{\lambda_{\max}(H_t)}=b(t)\to\infty$.  If the second alternative in
\eqref{eq:power-growth} holds, the invariant density
\[
 \sqrt{\det H_t\det A_t}=b(t)^{d_V}a(t)^{d_Z}
 \asymp t^{d_Vp_V+d_Zp_Z}
\]
diverges.  Theorem~\ref{thm:two-step-obstruction} applies on the cover and
every quotient.  If $a=t^{p_Z}$, $b=t^{p_V}$ and
$L_t=t^{-q_L}L_0$ with $L_0\ne0$, Proposition~\ref{prop:trace-cone}
also shows that $p_Z<1$ and $p_V+q_L<1$ are necessary and sufficient
for the intrinsic cone integral.  Thus the cone range is exact for this
specified power-law family.
\end{proof}

\subsection{A critical bracket exponent}
The strict power inequality used above is sufficient for the filling
constant to tend to zero.  At equality the exact constant can remain
finite, and its size is then determined by the anisotropy and the
coefficient of the bracket.

\begin{proposition}[The critical Bianchi II filling constant]
\label{prop:critical-bianchi-filling}
On the normalized Heisenberg algebra $\mathsf b(e_x,e_y)=z$, let
\[
 N=1,\qquad
 H_t=\begin{pmatrix}B_x^2t^{2p_x}&0\\0&B_y^2t^{2p_y}\end{pmatrix},
 \qquad A_t(z,z)=C^2t^{2q},\qquad B_x,B_y,C>0,
\]
with arbitrary smooth Schur coupling $L_t$.  Suppose
\[
 1+q-p_x-p_y=0.
\]
For every $T>0$ on which these data are defined,
\[
 \kappa_{\rm an}(T)=
 \begin{cases}
 \displaystyle\frac{C}{B_xB_y|p_x-p_y|},&p_x\ne p_y,\\[6pt]
 +\infty,&p_x=p_y.
 \end{cases}
\]
In particular, if $p_x\ne p_y$ and
$C/(B_xB_y)\le|p_x-p_y|/2$, the tail $(0,T)\times H_3$ is future
one-connected.  These are exact values of the variational constant;
the sufficient filling inequality is not asserted to be necessary.
\end{proposition}

\begin{proof}
For a based loop $\xi=(\xi_x,\xi_y)$ on a compact interval
$I\Subset(0,T)$, the signed area and energy are
\[
 Q=\int_I\xi_x\dot\xi_y\,\dd t,
 \qquad
 E=\int_I\left(\frac{B_x^2}{C}t^{2p_x-q}\dot\xi_x^2
              +\frac{B_y^2}{C}t^{2p_y-q}\dot\xi_y^2\right)\,\dd t.
\]
Write $\delta=p_y-p_x$, $a=\delta/2$, and $s=\log t$.  Since
$q=p_x+p_y-1$, make the invertible change
\[
 \xi_x(t)=\frac{\sqrt C}{B_x}t^{\delta/2}X(s),
 \qquad
 \xi_y(t)=\frac{\sqrt C}{B_y}t^{-\delta/2}Y(s).
\]
The functions $X,Y$ vanish at the endpoints of $\log I$; extension
by zero therefore places them in $W^{1,2}(\mathbb R)$ with compact
support in $(-\infty,\log T)$.  A direct substitution, using
$\dd t=t\,\dd s$, gives
\begin{align*}
 E&=\int_{\mathbb R}
       \bigl((X'+aX)^2+(Y'-aY)^2\bigr)\,\dd s\\
  &=\int_{\mathbb R}
       \bigl(X'^2+Y'^2+a^2(X^2+Y^2)\bigr)\,\dd s,\\
 Q&=\frac{C}{B_xB_y}\int_{\mathbb R}X(Y'-aY)\,\dd s.
\end{align*}
The cross terms in the energy integrate to zero because both
functions have zero trace.  Conversely, every such compactly supported
pair yields an admissible based loop on some compact subinterval of
$(0,T)$.

Suppose $a\ne0$, and set
\[
 E_X=\int(X'^2+a^2X^2),\qquad
 E_Y=\int(Y'^2+a^2Y^2).
\]
Since $\|X\|_2\le E_X^{1/2}/|a|$ and
$\|Y'-aY\|_2^2=E_Y$, Cauchy-Schwarz gives
\[
 |Q|\le\frac{C}{B_xB_y|a|}\sqrt{E_XE_Y}
       \le\frac{C}{2B_xB_y|a|}E.
\]
This proves the upper bound.

For sharpness, choose a nonzero $\eta\in C_c^\infty((0,1))$ and,
for $L>0$, put
\[
 X_L(s)=Y_L(s)=
 \eta\left(\frac{s-\log T+2L}{L}\right).
\]
Their support lies in $(\log T-2L,\log T-L)$, so each pair is
admissible on a compact interval strictly inside the tail.  Let
$U=\int_0^1\eta^2>0$ and $V=\int_0^1(\eta')^2$.  Then
\[
 |Q_L|=\frac{C}{B_xB_y}|a|LU,\qquad
 E_L=2(L^{-1}V+a^2LU),
\]
because $\int X_LX_L'=0$.  Consequently
\[
 \frac{|Q_L|}{E_L}\longrightarrow
 \frac{C}{2B_xB_y|a|}
 =\frac{C}{B_xB_y|p_x-p_y|}.
\]
This proves the exact supremum; no loop reaching $t=0$ has been used.

If $a=0$, the transformed expressions become
$E=\int(X'^2+Y'^2)$ and
$Q=(C/(B_xB_y))\int XY'$.  A based circle on a logarithmic interval
of length $L$ has $|Q|/E=CL/(4\pi B_xB_y)$.  Such intervals of
arbitrarily large finite length fit inside $(-\infty,\log T)$.
Thus the supremum is infinite.  The filling conclusion follows from
Proposition~\ref{thm:schur-filling}; the Schur coupling has not entered
the calculation.
\end{proof}

\begin{example}[Inextendibility at the critical exponent]
\label{ex:critical-inextendibility}
In the Heisenberg coframe
$\sigma^1=\dd z-x\,\dd y$, $\sigma^2=\dd x$, $\sigma^3=\dd y$,
consider
\[
 g=-\dd t^2+C^2t^{-3/2}(\sigma^1)^2
             +B_x^2t^{-1/2}(\sigma^2)^2
             +B_y^2t(\sigma^3)^2,
 \qquad 0<t<T.
\]
Here $(p_x,p_y,q)=(-1/4,1/2,-3/4)$, so
$1+q-p_x-p_y=0$ and $|p_x-p_y|=3/4$.  If
$C/(B_xB_y)\le3/8$, Proposition~\ref{prop:critical-bianchi-filling}
gives $\kappa_{\rm an}(T)\le1/2$.  The least-scale inverse is bounded
by
\[
 C^{-1}t^{3/4}+B_x^{-1}t^{1/4}+B_y^{-1}t^{-1/2},
\]
which is integrable at zero.  The horizontal eigenvalue
$B_x^2t^{-1/2}$ diverges, as does the invariant volume density
$CB_xB_y t^{-1/2}$.  Therefore
Theorem~\ref{thm:two-step-obstruction} proves past
$C^0$-inextendibility on the simply connected cover and on every
discrete quotient.  This example lies outside the strict bracket
inequalities: its time-normalized spatial bracket is the nonzero
constant $C/(B_xB_y)$.  It is a geometric example; no Einstein-matter
equation or energy condition is asserted.
\end{example}

\begin{remark}
For completeness, if $\chi=1+q-p_x-p_y<0$ in the same exact power
family, then $\kappa_{\rm an}(T)=+\infty$.  Indeed, for any fixed
nonzero-area based loop $\xi$ supported in $[a,b]\Subset(0,\infty)$,
define on $[\varepsilon a,\varepsilon b]$
\[
 \xi_{\varepsilon,i}(t)
     =\varepsilon^{(1+q-2p_i)/2}\xi_i(t/\varepsilon),
 \qquad i=x,y.
\]
Direct substitution leaves each term of the energy unchanged and
multiplies the area by $\varepsilon^\chi$.  For sufficiently small
$\varepsilon$ these intervals lie in $(0,T)$, and the area-to-energy
ratio tends to infinity.  Hence the critical proposition describes
the finite-constant boundary between the decaying constants for
$\chi>0$ and divergent constants for $\chi<0$, with the isotropic
critical case excluded explicitly.
\end{remark}

\subsection{Diagonal Bianchi II metrics}

The diagonal case admits a direct integral test without a power-law
assumption.  Use the coframe~\eqref{eq:heisenberg-coframe} and write
\begin{equation}
 g=-N(t)^2\dd t^2+a(t)^2(\dd z-x\dd y)^2
                  +b(t)^2\dd x^2+c(t)^2\dd y^2,
 \label{eq:diagonal-heisenberg-metric}
\end{equation}
where all four functions are smooth and positive on $(t_-,t_+)$.  Put
\begin{align}
 f_x(t)&=\frac{N(t)a(t)}{b(t)^2},&
 f_y(t)&=\frac{N(t)a(t)}{c(t)^2},\notag\\
 R_{\rm diag}(T)&=\int_{t_-}^T
                 \frac{N(t)a(t)}{\min\{b(t)^2,c(t)^2\}}\,\dd t,
 \label{eq:diagonal-area-data}\\
 V_{xy}(T)^2&=\int_{t_-}^T f_y(t)
                   \int_{t_-}^t f_x(r)\,\dd r\,\dd t,&
 V_{yx}(T)^2&=\int_{t_-}^T f_x(t)
                   \int_{t_-}^t f_y(r)\,\dd r\,\dd t.\notag
\end{align}
These are nonnegative improper integrals, with $+\infty$ allowed.

\begin{proposition}[Diagonal Heisenberg filling and inextendibility]
\label{prop:diagonal-heisenberg}
For~\eqref{eq:diagonal-heisenberg-metric},
\begin{equation}
 \kappa_{\rm an}(T)\le\kappa_{\rm diag}(T):=
 \min\left\{\frac{R_{\rm diag}(T)}{4\pi},
             \frac{V_{xy}(T)}2,\frac{V_{yx}(T)}2\right\}.
 \label{eq:diagonal-area-bound}
\end{equation}
If $\kappa_{\rm diag}(T)\le1/2$, then the lifted past slab is future
one-connected.  Suppose in addition that
\begin{equation}
 \int_{t_-}^T\frac{N(t)}{\min\{a(t),b(t),c(t)\}}\,\dd t<\infty
 \label{eq:diagonal-cone-test}
\end{equation}
and either $\limsup_{t\to t_-}\max\{b(t),c(t)\}=\infty$ or
$\limsup_{t\to t_-}a(t)b(t)c(t)=\infty$.  Then the past end is
$C^0$-inextendible on $H_3$ and on every discrete left quotient.
Future timelike geodesic completeness gives global
$C^0$-inextendibility.
\end{proposition}

\begin{proof}
For a based horizontal loop $\xi=(\xi_x,\xi_y)$ on
$I=[r_0,r_1]\Subset(t_-,T)$, its central area and energy are
\[
 Q=\int_I\xi_x\dot\xi_y\,\dd t
   =-\int_I\xi_y\dot\xi_x\,\dd t,\qquad
 E=\int_I\left(\frac{\dot\xi_x^2}{f_x}
                         +\frac{\dot\xi_y^2}{f_y}\right)\dd t.
\]
For a nonzero loop, $E>0$, and the variational quotient defining
$\kappa_{\rm an}$ is $|Q|/E$.
Since
$E\ge\int_I\min\{b^2,c^2\}|\dot\xi|^2/(Na)\,\dd t$,
the change of time and periodic Wirtinger estimate in
Proposition~\ref{prop:exact-area}, with bracket norm $B_0=1$, give
$|Q|\le R_{\rm diag}(T)E/(4\pi)$ whenever that integral is finite.
For the ordered estimate, write
$\dot\xi_x=\sqrt{f_x}\,u$ and
$\dot\xi_y=\sqrt{f_y}\,v$.  Then
\[
 Q=\int_{r_0}^{r_1}\sqrt{f_y(t)}\,v(t)
             \int_{r_0}^t\sqrt{f_x(r)}\,u(r)\,\dd r\,\dd t.
\]
The squared Hilbert--Schmidt norm of this triangular kernel is at most
$V_{xy}(T)^2$.  Hence
$|Q|\le V_{xy}(T)\|u\|_2\|v\|_2\le V_{xy}(T)E/2$.
The other expression for $Q$ gives the estimate with $V_{yx}$.
Taking the supremum proves~\eqref{eq:diagonal-area-bound}.

For completeness, the central correction can be written directly in
these coordinates.  Given two timelike curves with common endpoints,
put $\xi=(x_1-x_0,y_1-y_0)$ and
$\eta_j=\dot z_j-x_j\dot y_j$.  Interpolate $x,y$ affinely and set
\[
 c_\theta(t)=\theta(1-\theta)\frac QE
                  \frac{b^2\dot\xi_x^2+c^2\dot\xi_y^2}{Na},
 \qquad
 \eta_\theta=(1-\theta)\eta_0+\theta\eta_1+c_\theta,
\]
with $c_\theta=0$ if $E=0$.  Define $z_\theta$ by integrating
$\dot z_\theta=\eta_\theta+x_\theta\dot y_\theta$ from the common
initial point.  The integral of $c_\theta$ is
$\theta(1-\theta)Q$, which cancels the endpoint defect
$-\theta(1-\theta)Q$ of the horizontal interpolation.  Moreover,
\[
 \frac{a|c_\theta|}{N}
 \le\theta(1-\theta)\kappa_{\rm diag}(T)
                 \frac{b^2\dot\xi_x^2+c^2\dot\xi_y^2}{N^2}.
\]
Lemma~\ref{lem:convexity} therefore preserves strict timelikeness.
The interpolating curves depend continuously on $\theta$ and have fixed
endpoints, giving the asserted filling.

Choose the auxiliary metric
$k=(\dd z-x\dd y)^2+\dd x^2+\dd y^2$.  Then
$m_h=\min\{a,b,c\}$,
$\lambda_{\max}(H_t)=\max\{b^2,c^2\}$, and the invariant density is
$abc$.  Theorem~\ref{thm:two-step-obstruction} now gives the past and
global conclusions, including every discrete quotient.
\end{proof}

\subsection{General Bianchi II metrics}

Use the coframe~\eqref{eq:heisenberg-coframe} on $H_3$.
Every left-invariant Bianchi II metric has the form
\begin{equation}
 h_t=H_t(\dd x,\dd y)+
 a(t)^2\bigl(\sigma^1+\ell_t\bigr)^2,
 \label{eq:general-bII}
\end{equation}
where $H_t$ is a positive quadratic form in the horizontal variables and
$\ell_t$ is a horizontal one-form.  This is the one-dimensional-center case
of \eqref{eq:schur-form}.  Proposition~\ref{thm:schur-filling} therefore gives a
future timelike filling whenever $\kappa_{\rm an}(T)\le1/2$, with no separate
bound on $\ell_t$.  The coupling enters the inextendibility criterion through
the intrinsic cone scale $m_h$.  Theorem~\ref{thm:two-step-obstruction} applies when the cone
integral is finite and either a horizontal eigenvalue or the density
$a(t)\sqrt{\det H_t}$ diverges.  Proposition~\ref{prop:diagonal-heisenberg}
gives the diagonal criterion directly, and
Proposition~\ref{prop:asymptotic-transfer} proves its power-law comparison
form.  The general Schur metric also permits rotating horizontal
eigendirections, arbitrary horizontal-central coupling, and every discrete
left quotient, including compact Heisenberg nilmanifolds.

For instance, take
\[
 H_t=t^{2p_V}I_2,\qquad a(t)=t^{p_Z},\qquad
 \ell_t=t^{-q_L}(\lambda_2\,\dd x+\lambda_3\,\dd y),
\]
where $(\lambda_2,\lambda_3)\ne(0,0)$.  Conditions
\eqref{eq:power-causal-filling} and \eqref{eq:power-growth}, with
$(d_V,d_Z)=(2,1)$, give explicit $C^0$-inextendible Bianchi II ends even when
the coupling is unbounded.  For example,
$(p_Z,p_V,q_L)=(3/4,-1/4,1)$ satisfies the improved cone and filling
conditions but not the earlier estimate
$\max\{p_Z,p_V\}+q_L<1$.

The example is geometric rather than an Einstein-matter model with a
nonnegative normal energy density.  To see this, use the proper-time
orthonormal coframe
\[
 \theta^1=a(\sigma^1+\ell),\qquad
 \theta^2=b\,\dd x,\qquad \theta^3=b\,\dd y.
\]
Writing $H_a=\dot a/a$ and $H_b=\dot b/b$, the second fundamental form
$K=\tfrac12\dot h$ has orthonormal components
\[
 K_{11}=H_a,\qquad K_{22}=K_{33}=H_b,\qquad
 K_{1A}=\frac{a}{2b}\dot\ell_A\quad(A=2,3).
\]
The spatial scalar curvature is $R^{(3)}=-a^2/(2b^4)$, because the
sole orthonormal bracket has size $a/b^2$.  The Hamiltonian constraint
therefore gives, with $n=\partial_t$,
\begin{align}
 G(n,n)
 &=\frac12\bigl(R^{(3)}+(\operatorname{tr}K)^2-|K|^2\bigr)\notag\\
 &=H_b^2+2H_aH_b-\frac{a^2|\dot\ell|^2}{4b^2}
                  -\frac{a^2}{4b^4}\notag\\
 &=-\left(\frac5{16}+\frac{\lambda_2^2+\lambda_3^2}{4}\right)t^{-2}
   -\frac14t^{5/2}<0.
 \label{eq:coupling-example-negative-energy}
\end{align}
Thus its effective stress tensor $T=G$, in units $8\pi G=1$ with zero
cosmological constant, violates the weak energy condition.  This does
not affect the geometric inextendibility conclusion.  Section~\ref{sec:einstein-scalar} provides Einstein-scalar applications with nonnegative scalar energy density.

\subsection{Metrics comparable to anisotropic power laws}

On the Heisenberg group fix the coframe
$\sigma^1=\dd z-x\,\dd y$, $\sigma^2=\dd x$, $\sigma^3=\dd y$.
The following criterion requires no control of derivatives of the metric
relative to its leading power law.

\begin{proposition}[Power-law comparison criterion]
\label{prop:asymptotic-transfer}
Let $A_i>0$ and $p_i\in\R$ satisfy
\begin{equation}
 \chi:=1+p_1-p_2-p_3>0,\qquad
 \max_i p_i<1,\qquad
 \bigl(\min\{p_2,p_3\}<0\ \text{or}\ p_1+p_2+p_3<0\bigr).
 \label{eq:transfer-exponents}
\end{equation}
Put
\[
 h_t^0=\sum_{i=1}^3 A_i^2t^{2p_i}(\sigma^i)^2.
\]
Suppose $g=-\dd t^2+h_t$ is a smooth invariant orthogonal metric on
$(0,t_+)\times H_3$ and, for some $T_0>0$ and constants
$0<c_0\le C_0<\infty$,
\begin{equation}
 c_0h_t^0\le h_t\le C_0h_t^0\qquad(0<t<T_0).
 \label{eq:asymptotic-comparison}
\end{equation}
Then the past end is $C^0$-inextendible on the cover and on every quotient
by a discrete subgroup.  In particular the conclusion holds if
$(1-\eps(t))h_t^0\le h_t\le(1+\eps(t))h_t^0$ with $\eps(t)\to0$.
\end{proposition}

\begin{proof}
For $h_t^0$ the Schur data are
$A_t=A_1^2t^{2p_1}$,
$H_t=\operatorname{diag}(A_2^2t^{2p_2},A_3^2t^{2p_3})$, and $L_t=0$.
Set
\[
 f_2(t)=\frac{A_1}{A_2^2}t^{p_1-2p_2},\qquad
 f_3(t)=\frac{A_1}{A_3^2}t^{p_1-2p_3},\qquad
 D_j=1+p_1-2p_j\quad(j=2,3).
\]
For a based loop $(x,y)$, its central area and energy are
\[
 Q=\int x\dot y=-\int y\dot x,\qquad
 E=\int\left(\frac{\dot x^2}{f_2}+
                    \frac{\dot y^2}{f_3}\right).
\]
Writing $\dot x=\sqrt{f_2}\,u$ and $\dot y=\sqrt{f_3}\,v$, the
Hilbert-Schmidt bound for the triangular integral gives
\[
 |Q|\le V_{23}(T)\|u\|_2\|v\|_2
       \le\tfrac12V_{23}(T)E,\qquad
 V_{23}(T)^2=\int_0^T f_3(t)\int_0^t f_2(r)\,\dd r\,\dd t.
\]
The alternative expression for $Q$ gives the estimate with the two
indices interchanged.  Since $D_2+D_3=2\chi>0$, at least one $D_j$ is
positive.  If $D_2>0$, direct integration gives
\[
 V_{23}(T)^2=
 \frac{A_1^2}{A_2^2A_3^2}
 \frac{T^{2\chi}}{2\chi D_2}.
\]
This calculation remains valid when $D_3\le0$: only the inner
$f_2$ integral is evaluated separately at zero.  The other ordering
applies when $D_3>0$.  Consequently, with $D=\max\{D_2,D_3\}>0$,
\begin{equation}
 \kappa_{\rm an}^0(T)\le
 \frac{A_1}{2A_2A_3\sqrt{2\chi D}}T^\chi.
 \label{eq:power-ordered-bound}
\end{equation}
The least scale of $h_t^0$ satisfies
\[
 (m_{h^0}(t))^{-1}
 =\max_i A_i^{-1}t^{-p_i}
 \le\sum_i A_i^{-1}t^{-p_i},
\]
which is integrable at zero.  If a horizontal exponent is negative,
then $\lambda_{\max}(H_t^0)\to\infty$.  If $p_1+p_2+p_3<0$, the
invariant volume density $A_1A_2A_3t^{p_1+p_2+p_3}$ diverges.

Restriction to the center and the quotient characterization
\eqref{eq:quotient-characterization} show that
$c_0A_t^0\le A_t\le C_0A_t^0$ and
$c_0H_t^0\le H_t\le C_0H_t^0$.  Therefore the defining variational
quotients satisfy $J_I\ge c_0C_0^{-1/2}J_I^0$, and hence
\[
 \kappa_{\rm an}(T)\le\frac{\sqrt{C_0}}{c_0}
 \kappa_{\rm an}^0(T)\longrightarrow0.
\]
Also $m_h\ge\sqrt{c_0}\,m_{h^0}$ and
$\lambda_{\max}(H_t)\ge c_0\lambda_{\max}(H_t^0)$.
If the horizontal exponent is negative, the diameter alternative of
Theorem~\ref{thm:two-step-obstruction} applies after decreasing $T$.
In the volume case, comparison of determinants gives a volume density
at least $c_0^{3/2}A_1A_2A_3t^{p_1+p_2+p_3}$, and the volume
alternative of the same theorem applies.
\end{proof}

When $p_1+p_2+p_3=1$, the condition $\chi>0$ is exactly $p_1>0$,
and \eqref{eq:power-ordered-bound} has order $T^{2p_1}$.
The geometric criterion itself requires no field equation or Kasner
constraint.  Its exponent $\chi$ measures the decay of the ordered
bracket-area estimate.

\section{Arithmetic of compact Kasner slices}
\label{app:kasner-arithmetic}

The diameter of a compact Kasner slice depends on its lattice, although
the continuous-inextendibility conclusion does not.  Irrational lattices
also occur in cosmological causal pasts in \cite[Example~7(d)]{Lott2020}.

\begin{corollary}[Arithmetic dependence of the spatial diameter]
\label{cor:arithmetic-kasner}
Replace $\vartheta$ by $\theta\in\R$ in
Proposition~\ref{prop:irrational-kasner}, and write $h_t^\theta$ for
the resulting spatial metric.  Each spacetime is vacuum and globally
$C^0$-inextendible, and $\Vol(\mathbb T^3,h_t^\theta)=t$.
\begin{enumerate}[label=\textup{(\alph*)}]
\item If $c_0>0$, $1\le w<2$ and
\begin{equation}
 |p-\theta q|\ge c_0|q|^{-w}
 \quad(p\in\mathbb Z,\ q\in\mathbb Z\setminus\{0\}),
 \label{eq:kasner-diophantine}
\end{equation}
then, for $0<t\le1$ and constants $c,C>0$,
\begin{equation}
 c t^{1/6}\le\diam(\mathbb T^3,h_t^\theta)
 \le C t^{(2-w)/(3(w+1))}.
 \label{eq:kasner-diophantine-diameter}
\end{equation}
Thus $w=1$ gives diameter comparable to $t^{1/6}$.
For each $1<w<2$, \eqref{eq:kasner-diophantine} holds for almost
every $\theta$, with $c_0$ depending on $\theta,w$.
\item If $\theta$ is irrational and $q_j\to\infty$ are positive
integers with $p_j\in\mathbb Z$ and
\begin{equation}
 \epsilon_j=|p_j-\theta q_j|,\qquad q_j^2\epsilon_j\longrightarrow0,
 \label{eq:kasner-liouville}
\end{equation}
then, at $t_j=\epsilon_j/q_j\to0$,
\begin{equation}
 \diam(\mathbb T^3,h_{t_j}^\theta)
 \ge c(q_j^2\epsilon_j)^{-1/3}\longrightarrow\infty.
 \label{eq:kasner-liouville-diameter}
\end{equation}
An example is $\theta=\sum_{n\ge1}10^{-n!}$.
\end{enumerate}
\end{corollary}

\begin{proof}
The vacuum, causal-length, lifted-diameter and volume calculations in
Proposition~\ref{prop:irrational-kasner} use only an orthonormal
rotation of $\dd x,\dd y$, so apply for every $\theta$; future
completeness also follows from
Theorem~\ref{thm:homogeneous-future}\textup{(i)}.
The planar factor has area $t^{1/3}$, giving the lower diameter
bound by the same disk-covering argument.

We record the planar estimate used throughout this appendix.  For a
Euclidean lattice $\Lambda$ of covolume $F$, let $\ell$ be the
shortest nonzero dual length and $D_2=\diam(\R^2/\Lambda)$.
A shortest dual vector is primitive; in orthonormal coordinates it
is $(0,\ell)$, and $\Lambda$ has a basis
$(F\ell,0),(x,1/\ell)$.  Successive subtraction of these basis
vectors puts every point in a rectangle with half-widths
$F\ell/2$ and $1/(2\ell)$.  The integral covector also defines a
surjective circle map, whence
\[
 \frac1{2\ell}\le D_2
 \le\frac12\sqrt{\ell^{-2}+F^2\ell^2}.
\]
The disk of radius $\ell/2$ embeds in the dual torus of area $1/F$,
so $F\ell^2\le4/\pi$.  The preceding inequalities imply both
\begin{align}
 \frac1{2\ell}&\le D_2\le\frac2{\sqrt3\,\ell},
 \label{eq:planar-dual-diameter}\\
 0&\le D_2-\frac1{2\ell}\le\frac{F\ell}{2}
       \le\sqrt{\frac F\pi}.
 \label{eq:planar-sharp-diameter-error}
\end{align}
For the planar Kasner factor, with $c_\theta=\sqrt{1+\theta^2}$,
the integral covectors have squared norms
\begin{equation}
 Q_t(m,n)=c_\theta^{-2}
 \left[t^{2/3}(m+\theta n)^2+t^{-4/3}(n-\theta m)^2\right].
 \label{eq:kasner-dual-quadratic}
\end{equation}
Put $\ell(t)^2=\min_{(m,n)\ne(0,0)}Q_t(m,n)$ and
$D_\theta(t)=\diam(\mathbb T^3,h_t^\theta)$.  Since $F=t^{1/3}$
and the remaining circle has diameter $t^{2/3}/2$, the product metric
satisfies
\begin{equation}
 0\le D_\theta(t)-\frac1{2\ell(t)}
 \le\frac{t^{1/6}}{\sqrt\pi}+\frac{t^{2/3}}2.
 \label{eq:kasner-diameter-dual-error}
\end{equation}

Under \eqref{eq:kasner-diophantine}, put
$r=\sqrt{m^2+n^2}$ and $v=(n-\theta m)/c_\theta$.
Then $|v|\ge c_1r^{-w}$ for every nonzero integer pair (treat
$m=0$ separately).  For $t\le2^{-1/2}$,
\[
 Q_t=t^{2/3}r^2+(t^{-4/3}-t^{2/3})v^2
 \ge t^{2/3}r^2+\tfrac12c_1^2t^{-4/3}r^{-2w}
 \ge c_2t^{2(w-2)/(3(w+1))},
\]
where the last inequality splits at $r=t^{-1/(w+1)}$.
Equation~\eqref{eq:planar-dual-diameter}, the circle factor and a
compact-time comparison prove \eqref{eq:kasner-diophantine-diameter}.
On a bounded $\theta$-interval, the set where
$|p-\theta q|<q^{-w}$ for some $p$ has measure $O(q^{-w})$.
Borel-Cantelli, followed by a decrease of $c_0$ for finitely many
denominators and exclusion of rational $\theta$, proves the
almost-everywhere assertion.

For \textup{(b)}, substitution of $(m,n)=(q_j,p_j)$ gives
$Q_{t_j}(q_j,p_j)\le C_\theta q_j^{4/3}\epsilon_j^{2/3}$.
The lower circle bound proves \eqref{eq:kasner-liouville-diameter}.
For the displayed number, use $q_j=10^{j!}$ and its $j$th truncated
sum $p_j/q_j$: the positive tail gives
$0<\epsilon_j\le2q_j^{-j}$, hence $q_j^2\epsilon_j\to0$.
It also proves irrationality, since a rational limit of denominator
$b$ would give $\epsilon_j\ge1/b$.
\end{proof}

\begin{proposition}[Exact arithmetic criterion for diameter collapse]
\label{prop:kasner-arithmetic-criterion}
For irrational $\theta$, define
$A(\theta)=\liminf_{q\to\infty}q^2\|q\theta\|_{\R/\mathbb Z}$,
where $q$ is a positive integer and
$\|v\|_{\R/\mathbb Z}=\min_{p\in\mathbb Z}|v-p|$.
Then
\begin{equation}
 \limsup_{t\downarrow0}D_\theta(t)
 =\frac{A(\theta)^{-1/3}}
 {2^{2/3}\sqrt3\,(1+\theta^2)^{1/6}},
 \label{eq:kasner-exact-diameter-limsup}
\end{equation}
with $0^{-1/3}=\infty$ and $\infty^{-1/3}=0$.
Thus the diameter tends to zero exactly when $A(\theta)=\infty$,
has a finite positive upper limit exactly when $0<A(\theta)<\infty$,
and has unbounded excursions exactly when $A(\theta)=0$.
\end{proposition}

\begin{proof}
We prove, in the extended nonnegative real numbers,
\begin{equation}
 \liminf_{t\downarrow0}\ell(t)^2
 =K_\theta A(\theta)^{2/3},\qquad
 K_\theta=3\,2^{-2/3}c_\theta^{2/3}.
 \label{eq:kasner-exact-dual-liminf}
\end{equation}
The elementary minimum of $at^{2/3}+bt^{-4/3}$ for $a,b>0$
is $3\,2^{-2/3}a^{2/3}b^{1/3}$, attained at $t^2=2b/a$.
Irrationality makes both terms of \eqref{eq:kasner-dual-quadratic}
nonzero for a nonzero integer pair, so
\begin{equation}
 Q_t(m,n)\ge3\,2^{-2/3}c_\theta^{-2}
 \bigl(|m+\theta n|^2|n-\theta m|\bigr)^{2/3}.
 \label{eq:kasner-dual-optimized-lower}
\end{equation}
If $L=\liminf_{t\downarrow0}\ell(t)^2$ is finite, choose minimizing
pairs $(m_j,n_j)$ at $t_j\downarrow0$ with $Q_{t_j}(m_j,n_j)\to L$.
Boundedness implies $\delta_j=|n_j-\theta m_j|\to0$ and
$q_j=|m_j|\to\infty$: otherwise a fixed nonzero pair would have
$n_j=\theta m_j$.  Therefore
\[
 \frac{|m_j+\theta n_j|}{q_j}\to c_\theta^2,\qquad
 q_j^2\delta_j\ge q_j^2\|q_j\theta\|_{\R/\mathbb Z}.
\]
Equation~\eqref{eq:kasner-dual-optimized-lower} yields
$L\ge K_\theta A(\theta)^{2/3}$; in particular $A(\theta)=\infty$
forces $L=\infty$.  The lower bound is automatic if $L=\infty$.
For the reverse inequality when $A(\theta)<\infty$, take nearest
integers $p_j$ with $q_j^2\epsilon_j\to A(\theta)$, where
$\epsilon_j=|p_j-\theta q_j|$.  At
$s_j=\sqrt2\epsilon_j/|q_j+\theta p_j|\to0$ the scalar minimum
is attained, and
\[
 \ell(s_j)^2\le Q_{s_j}(q_j,p_j)
 =3\,2^{-2/3}c_\theta^{-2}
 \left(\frac{|q_j+\theta p_j|^2}{q_j^2}
             q_j^2\epsilon_j\right)^{2/3}
 \longrightarrow K_\theta A(\theta)^{2/3}.
\]
This proves \eqref{eq:kasner-exact-dual-liminf}.
Equation~\eqref{eq:kasner-diameter-dual-error} and reciprocation give
\eqref{eq:kasner-exact-diameter-limsup}, including its extended-real cases.
\end{proof}

\begin{proposition}[Collapse subsequences and rational slopes]
\label{prop:kasner-arithmetic-completion}
For every irrational $\theta$, there are $t_j\downarrow0$ and
$c,C>0$ such that
\begin{equation}
 c t_j^{1/6}\le D_\theta(t_j)\le C t_j^{1/6}.
 \label{eq:kasner-universal-collapse-subsequence}
\end{equation}
Thus $\liminf_{t\downarrow0}D_\theta(t)=0$.
For $\theta=a/b$, where $a\in\mathbb Z$, $b\in\mathbb N$ and
$\gcd(a,b)=1$, one instead has
\begin{equation}
 D_{a/b}(t)=\frac{t^{-1/3}}{2\sqrt{a^2+b^2}}+O(t^{5/3}).
 \label{eq:kasner-rational-diameter-asymptotic}
\end{equation}
\end{proposition}

\begin{proof}
For irrational $\theta$, write $Q_t(v)=|S_tv|^2$, where
\[
 S_t=c_\theta^{-1}
 \begin{pmatrix}t^{1/3}&\theta t^{1/3}\\
 -\theta t^{-2/3}&t^{-2/3}\end{pmatrix},\qquad
 \det S_t=t^{-1/3}.
\]
Every minimizing integer vector is primitive.  On a compact time
interval, uniform positive definiteness of $Q_t$ and the bound
$\ell(t)^2\le Q_t(1,0)$ leave only finitely many possible minimizers.
A unique minimizing pair $\{v,-v\}$ is therefore locally constant.
It cannot remain fixed on $(0,\tau)$: the dual disk bound gives
$\ell(t)^2\le4t^{-1/3}/\pi$, while each fixed $v=(m,n)\ne0$
has $Q_t(v)\ge c_\theta^{-2}(n-\theta m)^2t^{-4/3}$.
Hence every such interval contains two distinct minimizing pairs.
At a sequence of these times $t_j\downarrow0$, choose minimizing
primitive vectors $v_j,w_j$ not equal up to sign.  They are independent,
and therefore
\[
 t_j^{-1/3}\le|\det(S_{t_j}v_j,S_{t_j}w_j)|
 \le\ell(t_j)^2.
\]
Equation~\eqref{eq:kasner-diameter-dual-error} gives the upper
bound in \eqref{eq:kasner-universal-collapse-subsequence}; the
planar area gives the lower bound.

For $\theta=a/b$, put $R=a^2+b^2$.
Pairs with $bn-am=0$ are $k(b,a)$ and have
$Q_t=k^2Rt^{2/3}$; all others have $Q_t\ge R^{-1}t^{-4/3}$.
Thus $\ell(t)^2=Rt^{2/3}$ for $0<t<R^{-1}$.
The rectangle bound above and the orthogonal circle yield
\[
 \frac1{4Rt^{2/3}}\le D_{a/b}(t)^2
 \le\frac1{4Rt^{2/3}}+\frac{R+1}{4}t^{4/3}.
\]
Dividing the difference of squares by
$D_{a/b}(t)+t^{-1/3}/(2\sqrt R)\ge t^{-1/3}/\sqrt R$
gives a remainder at most $\sqrt R(R+1)t^{5/3}/4$.
\end{proof}

\begin{proposition}[Realization of every finite positive arithmetic value]
\label{prop:kasner-arithmetic-realization}
For every $A_*\in(0,\infty)$ there is an irrational $\theta\in(0,1)$
with $A(\theta)=A_*$.  In particular, the finite positive upper-limit
case in Proposition~\ref{prop:kasner-arithmetic-criterion} is nonempty.
\end{proposition}

\begin{proof}
Construct $\theta=[0;a_1,a_2,\ldots]$ by setting
$q_{-1}=0$, $q_0=1$, $p_{-1}=1$, $p_0=0$ and recursively
\[
 a_{n+1}=\max\{1,\lfloor q_n/A_*\rfloor\},\qquad
 (p_{n+1},q_{n+1})=a_{n+1}(p_n,q_n)+(p_{n-1},q_{n-1}).
\]
Then $q_n\to\infty$, $a_{n+1}\to\infty$ and
\begin{equation}
 q_n/a_{n+1}\longrightarrow A_*.
 \label{eq:kasner-cf-digit-choice}
\end{equation}
The recurrence gives
$p_nq_{n-1}-p_{n-1}q_n=(-1)^{n-1}$.
For the tail $\alpha_{n+1}=[a_{n+1};a_{n+2},\ldots]$ it gives
\[
 \theta=\frac{p_n\alpha_{n+1}+p_{n-1}}
 {q_n\alpha_{n+1}+q_{n-1}},\qquad
 e_n=q_n\theta-p_n=\frac{(-1)^n}{q_n\alpha_{n+1}+q_{n-1}}.
\]
These nonzero errors tend to zero, proving irrationality, and
$a_{n+1}<\alpha_{n+1}<a_{n+1}+1$ gives
\begin{equation}
 q_n^2|e_n|=\frac{q_n}{\alpha_{n+1}+q_{n-1}/q_n}
 \longrightarrow A_*.
 \label{eq:kasner-cf-error-limit}
\end{equation}
To control other denominators, the determinant identity writes
$(p,q)=u(p_n,q_n)+v(p_{n+1},q_{n+1})$ with integer $u,v$.
If $0<q<q_{n+1}$, either $v=0$, $u\ne0$, or $u,v$ have
opposite signs.  Since $e_n,e_{n+1}$ have opposite signs, both
cases give $|q\theta-p|\ge|e_n|$.
For large $n$, $p_n$ is a nearest integer because $|e_n|<1/2$.
Thus $q_n\le q<q_{n+1}$ implies
$q^2\|q\theta\|_{\R/\mathbb Z}\ge q_n^2|e_n|$.
Taking lower limits, and using the convergents for the reverse
inequality, proves $A(\theta)=A_*$.
\end{proof}

These propositions realize collapse, finite positive diameter upper
limits, and unbounded excursions among irrational slopes; every such
slope has collapse subsequences.  Rational slopes have the divergent
asymptotic~\eqref{eq:kasner-rational-diameter-asymptotic}.
All cases retain the same local vacuum geometry and continuous
inextendibility.

\section{Spectral formulas for the filling constants}
\label{app:area-spectrum}

\subsection{Exact constants with a conformal central metric}
\label{subsec:conformal-central-spectrum}

The horizontal metric need not be a scalar block for the variational
constant to admit a finite-dimensional characterization.  Suppose that
the central metric is conformal to a fixed inner product $A_0$ on $Z$:
\begin{equation}
 A_t=a(t)^2A_0,\qquad a(t)>0.
 \label{eq:conformal-central-metric}
\end{equation}
Fix an auxiliary inner product on $V$, put $d=\dim V$, and use this
inner product to represent $H_t$ by a positive-definite symmetric
operator.  For $z\in Z$ define the skew-adjoint operator $J_z$ on $V$ by
\begin{equation}
 \langle J_zv,w\rangle=A_0\bigl(z,\mathsf b(v,w)\bigr),
 \qquad M(t)=\frac{H_t}{N(t)a(t)}.
 \label{eq:central-direction-operator}
\end{equation}
For $I=[\alpha,\beta]\Subset(t_-,T)$ set
\begin{align}
 E_I(\xi)&=\int_I\langle M(t)\dot\xi,\dot\xi\rangle\,\dd t,
 & Q_{I,z}(\xi)&=A_0\bigl(z,Q_I(\xi)\bigr),\notag\\
 c_z(I)&=\sup_{0\ne\xi\in W^{1,2}_0(I;V)}
                  \frac{|Q_{I,z}(\xi)|}{E_I(\xi)}.
 \label{eq:directional-area-constant}
\end{align}
These definitions allow $J_z$ to be singular, or identically zero.

\begin{proposition}[Central directions and the endpoint equation]
\label{prop:central-endpoint-spectrum}
Assume \eqref{eq:conformal-central-metric}, with arbitrary smooth positive
horizontal metric $H_t$ and arbitrary horizontal-central coupling $L_t$.
For $\omega\in\R$ and $z\in Z$, let $\Phi_{\omega,z}$ solve
\begin{equation}
 \dot\Phi_{\omega,z}(t)
   =\omega J_zM(t)^{-1}\Phi_{\omega,z}(t),
 \qquad \Phi_{\omega,z}(\alpha)=I_d,
 \label{eq:general-area-momentum}
\end{equation}
and define the endpoint matrix
\begin{equation}
 B_{\omega,z}(I)
   =\int_\alpha^\beta M(t)^{-1}\Phi_{\omega,z}(t)\,\dd t.
 \label{eq:general-area-endpoint}
\end{equation}
If $J_z\ne0$, the set
\[
 \{\omega\in\R\setminus\{0\}:\det B_{\omega,z}(I)=0\}
\]
is nonempty, and the minimum
\begin{equation}
 \omega_{*,z}(I)
 =\min\{|\omega|:\omega\in\R\setminus\{0\},\quad
                  \det B_{\omega,z}(I)=0\}
 \label{eq:general-area-first-frequency}
\end{equation}
exists and is positive.  If $J_z=0$, the endpoint matrix is invertible
for every $\omega$, and we put $\omega_{*,z}(I)=+\infty$.
Then
\begin{equation}
 c_z(I)=\frac1{2\omega_{*,z}(I)},\qquad
 \kappa_{\rm an}(T)
  =\sup_{I\Subset(t_-,T)}\max_{|z|_{A_0}=1}
                         \frac1{2\omega_{*,z}(I)},
 \label{eq:general-area-exact-spectrum}
\end{equation}
where $1/(+\infty)=0$.  In particular, the maximum over central
directions is attained on each compact interval.  Both signs of $\omega$
are included in \eqref{eq:general-area-first-frequency}.
\end{proposition}

\begin{proof}
Since $M$ is smooth and positive definite on $I$, $E_I^{1/2}$ is a Hilbert
norm equivalent to the usual norm on $W^{1,2}_0(I;V)$.  Poincar\'e's
inequality and boundedness of the bracket give
\begin{equation}
 |Q_I(\xi)|_{A_0}\le C_I E_I(\xi).
 \label{eq:general-area-bound}
\end{equation}
If $J_z=0$, then $Q_{I,z}=0$, $\Phi_{\omega,z}=I_d$, and
$B_{\omega,z}=\int_I M^{-1}\,\dd t$ is positive definite.  All the stated
conclusions for this direction follow.  Henceforth suppose $J_z\ne0$.
Choose $v,w\in V$ with $\langle J_zv,w\rangle\ne0$.  A based ellipse in
their span has nonzero signed area, so $0<c_z(I)<\infty$.

The quadratic form $Q_{I,z}$ is weakly continuous on bounded subsets of
$W^{1,2}_0(I;V)$.  Indeed, weak convergence there gives strong convergence
of the functions in $L^2$, by compactness, and weak convergence of their
derivatives in $L^2$.  These convergences pass to
\[
 Q_{I,z}(\xi)
   =\frac12\int_I\langle J_z\xi,\dot\xi\rangle\,\dd t.
\]
Thus $|Q_{I,z}|$ attains a positive maximum on the weakly compact ball
$E_I\le1$.  Homogeneity forces every maximizer to satisfy $E_I=1$.
Choose one, denote it by $\xi$, and put
$\lambda=Q_{I,z}(\xi)$.  Then $|\lambda|=c_z(I)>0$.  The first
variation of the signed quadratic form at this constrained extremum gives
\[
 -\int_I\langle J_z\dot\xi,\eta\rangle\,\dd t
   =2\lambda\int_I\langle M\dot\xi,\dot\eta\rangle\,\dd t
 \qquad\bigl(\eta\in W^{1,2}_0(I;V)\bigr).
\]
The multiplier is $\lambda$ because evaluation of the first-variation
identity on $\eta=\xi$ gives $2Q_{I,z}(\xi)=2\lambda E_I(\xi)$.
Consequently
\begin{equation}
 2\lambda(M\dot\xi)'=J_z\dot\xi
 \label{eq:general-area-euler-lagrange}
\end{equation}
in distributions.  With $p=M\dot\xi$ and
$\omega=(2\lambda)^{-1}$ this becomes
\[
 p'=\omega J_zM^{-1}p.
\]
Initially $p\in L^2$, and the equation gives $p\in W^{1,2}$.  Its
continuous representative therefore solves the integral equation; the
smooth coefficients then give the usual classical regularity.  In
particular $p(t)=\Phi_{\omega,z}(t)p_0$, where $p_0=p(\alpha)\ne0$.
The last inequality follows from uniqueness for the linear initial-value
problem, since $p_0=0$ would imply $\xi=0$.
Since $\xi(\alpha)=0$, integration of $\dot\xi=M^{-1}p$ gives
\begin{equation}
 \xi(t)=\left(\int_\alpha^t
           M(r)^{-1}\Phi_{\omega,z}(r)\,\dd r\right)p_0.
 \label{eq:general-area-loop-reconstruction}
\end{equation}
The other endpoint condition is exactly
$B_{\omega,z}(I)p_0=0$.  Thus a nonzero return frequency exists, with
\begin{equation}
 |\omega|=\frac1{2c_z(I)}.
 \label{eq:general-area-extremal-frequency}
\end{equation}

Conversely, suppose $\omega\ne0$ and
$0\ne p_0\in\ker B_{\omega,z}(I)$.  Define
$p=\Phi_{\omega,z}p_0$ and $\xi$ by
\eqref{eq:general-area-loop-reconstruction}.  Then $\xi$ is a based
loop with $\dot\xi=M^{-1}p$ and $p'=\omega J_z\dot\xi$.
It is nonzero, since $\dot\xi(\alpha)=M(\alpha)^{-1}p_0\ne0$.
Integration by parts yields
\[
 E_I(\xi)=\int_I\langle p,\dot\xi\rangle\,\dd t
  =-\int_I\langle p',\xi\rangle\,\dd t
  =2\omega Q_{I,z}(\xi).
\]
Therefore $|\omega|\ge(2c_z(I))^{-1}$.  Combined with
\eqref{eq:general-area-extremal-frequency}, this proves the existence
and positivity of the minimum and the first identity in
\eqref{eq:general-area-exact-spectrum}.  Notice also that
$B_{0,z}(I)=\int_I M^{-1}\,\dd t$ is positive definite; thus the zero
frequency is never an endpoint root.

It remains to recover the vector-valued area.  Under
\eqref{eq:conformal-central-metric}, its defining quotient satisfies
\[
 J_I(\xi)=\frac{E_I(\xi)}{|Q_I(\xi)|_{A_0}}
 \qquad\text{when }Q_I(\xi)\ne0.
\]
By \eqref{eq:general-area-bound},
\[
 |c_z(I)-c_w(I)|\le C_I|z-w|_{A_0}.
\]
Thus $c_z(I)$ is continuous in $z$ and attains a maximum on the unit
central sphere.  The dual characterization of the norm and interchange
of two suprema now give
\begin{align*}
 \sup_{0\ne\xi\in W^{1,2}_0(I;V)}
      \frac{|Q_I(\xi)|_{A_0}}{E_I(\xi)}
 &=\sup_{0\ne\xi}\sup_{|z|_{A_0}=1}
      \frac{|Q_{I,z}(\xi)|}{E_I(\xi)}\\
 &=\max_{|z|_{A_0}=1}c_z(I).
\end{align*}
Taking the supremum over compact intervals proves the second identity
in \eqref{eq:general-area-exact-spectrum}.  None of these expressions
contains the Schur coupling $L_t$.
\end{proof}

\begin{remark}
The endpoint matrix is essential when $J_z$ is singular.  Integration
of \eqref{eq:general-area-momentum} gives
\begin{equation}
 \Phi_{\omega,z}(\beta)-I_d
       =\omega J_zB_{\omega,z}(I).
 \label{eq:general-area-return-endpoint}
\end{equation}
If $J_z$ is invertible and $\omega\ne0$, the endpoint condition is
equivalent to $\det(\Phi_{\omega,z}(\beta)-I_d)=0$.  If $J_z$ is
singular, that determinant vanishes for every frequency and does not
characterize based loops.  For example, on the free two-step algebra
$V=Z=\R^3$ with $\mathsf b(v,w)=v\times w$, one has
$J_zv=z\times v$ and $\ker J_z=\R z$ for $z\ne0$.
Formula \eqref{eq:general-area-exact-spectrum} applies without any
nondegeneracy assumption on these operators.
\end{remark}

For a genuinely time-dependent central anisotropy, the same endpoint
equation can still evaluate the auxiliary constant $\kappa_{\rm Sch}$:
use the fixed norm in \eqref{eq:computable-energy} and replace $a$ by
$a_+$.  The compact-interval proof requires only continuity and positive
definiteness of $M$, so continuity of $a_+$ suffices.  In that case it
gives the upper bound
$\kappa_{\rm an}\le\kappa_{\rm Sch}$ rather than, in general, an exact
evaluation of $\kappa_{\rm an}$.  Its direction-dependent denominator in
\eqref{eq:kappa-an} cannot be replaced by a common energy without that
distinction.

\begin{corollary}[Endpoint equation for arbitrary central anisotropy]
\label{cor:optimal-central-spectrum}
Fix auxiliary inner products and, for each nonzero $\lambda\in Z^*$,
define
\[
 \langle J_\lambda v,w\rangle=\lambda(\mathsf b(v,w)),\qquad
 M_\lambda(t)=\frac{\sqrt{A_t^{-1}(\lambda,\lambda)}}{N(t)}H_t.
\]
For a compact interval $I=[a,b]$, let
\[
 \dot\Phi_{\omega,\lambda}
       =\omega J_\lambda M_\lambda^{-1}\Phi_{\omega,\lambda},
 \qquad \Phi_{\omega,\lambda}(a)=I,
 \qquad
 B_{\omega,\lambda}=
       \int_a^b M_\lambda(t)^{-1}\Phi_{\omega,\lambda}(t)\,\dd t.
\]
When $J_\lambda\ne0$, let $\omega_{*,\lambda}(I)$ be the least
absolute value of a nonzero real root of $\det B_{\omega,\lambda}=0$;
when $J_\lambda=0$, set $\omega_{*,\lambda}(I)=+\infty$.
Then
\[
 \kappa_{\rm opt}(T)=
 \sup_{I\Subset(t_-,T)}\max_{|\lambda|_*=1}
             \frac{1}{2\omega_{*,\lambda}(I)}.
\]
The roots exist and have a positive least modulus whenever
$J_\lambda\ne0$, and the displayed maximum is attained on each
compact interval.  This evaluates $\kappa_{\rm opt}$, rather than
$\kappa_{\rm an}$, for a nonconformal central metric.
\end{corollary}

\begin{proof}
For fixed $\lambda$, the denominator defining $k_I$ is the positive
quadratic energy $E_{I,\lambda}(\xi)=\int_I
\langle M_\lambda\dot\xi,\dot\xi\rangle\,\dd t$, whereas the numerator
is the absolute value of $\frac12\int_I
\langle J_\lambda\xi,\dot\xi\rangle\,\dd t$.
The variational argument of
Proposition~\ref{prop:central-endpoint-spectrum}, from existence of an
energy-normalized maximizer through its Euler-Lagrange equation and
endpoint reconstruction, applies to this fixed positive matrix
$M_\lambda$ without any change.  It gives
\[
 \sup_{\xi\ne0}\frac{|\lambda(Q_I(\xi))|}{E_{I,\lambda}(\xi)}
       =\frac1{2\omega_{*,\lambda}(I)}.
\]
On the auxiliary unit covector sphere, $M_\lambda$ is uniformly positive
on $I$ and depends continuously on $\lambda$.  Its energy is uniformly
equivalent to $\|\dot\xi\|_2^2$; both its coefficients and the area
form vary continuously and uniformly.  Hence these Rayleigh suprema
depend continuously on $\lambda$ (subtract the two quotients and use
the uniform denominator bound).  They attain a maximum on the sphere.
Interchanging the suprema over $\lambda$ and $\xi$ proves the formula.
\end{proof}

\subsection{The anisotropic Bianchi II filling constant}

In the three-dimensional Heisenberg group the exact filling constant can
also be determined when the horizontal metric has unequal, time-dependent
eigenvalues and eigendirections.  Choose a basis $e_1,e_2$ of $V$ and a
generator $z$ of $Z$ such that $\mathsf b(e_1,e_2)=z$, and write
$A_t(z,z)=a(t)^2$.  We identify $H_t$ with its positive-definite symmetric
$2\times2$ matrix in this basis and set
\begin{equation}
 M(t)=\frac{H_t}{N(t)a(t)},
 \qquad
 J=\begin{pmatrix}0&-1\\1&0\end{pmatrix}.
 \label{eq:area-hamiltonian-matrix}
\end{equation}
For $I=[\alpha,\beta]\Subset(t_-,T)$ define
\begin{equation}
 c(I)=\sup_{0\ne\xi\in W^{1,2}_0(I;\mathbb R^2)}
 \frac{\left|\frac12\int_I\langle J\xi,\dot\xi\rangle\,\dd t\right|}
 {\int_I\langle M(t)\dot\xi,\dot\xi\rangle\,\dd t}.
 \label{eq:compact-area-constant}
\end{equation}
Since $Z$ is one dimensional, the definition
\eqref{eq:kappa-an} gives
$\kappa_{\rm an}(T)=\sup_{I\Subset(t_-,T)}c(I)$.

\begin{proposition}[Hamiltonian return formula]
\label{prop:area-return}
For each real $\omega$, let $\Phi_\omega$ solve
\begin{equation}
 \dot\Phi_\omega(t)=\omega J M(t)^{-1}\Phi_\omega(t),
 \qquad \Phi_\omega(\alpha)=I_2.
 \label{eq:area-return-ode}
\end{equation}
There is a nonzero real $\omega$ for which
$\det(\Phi_\omega(\beta)-I_2)=0$.  Moreover, the minimum
\begin{equation}
 \omega_*(I)=\min\bigl\{|\omega|:
 \omega\in\mathbb R\setminus\{0\},\quad
 \det(\Phi_\omega(\beta)-I_2)=0\bigr\}
 \label{eq:first-return-frequency}
\end{equation}
exists, is positive, and satisfies
\begin{equation}
 c(I)=\frac1{2\omega_*(I)},\qquad
 \kappa_{\rm an}(T)
 =\sup_{I\Subset(t_-,T)}\frac1{2\omega_*(I)}.
 \label{eq:exact-anisotropic-return}
\end{equation}
These equalities hold for arbitrary horizontal-central coupling $L_t$.
They evaluate the variational constant exactly; they do not assert that
$\kappa_{\rm an}(T)\le1/2$ is necessary for future one-connectedness.
\end{proposition}

\begin{proof}
Apply Proposition~\ref{prop:central-endpoint-spectrum} with the unit
central generator $z$, $\dim V=2$, and $J_z=J$.  The two unit central
directions differ by sign and give the same absolute area constant.
Since $J$ is invertible, identity
\eqref{eq:general-area-return-endpoint} shows that, for $\omega\ne0$,
\[
 \det\bigl(\Phi_\omega(\beta)-I_2\bigr)
   =\omega^2\det B_{\omega,z}(I).
\]
Thus the return frequencies in \eqref{eq:first-return-frequency} are
exactly the endpoint frequencies in
\eqref{eq:general-area-first-frequency}.  Their nonempty set has a
positive least modulus, and
\eqref{eq:general-area-exact-spectrum} gives
\eqref{eq:exact-anisotropic-return}.  The arbitrary Schur coupling does
not enter either equation.
\end{proof}

For a scalar horizontal block, $M(t)=m(t)I_2$, equation
\eqref{eq:area-return-ode} integrates to
\[
 \Phi_\omega(\beta)
 =\exp\left(\omega J\int_\alpha^\beta m(t)^{-1}\,\dd t\right).
\]
Its first nonzero return frequency has modulus
$2\pi/\int_I m^{-1}\,\dd t$.  Hence
$c(I)=(4\pi)^{-1}\int_I Na/b^2\,\dd t$, recovering
Proposition~\ref{prop:exact-area} for the normalized Heisenberg bracket.
For a general horizontal matrix, \eqref{eq:area-return-ode} reduces the
exact variational problem to a $2\times2$ linear initial-value problem;
the definition \eqref{eq:first-return-frequency} includes both signs of
$\omega$.

\begin{corollary}[A rotating horizontal block]
\label{cor:rotating-area}
Let $I=[\alpha,\beta]$, $L=\beta-\alpha$, and
$\mathcal R(\theta)=\exp(\theta J)$.  Suppose that
\begin{equation}
 M(t)=\mathcal R\left(\frac{\pi(t-\alpha)}L\right)
 \begin{pmatrix}m_1&0\\0&m_2\end{pmatrix}
 \mathcal R\left(\frac{\pi(t-\alpha)}L\right)^{\!\mathsf T},
 \qquad m_1,m_2>0.
 \label{eq:rotating-horizontal-block}
\end{equation}
Writing $S=m_1+m_2$ and $P=m_1m_2$, the exact constant on this interval is
\begin{equation}
 c(I)=\frac{L}{\pi\bigl(\sqrt{S^2+32P}-S\bigr)}.
 \label{eq:rotating-area-constant}
\end{equation}
For the nonrotating matrix $\operatorname{diag}(m_1,m_2)$ the corresponding
constant is $c_{\rm fixed}(I)=L/(4\pi\sqrt P)$.  In particular,
\begin{equation}
 \frac{c(I)}{c_{\rm fixed}(I)}
 =\frac{\sqrt{S^2+32P}+S}{8\sqrt P}\ge1,
 \label{eq:rotating-area-ratio}
\end{equation}
with strict inequality when $m_1\ne m_2$.
\end{corollary}

\begin{proof}
Set $\Omega=\pi/L$ and $D=\operatorname{diag}(m_1,m_2)$.
In the momentum equation from \eqref{eq:area-return-ode}, the substitution
$p(t)=\mathcal R(\Omega(t-\alpha))q(t)$ gives
\[
 q'=B_\omega q,\qquad
 B_\omega=J(\omega D^{-1}-\Omega I_2)
 =\begin{pmatrix}
 0&-(\omega/m_2-\Omega)\\
 \omega/m_1-\Omega&0
 \end{pmatrix}.
\]
Consequently
$\Phi_\omega(\beta)=-\exp(LB_\omega)$.  Put
$d_\omega=(\omega/m_1-\Omega)(\omega/m_2-\Omega)$; then
$B_\omega^2=-d_\omega I_2$ and
$\det\exp(LB_\omega)=1$.  The return condition is equivalent to
$\operatorname{tr}\exp(LB_\omega)=-2$.
If $d_\omega<0$, this trace is
$2\cosh(L\sqrt{-d_\omega})>2$, and if $d_\omega=0$ it is $2$,
including the nilpotent case.  Thus neither case gives a return.
For $d_\omega>0$, the return condition is exactly
\[
 L\sqrt{d_\omega}=(2n+1)\pi,
 \qquad n=0,1,2,\ldots.
\]
For $n=0$ its roots are $\omega=0$ and $\omega=\Omega S$.
For $n\ge1$ they are
\[
 \omega_n^\pm=\frac\Omega2
 \left(S\pm\sqrt{S^2+16n(n+1)P}\right).
\]
The positive roots increase with $n$, and the moduli of the negative
roots also increase with $n$.  Since $4P\le S^2$,
\[
 |\omega_1^-|
 =\frac\Omega2\bigl(\sqrt{S^2+32P}-S\bigr)
 \le\Omega S.
\]
After excluding the zero root, the least modulus is therefore
$|\omega_1^-|$.  Proposition~\ref{prop:area-return} proves
\eqref{eq:rotating-area-constant}.

For the fixed matrix $D$, the substitution $\eta=D^{1/2}\xi$ changes
the energy to $\int_I|\dot\eta|^2\,\dd t$ and the signed area to
$(\det D)^{-1/2}\frac12\int_I\langle J\eta,\dot\eta\rangle\,\dd t$.
The scalar formula gives $c_{\rm fixed}(I)=L/(4\pi\sqrt P)$.
Rationalizing the denominator in \eqref{eq:rotating-area-constant}
gives \eqref{eq:rotating-area-ratio}.  The function
$s\mapsto\sqrt{s^2+32P}+s$ is strictly increasing for $s>0$,
and $S\ge2\sqrt P$, with equality precisely when $m_1=m_2$.
This proves the final assertion.
\end{proof}

Thus the exact filling constant is sensitive to the evolution of the
horizontal eigendirections, even when the two instantaneous eigenvalues
are fixed.  This comparison concerns the filling functional, not a
necessary threshold for causal connectedness or an Einstein evolution
stability statement.

\section{Massive-scalar future asymptotics}
\label{app:massive-future}

\begin{corollary}[A massive scalar field with positive cosmological constant]
\label{cor:massive-de-sitter-future}
Let $\mathcal V(\phi)=\Lambda+m^2\phi^2/2$, where $\Lambda>0$ and
$m>0$.  For every set of data in
Proposition~\ref{prop:potential-solutions} with
$\min\{p_2,p_3\}<0$, the maximal homogeneous development is globally
$C^0$-inextendible and future causally geodesically complete on the
cover and every lattice quotient.  There are a positive left-invariant
metric $h_\infty$ and constants $b,C>0$ such that, in the fixed invariant
frame, for all sufficiently large $t$,
\begin{equation}
 |\phi(t)|+|\dot\phi(t)|+
 \bigl|h_t^{-1}K-H_\Lambda\operatorname{Id}\bigr|_{h_\infty}
 +\bigl|e^{-2H_\Lambda t}h_t-h_\infty\bigr|_{h_\infty}
 \le Ce^{-bt},\qquad H_\Lambda=\sqrt{\Lambda/3}.
 \label{eq:massive-de-sitter-asymptotics}
\end{equation}
\end{corollary}

\begin{proof}
Global inextendibility and future causal completeness follow from
Corollary~\ref{cor:polynomial-potential} and
Proposition~\ref{prop:nonnegative-future-completeness}.
The future-asymptotic result of
\cite[Theorem~1]{Rendall2004Scalar} applies to initially expanding
Bianchi I-VIII solutions which exist for all future proper time,
with a $C^2$ scalar potential satisfying the following conditions:
\[
 \inf_{\R}\mathcal V>0,\qquad
 \sup_I|\mathcal V'|<\infty
   \ \hbox{whenever $I\subset\R$ is an interval and}
          \ \sup_I\mathcal V<\infty,
\]
and $\mathcal V'$ has a limit in $\R\cup\{-\infty,+\infty\}$
at each end of the scalar line.  Any additional matter must obey the
dominant and strong energy conditions; here there is no additional
matter.  The normalizations are related by the constant scalar and
potential rescaling that sets $8\pi G=1$.
For the present potential, $\mathcal V\ge\Lambda>0$,
\[
 \mathcal V(\phi)\le V_1
 \quad\Longrightarrow\quad
 |\phi|\le\frac{\sqrt{2(V_1-\Lambda)}}m,
 \qquad |\mathcal V'(\phi)|=m^2|\phi|,
\]
whenever the sublevel set is nonempty, and
$\mathcal V'(\phi)=m^2\phi\to\pm\infty$ as
$\phi\to\pm\infty$.  Bianchi II is among the allowed types,
and future global existence has already been proved.

The cited theorem and its proof give
$|\Sigma|\to0$, $R\to0$, $\dot\phi\to0$, and
$\mathcal V'(\phi(t))\to0$.  Thus $\phi\to0$.
Since the sole Bianchi II orthonormal structure coefficient satisfies
$R=-n^2/2$, we have $n\to0$; the Hamiltonian constraint gives
$\theta\to\sqrt{3\Lambda}$.  These limits place every member in
the set $\mathcal G$ defined in
\cite[Proposition~11.1]{Ringstrom2025}, since
$\mathcal V'(0)=0$ and $\mathcal V''(0)=m^2>0$.
The global assumptions of that proposition also hold:
$\mathcal V$ is smooth and coercive, its lower bound is positive,
and the only zero of $\mathcal V'$ is simple.
Consequently \cite[Proposition~11.3]{Ringstrom2025} gives exactly
\eqref{eq:massive-de-sitter-asymptotics}.
This use of Proposition~11.3 applies to every solution just considered,
because the defining limits have been established for each one.
The estimates descend to every lattice quotient.
\end{proof}

These solutions have a continuously inextendible anisotropic past and
an exponentially isotropizing future.  The nonzero limiting
$H_\Lambda$ describes accelerated late-time expansion.

\subsection{Massive fields in all odd Heisenberg dimensions}
The following argument proves the future estimates directly for the
higher-dimensional diagonal developments constructed in this paper.

\begin{corollary}[Massive fields with a positive cosmological constant]
\label{cor:heisenberg-massive-all-dimensions}
Let $n=2m+1\ge3$ and
$\mathcal V(\phi)=\Lambda+M^2\phi^2/2$, where $\Lambda>0$ and $M>0$.
Every maximal initially expanding diagonal homogeneous solution of
\eqref{eq:h5-field-equations} on $H_n$ exists for all future proper
time and is future timelike and null geodesically complete on every
discrete left quotient.  Put
\[
 H_\Lambda=\sqrt{\frac{2\Lambda}{n(n-1)}}.
\]
There are a positive left-invariant metric $h_\infty$ and constants
$b,C>0$ such that, in the fixed invariant frame, for all sufficiently
large $t$,
\begin{equation}
 |\phi(t)|+|\dot\phi(t)|+
 \bigl|h_t^{-1}K-H_\Lambda\operatorname{Id}\bigr|_{h_\infty}
 +\bigl|e^{-2H_\Lambda t}h_t-h_\infty\bigr|_{h_\infty}
 \le Ce^{-bt}.
 \label{eq:heisenberg-massive-future}
\end{equation}
In particular, every singularity datum of
Proposition~\ref{prop:h5-einstein} with a negative horizontal
exponent gives a globally $C^0$-inextendible maximal homogeneous
development on every discrete left quotient, in every odd spatial
dimension.
\end{corollary}

\begin{proof}
Write $H_a=\dot a_a/a_a$, $\theta=\sum_aH_a$,
$\Sigma_a=H_a-\theta/n$, $\psi=\dot\phi$, and
$b_r=a_0/(a_{2r-1}a_{2r})$.  In the spatial orthonormal frame, set
\[
 r_0(b)=\frac12\sum_r b_r^2,\qquad
 r_{2r-1}(b)=r_{2r}(b)=-\frac12b_r^2,\qquad
 R(b)=-\frac12\sum_r b_r^2.
\]
These are the diagonal spatial Ricci entries and their trace,
respectively.  The proper-time spatial field equations are
$\dot H_a=-\theta H_a-r_a(b)+2\mathcal V/(n-1)$.
Subtracting their average, and differentiating the structure
coefficients, gives the autonomous system
\begin{align}
 \dot\Sigma_a&=-\theta\Sigma_a-r_a(b)+R(b)/n,
                   &\sum_a\Sigma_a&=0,\notag\\
 \dot b_r&=\left(-\theta/n+
          \Sigma_0-\Sigma_{2r-1}-\Sigma_{2r}\right)b_r,
                   &\dot\phi&=\psi,\notag\\
 \dot\psi&=-\theta\psi-M^2\phi.
 \label{eq:heisenberg-massive-autonomous}
\end{align}
On the expanding branch the Hamiltonian constraint determines
$\theta$ as the smooth positive function
\begin{equation}
 \theta=
 \left[\frac n{n-1}\left(
       \sum_a\Sigma_a^2+\psi^2+2\Lambda+M^2\phi^2
                         +\frac12\sum_r b_r^2\right)\right]^{1/2}.
 \label{eq:heisenberg-massive-expansion}
\end{equation}
In particular $\theta\ge\theta_\Lambda:=nH_\Lambda>0$.
The trace evolution gives
\begin{equation}
 -\dot\theta=
 \frac n{n-1}\left(\sum_a\Sigma_a^2+\psi^2\right)
                +\frac1{2(n-1)}\sum_r b_r^2\ge0.
 \label{eq:heisenberg-massive-lyapunov}
\end{equation}

Consequently $\theta_\Lambda\le\theta(t)\le\theta(t_0)$.
Equation~\eqref{eq:heisenberg-massive-expansion} bounds all variables
$\Sigma_a,b_r,\phi,\psi$.  Moreover $H_a=\theta/n+\Sigma_a$ is
bounded.  At a finite endpoint, integration of $\dot a_a=H_aa_a$
keeps every $a_a$ bounded above and bounded away from zero.
The homogeneous evolution equations therefore extend the solution
across any finite future endpoint.  Thus its future proper-time
interval is infinite.  This argument has no dimension restriction.

For clarity, convergence follows directly from the finite-dimensional
system, without a genericity assumption.  Regard
$x=(\Sigma,b,\phi,\psi)$ as a variable in the space
$\sum_a\Sigma_a=0$, and use
\eqref{eq:heisenberg-massive-expansion} to eliminate $\theta$.
The orbit $x(t)$ has compact closure and its vector field is smooth.
The decreasing function $\theta(x(t))$ has a limit.  Its value is
constant on the set of accumulation points of $x(t)$ as
$t\to\infty$.  This set is nonempty and invariant: a convergent
sequence $x(t_j)$, together with continuous dependence of solutions,
shows that every fixed finite segment of the trajectory through a
limit point is again a limit of segments through $x(t_j)$.
Equation~\eqref{eq:heisenberg-massive-lyapunov} then forces
$\Sigma=0$, $b=0$, and $\psi=0$ along every such trajectory.
The last equation of \eqref{eq:heisenberg-massive-autonomous}
forces $\phi=0$ there as well.  Hence the accumulation set consists
only of $x=0$, and
\[
 \Sigma_a\longrightarrow0,\qquad b_r\longrightarrow0,
 \qquad \phi,\psi\longrightarrow0,
 \qquad \theta\longrightarrow\theta_\Lambda.
\]

The convergence is exponential.  The derivative of
$\theta(x)$ at $x=0$ vanishes.  The linearization of
\eqref{eq:heisenberg-massive-autonomous} at that point is block
diagonal: on the trace-free shear variables it is
$-\theta_\Lambda\operatorname{Id}$, on the $b$ variables it is
$-H_\Lambda\operatorname{Id}$, and on $(\phi,\psi)$ it is
\[
 A_\phi=\begin{pmatrix}0&1\\-M^2&-\theta_\Lambda\end{pmatrix}.
\]
Its scalar eigenvalues are the roots of
$z^2+\theta_\Lambda z+M^2=0$.  Every eigenvalue of the full
linearization $A$ has strictly negative real part.
One may see the resulting decay directly by defining
\[
 P=\int_0^\infty e^{A^{\mathsf T}s}e^{As}\,\dd s.
\]
The integral converges, $P$ is positive definite, and
$A^{\mathsf T}P+PA=-\operatorname{Id}$.
Since $\dot x=Ax+O(|x|^2)$,
\[
 \frac{\dd}{\dd t}(x^{\mathsf T}Px)
   =-|x|^2+O(|x|^3)\le-\tfrac12|x|^2
\]
once $x$ is sufficiently small.  Equivalence of the positive
quadratic form with $|x|^2$ proves $|x(t)|\le Ce^{-bt}$ for some
$b>0$.  In particular $H_a-H_\Lambda=O(e^{-bt})$.
Integration gives positive constants $A_a^\infty$ such that
\[
 e^{-H_\Lambda t}a_a(t)
       =A_a^\infty(1+O(e^{-bt})).
\]
Taking $h_\infty=\sum_a(A_a^\infty)^2(\omega^a)^2$ proves
\eqref{eq:heisenberg-massive-future}.

For future causal completeness, along an affinely parametrized
future causal geodesic let $E=\dd t/\dd\lambda>0$ and
$w=\dd x/\dd t$.  The time component of the geodesic equation is
\[
 \frac{\dd}{\dd t}\log E=-K(w,w),\qquad h_t(w,w)\le1.
\]
For all sufficiently large $t$, every $H_a\ge H_\Lambda/2$, so
$K$ is positive definite and $E(t)$ is nonincreasing.  The affine
length is therefore at least a positive constant times the elapsed
proper time and diverges as $t\to\infty$.
If $t$ has a finite future upper limit instead, comparison with a
fixed complete invariant metric and the geodesic energy equation
bound the position and tangent on the corresponding finite slab.
The geodesic equation then continues the curve.  This proves both
timelike and null completeness; lifting proves it on every discrete
left quotient.

Finally, the quadratic potential satisfies
\eqref{eq:potential-growth} with any prescribed $c>0$.
Choose $c$ so that $c|q|<2$ and apply
Proposition~\ref{prop:h5-einstein} to the stated singularity data.
The resulting tail is expanding and has a past
$C^0$-inextendible end whenever a horizontal exponent is negative.
The future completeness just proved, the Cauchy property of the
product slices, and the one-sided completeness obstruction
\cite{GLS2018} give global $C^0$-inextendibility.
\end{proof}

\section{Metric comparison and invariant shift}
\label{sec:stability}

The strict form of the two-step criterion is open under a uniform
two-sided comparison on an entire past tail.  The estimate applies
directly to the optimal allocation constant $\kappa_{\rm opt}$.
Such comparison is a hypothesis on the evolving metrics, not a conclusion
about the Einstein initial-value problem.  Small changes in asymptotic
exponents, for example, need not satisfy it:
$t^{2(p+\epsilon)}/t^{2p}=t^{2\epsilon}$ is not uniformly bounded above
and below near zero when $\epsilon\ne0$.

\begin{proposition}[Quantitative two-sided comparison]
\label{thm:stability}
Let
\[
 g=-N^2\dd t^2+h_t,\qquad
 \widetilde g=-\widetilde N^2\dd t^2+\widetilde h_t
\]
be invariant orthogonal evolutions on the same two-step group, or data
descending to the same quotient.  Suppose that, on $(t_-,T_0)$,
\begin{equation}
 (1-\eps)N\le\widetilde N\le(1+\eps)N,\qquad
 (1-\eps)h_t\le\widetilde h_t\le(1+\eps)h_t,\qquad 0<\eps<1.
 \label{eq:metric-comparison}
\end{equation}
Then, for $T\le T_0$ and
$C_\eps=(1+\eps)^{3/2}/(1-\eps)$,
\begin{equation}
 \widetilde\kappa_{\rm opt}(T)
 \le C_\eps\,\kappa_{\rm opt}(T),\qquad
 \widetilde\kappa_{\rm an}(T)
 \le C_\eps\,\kappa_{\rm an}(T).
 \label{eq:kappa-stability}
\end{equation}
If $\kappa_{\rm opt}(T)\le\tfrac12-\delta$ for
$0<\delta<\tfrac12$ and
\begin{equation}
 \frac{(1+\eps)^{3/2}}{1-\eps}
 \left(\frac12-\delta\right)\le\frac12,
 \label{eq:admissible-epsilon}
\end{equation}
then $\widetilde\kappa_{\rm opt}(T)\le\tfrac12$.
The largest admissible $\eps=\eps_*(\delta)$ is the unique root in
$(0,1)$ of equality in \eqref{eq:admissible-epsilon}, and
\begin{equation}
 \eps_*(\delta)=\frac45\delta+O(\delta^2)
 \qquad(\delta\downarrow0).
 \label{eq:epsilon-expansion}
\end{equation}
The comparison also preserves finiteness of $\int_{t_-}^T N/m_h$ and
each of the metric-divergence conditions in
Theorem~\ref{thm:two-step-obstruction}.
Hence every strict instance of that theorem remains past
$C^0$-inextendible under sufficiently small comparisons of the form
\eqref{eq:metric-comparison}.
The global conclusion also persists when the perturbed spacetime is
independently future timelike geodesically complete.
\end{proposition}

\begin{proof}
Restriction of \eqref{eq:metric-comparison} to $Z$ gives
$(1-\eps)A_t\le\widetilde A_t\le(1+\eps)A_t$.  The quotient
characterization
\begin{equation}
 H_t(v,v)=\inf_{z\in Z}h_t(sv+z,sv+z)
 \label{eq:quotient-characterization}
\end{equation}
gives the same comparison for $H_t$ and $\widetilde H_t$.
Minimizing over $[\mathfrak n,\mathfrak n]$ gives it for the
abelianization metrics.  Inversion reverses the central comparison, so
\[
 \widetilde A_t^{-1}\ge(1+\eps)^{-1}A_t^{-1}.
\]
For every nonzero loop $\xi\in W^{1,2}_0(I;V)$ and
$0\ne\lambda\in Z^*$, its optimal-allocation denominator therefore obeys
\begin{align*}
 &\int_I\frac{\widetilde H_t(\dot\xi,\dot\xi)}{\widetilde N(t)}
        \sqrt{\widetilde A_t^{-1}(\lambda,\lambda)}\,\dd t\\
 &\qquad\ge
 \frac{1-\eps}{(1+\eps)^{3/2}}
 \int_I\frac{H_t(\dot\xi,\dot\xi)}{N(t)}
        \sqrt{A_t^{-1}(\lambda,\lambda)}\,\dd t.
\end{align*}
The numerator $|\lambda(Q_I(\xi))|$ is metric independent.
Taking the suprema in \eqref{eq:kappa-opt} first over $\lambda$ and
then over $I,\xi$ proves the first inequality in
\eqref{eq:kappa-stability}.
Likewise, for $Q_I(\xi)\ne0$,
\[
 \widetilde J_I(\xi)
 \ge\frac{1-\eps}{(1+\eps)^{3/2}}J_I(\xi),
\]
which proves the second inequality.
The strict-margin conclusion follows from
\eqref{eq:admissible-epsilon}.
The left side of the equivalent inequality
$(1+\eps)^{3/2}(1-2\delta)\le1-\eps$ is strictly increasing in $\eps$
and the right side is strictly decreasing.
Their values at zero and one give a unique equality point in $(0,1)$.
Expansion at $(\eps,\delta)=(0,0)$ gives
\eqref{eq:epsilon-expansion}.

Using the same fixed auxiliary metric for both evolutions,
\[
 m_{\widetilde h}\ge\sqrt{1-\eps}\,m_h,\qquad
 \frac{\widetilde N}{m_{\widetilde h}}
 \le\frac{1+\eps}{\sqrt{1-\eps}}\frac N{m_h}.
\]
Thus the finite cone integral persists.  Writing $d=\dim\mathsf N$,
the quotient and determinant comparisons give
\begin{align*}
 \lambda_{\max}(\overline{\widetilde h}_t)
    &\ge(1-\eps)\lambda_{\max}(\overline h_t),\\
 \det\widetilde H_t\det\widetilde A_t
    &\ge(1-\eps)^d\det H_t\det A_t.
\end{align*}
These preserve the two divergence alternatives in
Theorem~\ref{thm:two-step-obstruction}.
In particular, their horizontal and fixed-covector sufficient
conditions are also preserved, since
\[
 \lambda_{\max}(\widetilde H_t)
    \ge(1-\eps)\lambda_{\max}(H_t),\qquad
 \widetilde h_t^{-1}(\alpha,\alpha)
    \le(1-\eps)^{-1}h_t^{-1}(\alpha,\alpha).
\]
Proposition~\ref{thm:schur-filling} and
Theorem~\ref{thm:two-step-obstruction} now give the stated
inextendibility conclusions.
\end{proof}

\begin{remark}[Why the intrinsic scale is needed]
For $H_t=I$, $A_t=t^2I$, $L_t=0$ and $N=t$, replace $L_t$ by
$\widetilde L_t=(c/t)P$, with $\|P\|=1$ and fixed small $c>0$.
In variables $(v,tz)$ the two quadratic forms have uniformly close
positive Gram matrices, so the full metrics are uniformly bilipschitz.
Nevertheless $m_{\rm Sch}\asymp t$, while
$\widetilde m_{\rm Sch}\asymp t^2/c$.
The corresponding Schur-bound integral changes from finite to infinite.
In contrast, $m_h\asymp m_{\widetilde h}\asymp t$, as required by
Proposition~\ref{thm:stability}.
The Schur bound is a sufficient estimate, not an invariant hypothesis
to be preserved separately.
\end{remark}

\subsection{Invariant shift on a fixed quotient}

Let $\omega$ be the left Maurer-Cartan form on $\mathsf N$ and let
$\beta(t)\in\mathfrak n$ be smooth.
It descends with $\omega$ to any left quotient, so the metric
\begin{equation}
 g_\beta=-N^2\dd t^2+h_t(\omega+\beta\dd t,\omega+\beta\dd t)
 \label{eq:shift-metric}
\end{equation}
is well defined on $(t_-,t_+)\times\Gamma\backslash\mathsf N$.

\begin{proposition}[Removal of invariant shift]
\label{prop:shift}
Let $a(t)$ solve $a^{-1}a'=-\beta(t)$.
The diffeomorphism $F(t,[x])=(t,[xa(t)])$ transforms
\eqref{eq:shift-metric} into
\[
 F^*g_\beta=-N^2\dd t^2+\bar h_t,\qquad
 \bar h_t=(\Ad_{a(t)^{-1}})^*h_t.
\]
The forms $H_t,A_t$ and the constants
$\kappa_{\rm opt},\kappa_{\rm an}$ are unchanged.
The conclusions of Theorem~\ref{thm:two-step-obstruction} therefore
hold with the cone integral computed using $m_{\bar h}$.
If the horizontal component $v_a(t)$ of $a(t)$ is bounded on the
past tail, finiteness of $\int_{t_-}^T N/m_h$ already implies that
cone condition.
It suffices that the horizontal drift $\beta_V$ be integrable;
no integrability of the central drift is needed.
\end{proposition}

\begin{proof}
Right multiplication is well defined on a left quotient for every $a$;
no lattice-preserving automorphism is required.
In two-step exponential coordinates the equation for $a$ consists
of an integral for its horizontal component followed by an integral
for its central component.
Thus its solution exists throughout the open time interval.
The Maurer-Cartan identity gives
\[
 F^*\omega=\Ad_{a^{-1}}\omega+a^{-1}a'\dd t,
\]
which proves the metric formula.
The adjoint map fixes the center pointwise and induces the identity
on $\mathfrak n/Z$.
It changes only the Schur coupling.
Hence $H_t,A_t,Q_I$, and both
$\kappa_{\rm opt}$ and $\kappa_{\rm an}$ are unchanged.
It also induces the identity on
$\mathfrak n/[\mathfrak n,\mathfrak n]$, so the abelianization metric
is unchanged.
The determinant and horizontal-growth alternatives are preserved,
and every annihilator covector is fixed by the adjoint map.

In exponential coordinates $v_a'=-\beta_V$ and
$\Ad_{a^{-1}}(v,z)=(v,z-\mathsf b(v_a,v))$.
Boundedness of $v_a$ therefore makes this map and its inverse
uniformly bounded, independently of the central component of $a$.
In particular, $\bar h_t\ge c\,m_h(t)^2k$ for a fixed $c>0$.
This proves the sufficient cone condition and the assertion about
integrable horizontal drift.
Finally $F$ preserves the time coordinate, so it preserves the end
under consideration and its extendibility.
\end{proof}

\section{Analytic horizons and uniform vacuum asymptotics}\label{app:heisenberg-bifurcate}

On the universal cover both regular horizon components can be included
in a single analytic extension.  The obstruction to taking its full
lattice quotient occurs at their intersection.

\begin{proposition}[Bifurcate Heisenberg horizons]
\label{prop:heisenberg-bifurcate}
For every $m\ge1$, the exceptional vacuum metric
\eqref{eq:heisenberg-exceptional-metric} on
$(0,\infty)\times H_{2m+1}$ admits an analytic Ricci-flat extension
to $\mathbb R^{2m+2}$.  Its past boundary consists of the future
portions of two nondegenerate Killing horizons meeting in a spacelike
bifurcation surface diffeomorphic to $\mathbb R^{2m}$.  The surface
gravities are $+1$ and $-1$ for the boost normalization used below.
Every nontrivial central translation fixes the bifurcation surface.
Consequently, for a lattice $\Gamma<H_{2m+1}$, the extended lattice
action is not free or properly discontinuous there.  The regular
one-sided quotient extensions remain those of
Proposition~\ref{prop:heisenberg-vacuum-dichotomy}.
\end{proposition}

\begin{proof}
Retain $b_r,F_r,H,h_r$ from
\eqref{eq:heisenberg-horizon-functions}, and put
\[
 \rho=\sqrt s,\qquad \zeta=A_0z,\qquad
 \mathcal C=A_0\sum_{r=1}^m x_r\dd y_r.
\]
Then \eqref{eq:heisenberg-exceptional-metric} becomes
\begin{equation}
 g=-H(\rho^2)\dd\rho^2
   +\frac{\rho^2}{H(\rho^2)}(\dd\zeta-\mathcal C)^2
   +\sum_rF_r(\rho^2)h_r.
 \label{eq:heisenberg-bifurcate-polar}
\end{equation}
Set
\[
 \mathsf T=\rho\cosh\zeta,\qquad
 Z=\rho\sinh\zeta,\qquad
 s=\mathsf T^2-Z^2,
\]
and define
\[
 \Theta=\mathsf T\dd\mathsf T-Z\dd Z,\qquad
 \Psi=\mathsf T\dd Z-Z\dd\mathsf T,\qquad
 P(s)=\frac{H(s)-1}{s^2}.
\]
The function $P$ extends to a polynomial at $s=0$, since
$H(s)=\prod_r(1+b_rs^2)$.  Products of one-forms below are
symmetrized tensor products.  On the future cone
$\{\mathsf T>|Z|\}$, one has
$\Theta=\rho\dd\rho$ and $\Psi=s\dd\zeta$.  Therefore the metric is
\begin{align}
 \widetilde g={}&g_0-sP(s)\Theta^2
   -\frac{sP(s)}{H(s)}(\Psi-s\mathcal C)^2
   +s^2\sum_r b_rh_r,\label{eq:heisenberg-bifurcate-metric}\\
 g_0={}&-\dd\mathsf T^2+\dd Z^2
        -2\Psi\mathcal C+s\mathcal C^2+\sum_rh_r.
 \notag
\end{align}
Indeed, subtracting $g_0$ from
\eqref{eq:heisenberg-bifurcate-polar} gives
$-(H-1)\Theta^2/s+(H^{-1}-1)(\Psi-s\mathcal C)^2/s$
and the displayed horizontal correction.
Since $H(s)>0$ for every real $s$, all coefficients in
\eqref{eq:heisenberg-bifurcate-metric} are real analytic on
$\mathbb R^{2m+2}$.

On the future cone the determinant of the $(\rho,\zeta)$ block
in the coframe $(\dd\rho,\dd\zeta-\mathcal C,\dd x_r,\dd y_r)$
is $-\rho^2$.  The horizontal determinant is
$\prod_rF_r(s)^2\prod_{i=1}^{2m}A_i^2
=H(s)^2\prod_{i=1}^{2m}A_i^2$.
The Jacobian of $(\rho,\zeta)\mapsto(\mathsf T,Z)$ is $\rho$.
Thus, in the Cartesian coordinates,
\begin{equation}
 \det\widetilde g=-H(s)^2\prod_{i=1}^{2m}A_i^2<0.
 \label{eq:heisenberg-bifurcate-determinant}
\end{equation}
Both sides are analytic, so this identity holds on all of
$\mathbb R^{2m+2}$.  The metric is everywhere nondegenerate;
connectedness and its signature on the future cone imply that it
is Lorentzian everywhere.  The coordinate transformation is a
diffeomorphism from the original spacetime onto that proper open
cone.  Its Ricci tensor is analytic and vanishes on the cone,
and hence vanishes everywhere.

The field
\[
 K=Z\partial_{\mathsf T}+\mathsf T\partial_Z
\]
preserves $s,\Theta,\Psi$ and the horizontal forms, so it is Killing.
On the original cone $K=\partial_\zeta$, giving the identities
\begin{equation}
 K^\flat=\frac{\Psi-s\mathcal C}{H(s)},\qquad
 \widetilde g(K,K)=\frac{s}{H(s)},\qquad
 \widetilde g^{-1}(\dd s,\dd s)=-\frac{4s}{H(s)}.
 \label{eq:heisenberg-bifurcate-identities}
\end{equation}
Analyticity extends them to the whole space.  The hypersurfaces
$\mathcal H_+=\{\mathsf T=Z\}$ and
$\mathcal H_-=\{\mathsf T=-Z\}$ are null.  Away from their
intersection $K$ is a nonvanishing null normal and
$\dd s=\mp2K^\flat$, respectively.  Since $H(0)=1$,
$\dd(\widetilde g(K,K))=\dd s$ there.  The identity
$\dd(\widetilde g(K,K))=-2\varkappa K^\flat$ gives
$\varkappa=+1$ on $\mathcal H_+$ and $\varkappa=-1$ on
$\mathcal H_-$.  At their intersection
\[
 \mathcal B=\{\mathsf T=Z=0\}\cong\mathbb R^{2m}
\]
the metric is $-\dd\mathsf T^2+\dd Z^2+\sum_rh_r$.
Thus the two hypersurfaces remain null there, $\mathcal B$ is
spacelike with induced metric $\sum_rh_r$, and $K$ vanishes
precisely on $\mathcal B$.

The boundary of the original cone consists exactly of the portions
$\mathsf T\ge0$ of these hypersurfaces.  Every point of this
boundary is timelike accessible from the cone.  For fixed horizontal
coordinates, the curve $(\mathsf T,Z)=(\lambda,0)$ has squared
tangent norm $-H(\lambda^2)<0$ and reaches $\mathcal B$ as
$\lambda\downarrow0$.  For $u>0$ and either sign, the curves
$(\mathsf T,Z)=(u+\lambda,\pm u)$ have
$s=\lambda(2u+\lambda)>0$ and squared tangent norm
\[
 -\frac{H(s)\mathsf T^2}{s}
       +\frac{Z^2}{sH(s)}
 =-H(s)-\frac{(H(s)^2-1)Z^2}{sH(s)}<0.
\]
Here $H(s)\ge1$, and $\dd s/\dd\lambda=2(u+\lambda)>0$,
so these curves are future directed in the original region.
They reach every regular point of its past boundary.

Finally write $U=\mathsf T+Z$ and $V=\mathsf T-Z$.
Left translation by a Heisenberg element $(p,q,r)$, with
$p,q\in\mathbb R^m$ and $r\in\mathbb R$, induces
\[
 (U,V,x,y)\longmapsto
 \bigl(e^{A_0(r+p\cdot y)}U,
       e^{-A_0(r+p\cdot y)}V,x+p,y+q\bigr).
\]
This analytic map preserves the metric on the original cone and
therefore everywhere.  Every central element $(0,0,r)$ fixes
$U=V=0$ pointwise.  As proved in
Proposition~\ref{prop:heisenberg-vacuum-dichotomy}, every lattice
meets the center in a nonzero discrete subgroup.  Each point of
$\mathcal B$ then has an infinite stabilizer.  The resulting action
is neither free nor properly discontinuous at $\mathcal B$,
which proves the asserted limitation on this quotient construction.
\end{proof}

\subsection{Uniform asymptotics near the four-dimensional horizon}
\label{app:uniform-vacuum-proof}

\begin{proof}[Proof of Proposition~\ref{prop:uniform-vacuum-horizon}]
Write $n=n(h)$, let $\alpha$ be given by
\eqref{eq:all-vacuum-data-formula}, and put
\[
 D=\sqrt{4\alpha^2+n^2},\qquad
 u=\sqrt{1+r^2},\qquad a=u-\tfrac12.
\]
Use the exact integration in
Proposition~\ref{prop:all-expanding-vacuum-data}, and set $w=k\tau$,
with $w_0$ its value at the selected regular slice.  The central
expansion eigenvalue gives
\[
 \alpha=-\frac{k}{2N}\tanh w_0,\qquad
 \frac{k}{N}=D,\qquad
 w_0=-\operatorname{arsinh}(2\alpha/n).
\]
The horizontal eigenvalue difference has the sign of $\beta-\gamma$.
Consequently the negative eigenvector of $\mathcal A$ corresponds
to the smaller of $\beta,\gamma$, and their two normalized values
are $(u-r)/2$ and $(u+r)/2$.
The exact scale factors therefore give, with $x=e^{2w}$ and
$x_0=e^{2w_0}$,
\begin{equation}
 \ell_\pm(w)=
 e^{a p_\pm(w-w_0)}
 \left(\frac{1+x}{1+x_0}\right)^{1/2}.
 \label{eq:uniform-horizon-exact-scales}
\end{equation}
These are lengths of vector fields that are unit at the regular
slice, so all coframe normalization constants have cancelled.

Define
\[
 F_a(x)=a\int_0^1 y^{a-1}\sqrt{1+xy^2}\,\dd y,
 \qquad x\ge0.
\]
Integration of the exact lapse from $\tau=-\infty$ gives
\begin{equation}
 t(w)=\frac{N_0}{\sqrt2ka}e^{aw}F_a(e^{2w}),\qquad
 \frac{t(w)}{t_{\mathrm{ref}}}
   =e^{a(w-w_0)}\frac{F_a(x)}{F_a(x_0)}.
 \label{eq:uniform-horizon-exact-time}
\end{equation}
In particular the regular proper-time age is determined by the data:
\[
 t_{\mathrm{ref}}=
 \frac{F_a(x_0)}{aD\sqrt{1+x_0}}.
\]
It is positive and varies continuously on the parameter sets in
the statement.  Combining
\eqref{eq:uniform-horizon-exact-scales} and
\eqref{eq:uniform-horizon-exact-time} yields the exact identity
\begin{equation}
 \ell_\pm(t)=B_\pm
 \left(\frac{t}{t_{\mathrm{ref}}}\right)^{p_\pm}
 \frac{\sqrt{1+x}}{F_a(x)^{p_\pm}},\qquad
 B_\pm=\frac{F_a(x_0)^{p_\pm}}{\sqrt{1+x_0}}.
 \label{eq:uniform-horizon-exact-amplitude}
\end{equation}

We now make all estimates uniform.  Shrinking $\varepsilon$ if
necessary, compactness of $\mathcal K$ gives
\[
 c|\mathcal A|_h\le r\le C|\mathcal A|_h,\qquad
 |p_\pm|\le Cr,\qquad \tfrac12\le a\le a_1
\]
for fixed $c,C,a_1>0$; moreover $w_0$ ranges over a compact interval.
The defining integral implies
\[
 1\le F_a(x)\le\sqrt{1+x}.
\]
For $z=t/t_{\mathrm{ref}}\in(0,1]$,
\eqref{eq:uniform-horizon-exact-time} consequently gives
\[
 x\le x_0F_a(x_0)^{2/a}z^{2/a}\le C z^\eta,
 \qquad \eta=2/a_1>0.
\]
All constants here are uniform, including when $\mathcal A=0$.
Furthermore,
\[
 \left|\tfrac12\log(1+x)-p_\pm\log F_a(x)\right|
 \le Cx\le Cz^\eta.
\]

At the projected datum $(h,s,0)$ one has
\[
 \alpha_0=\frac{n^2-s^2}{4s},\qquad
 w_{00}=\log(s/n),\qquad
 x_{00}=s^2/n^2,
\]
and $p_-=p_+=0$.  Thus its limiting horizontal length is
$(1+x_{00})^{-1/2}=B_{\mathrm H}$.
Since $\alpha-\alpha_0=|\mathcal A|_h^2/(2s)$,
smooth dependence on the compact parameter set gives
\[
 |x_0-x_{00}|\le Cr^2,\qquad
 \left|\log\frac{B_\pm}{B_{\mathrm H}}\right|
 =\left|p_\pm\log F_a(x_0)
       +\tfrac12\log\frac{1+x_{00}}{1+x_0}\right|
 \le Cr.
\]
Taking logarithms in~\eqref{eq:uniform-horizon-exact-amplitude}
proves~\eqref{eq:uniform-horizon-stretching}.

Finally $p_-=-r+O(r^2)$ and $p_+=r+O(r^2)$, with uniform
constants.  Substitute $z=e^{-\lambda/r}$ into
\eqref{eq:uniform-horizon-stretching}.  Uniformly for
$\lambda\in[\lambda_0,\lambda_1]$,
\[
 \log\frac{\ell_-(t_{\mathrm{ref}}e^{-\lambda/r})}
                {B_{\mathrm H}}
       =\lambda+O\bigl(r+e^{-\eta\lambda_0/r}\bigr),\qquad
 \log\frac{\ell_+(t_{\mathrm{ref}}e^{-\lambda/r})}
                {B_{\mathrm H}}
       =-\lambda+O\bigl(r+e^{-\eta\lambda_0/r}\bigr).
\]
The exponents remain in a fixed bounded interval.  Exponentiating
therefore proves~\eqref{eq:uniform-horizon-crossover}.
\end{proof}

\section{Coordinates at the four-dimensional Taub horizon}
\label{app:taub-coordinates}

The symmetric branch $\beta=\gamma=k/2$ of
\eqref{eq:scalar-family} is the $m=1$ specialization of the preceding
horizon construction.  Retain $A,B,C,N_0,k>0$, with $k=N_0A/(BC)$,
and put
\begin{equation}
 \ell_0=\frac{\sqrt2N_0}{k},\qquad
 \chi=\frac{Ak}{N_0},\qquad
 \varrho=\ell_0e^{k\tau/2},\qquad \zeta=\chi z.
 \label{eq:taub-variables}
\end{equation}
Writing $q=\varrho^4/\ell_0^4$,
$\mathcal C=\chi x\,\dd y$ and
$h_0=(B^2\dd x^2+C^2\dd y^2)/2$ gives
\begin{equation}
 g=-(1+q)\dd\varrho^2
   +\frac{\varrho^2}{1+q}(\dd\zeta-\mathcal C)^2
   +(1+q)h_0.
 \label{eq:taub-rho}
\end{equation}
Thus \eqref{eq:heisenberg-bifurcate-polar} applies with
$A_0=\chi$, $A_1=B/\sqrt2$, $A_2=C/\sqrt2$ and
$b_1=\ell_0^{-4}$, since $\chi\ell_0^2=2BC$.
This symmetric Ricci-flat family is classical
\cite{Taub1951,ChristodoulakisEtAl2001}.
On the universal cover, $\mathsf T=\varrho\cosh\zeta$ and
$Z=\varrho\sinh\zeta$ identify the original region with
$\{\mathsf T>|Z|\}$.  Proposition~\ref{prop:heisenberg-bifurcate}
gives its analytic bifurcate extension, with bifurcation metric $h_0$
and surface gravities $\pm1$ for
$K=Z\partial_{\mathsf T}+\mathsf T\partial_Z$.

For the quotient, write $U=\mathsf T+Z$, $V=\mathsf T-Z$.
A lattice element $(p,q_0,r)$ acts by
\[
 (U,V,x,y)\longmapsto
 \bigl(e^{\chi(r+py)}U,e^{-\chi(r+py)}V,x+p,y+q_0\bigr).
\]
On $U>0$ use $w=\chi^{-1}\log(U/\ell_0)$; on $V>0$ use
$w=-\chi^{-1}\log(V/\ell_0)$.  In both charts $s=UV$, and the
action is $(s,x,y,w)\mapsto(s,x+p,y+q_0,w+r+py)$.
With $\eta=\chi(\dd w-x\dd y)$ and $q=s^2/\ell_0^4$, the metrics are
\begin{equation}
 g_\pm=-\frac{s(2+q)}{4\ell_0^4(1+q)}\dd s^2
       \mp\frac{\eta\,\dd s}{1+q}
       +\frac{s}{1+q}\eta^2+(1+q)h_0.
 \label{eq:compact-taub-metric}
\end{equation}
Products of one-forms are symmetrized.  These agree with the one-sided
metrics of Proposition~\ref{prop:heisenberg-vacuum-dichotomy} after
analytic changes of the central coordinate.  For
$K=\chi^{-1}\partial_w$ one has $g_\pm(K,K)=s/(1+s^2/\ell_0^4)$.
The central lattice gives closed null generators at $s=0$ and closed
timelike orbits at $s<0$.  Its nonzero elements fix $U=V=0$, so the
bifurcate cover cannot be quotiented by this action at the bifurcation
surface.  For $\beta\ne\gamma$ the negative horizontal exponent instead
gives global $C^0$-inextendibility by
Proposition~\ref{prop:einstein-scalar}.

The compact horizons are circle bundles over a torus, in agreement with
the homogeneous horizon theory \cite{ChruscielRendall1995} and the
closed-generator classification \cite{BustamanteReiris2021}.
The normalization and Killing-field results of
\cite{ReirisBustamante2021,PetersenRacz2023} give the corresponding
smooth rigidity context.

\paragraph{Acknowledgments.}
The author is deeply grateful to his family for their steady support and
encouragement.

\paragraph{Funding.}
This work was funded by the Indonesian Endowment Fund for Education (LPDP),
on behalf of the Indonesian Ministry of Higher Education, Science and
Technology, and managed under the EQUITY Program
(Contract No.~4298/B3/DT.03.08/2025).

\paragraph{Data availability.}
No data were generated or analyzed in this study.

\paragraph{Competing interests.}
The author declares no competing interests.

\end{document}